\documentclass{amsart}
\usepackage{amssymb,amsmath, amsthm,latexsym}
\usepackage{graphics}
\usepackage{amscd}
\usepackage{graphics}
\newcommand{\cal}[1]{\mathcal{#1}}
\theoremstyle{plain}
\newtheorem{theorem}{Theorem}
\newtheorem{lemma}{Lemma}[section]
\newtheorem{theo}[lemma]{Theorem}
\newtheorem{proposition}[lemma]{Proposition}
\newtheorem{corollary}[lemma]{Corollary}
\theoremstyle{definition}
\newtheorem{definition}[lemma]{Definition}
 
\newtheorem{remark}[lemma]{Remark}
\parskip=\smallskipamount

\let\egthree=\phi
\let\phi=\varphi
\let\varphi=\egthree 

\begin{document}
\title[Free factor and free splitting graph 
boundary]{The boundary of the free factor graph 
and the free splitting graph}
\author{Ursula Hamenst\"adt}
\thanks
{
Partially supported by ERC grant 10160104}
\date{August 25, 2014}

\begin{abstract} 
We show that the Gromov boundary of the 
free factor
graph for the free group $F_n$ with $n\geq 3$ generators
is the space of equivalence classes of 
minimal very small 
indecomposable 
projective $F_n$-trees without point stabilizer containing a free factor
equipped with a quotient topology. 
Here two such trees are equivalent if the union of their metric
completions with their Gromov boundaries are 
$F_n$-equivariantly homeomorphic
with respect to the observer's topology.
The boundary of the 
cyclic splitting graph is the space
of equivalence classes of trees which either are
indecomposable or split as very large graph of actions.
The boundary of the free splitting graph
is the space of equivalence classes of trees which either 
are indecomposable or which split as large graphs of actions.
\end{abstract}

\maketitle

\setcounter{tocdepth}{10}
\tableofcontents

\parskip=\bigskipamount

\section{Introduction}

A  \emph{free splitting} of the free group $F_n$ 
with $n\geq 3$ generators is a one-edge graph of groups
decomposition
of $F_n$ with trivial edge group. Thus it either is a 
decomposition 
as a free product $F_n=A*B$ 
where $A,B$ is a proper 
subgroups of $F_n$, or as an HNN-extension $F_n=A*$.
The subgroups $A,B$ are \emph{free factors}
of $F_n$.

A vertex of the \emph{free splitting graph} 
is a conjugacy class of a free splitting 
of $F_n$. Two such free splittings
are connected by an edge of length one if up to conjugation, they
have a common refinement. The free splitting graph is
hyperbolic in the sense of Gromov \cite{HM13,HH12}.
The \emph{outer automorphism group} ${\rm Out}(F_n)$ of 
$F_n$ acts on the free splitting graph as a group of
simplicial isometries. 

A \emph{cyclic splitting} of $F_n$ is a one-edge graph of 
groups decomposition of $F_n$ with infinite cyclic edge group.
Thus it either is a decomposition as
an amalgamated product $F_n=A*_{\langle w\rangle}B$ 
or as an HNN-extension $F_n=A*_{\langle w\rangle}$ 
where $\langle w\rangle$ is an infinite cyclic subgroup of $F_n$.
A vertex of the \emph{cyclic splitting graph} is a conjugacy class
of a free or of a cyclic splitting of $F_n$. Two such splittings
are connected by an edge of length one if up to conjugaction,
they have a common refinement. The cyclic splitting graph
is hyperbolic in the sense of Gromov \cite{Mn12}. The outer 
automorphism group ${\rm Out}(F_n)$ of $F_n$ acts
on the cyclic splitting graph as a group of simplicial isometries.
The free splitting graph is a coarsely dense 
${\rm Out}(F_n)$-invariant subgraph of the cyclic splitting
graph.

The \emph{free factor graph} is the graph ${\cal F\cal F}$
whose vertices are conjugacy classes of free factors of $F_n$.
Two free factors $A,B$ are connected
by an edge of length one if up to conjugation,
either $A<B$ or $B<A$. 
The free factor
graph is hyperbolic in the sense of Gromov \cite{BF11}.
The outer automorphism group ${\rm Out}(F_n)$ of $F_n$
acts on the free factor graph as a group of simplicial
isometries.

The goal of this article is to determine the Gromov boundaries
of these three graphs.

Let $cv(F_n)$ be the unprojectivized \emph{Outer space} of
minimal free simplicial $F_n$-trees,
with its boundary $\partial cv(F_n)$ of minimal
\emph{very small} $F_n$-trees \cite{CL95,BF92} which 
either are not simplicial or which are not free.
Write $\overline{ cv(F_n)}=cv(F_n)\cup \partial cv(F_n)$.

A tree
$T\in \partial cv(F_n)$ is called \emph{indecomposable}
if for any finite, non-degenerate arcs $I,J\subset T$, 
there are elements $g_1,\dots,g_r\in F_n$ so that 
\[J\subset g_1I\cup \dots \cup g_rI\] and such that
$g_iI\cap g_{i+1}I$ is non-degenerate for $i\leq r-1$
(Definition 1.17 of \cite{G08}). 
%An \emph{alignment preserving map}
%between two $F_n$-trees is an equivariant map
%$f:T\to T^\prime$ with the property that the preimage
%of every point $x\in T^\prime$ is convex.
%The tree $T^\prime$ is called a \emph{projection} of $T$.
%A projection $T^\prime$ is \emph{maximal} if $T^\prime$ does not
%admit nontrivial projections.
%
%Now let $T\in \partial cv(F_n)$ be a tree with
%dense orbits. Then the metric completion 
%$\overline{T}$ of $T$
%is defined. The union $\hat T=\overline{T}\cup \partial T$ of 
%$\overline T$ with the Gromov boundary 
%$\partial T$ has a natural \emph{observers' topology}
%so that $\hat T$ equipped with this topology is
%compact \cite{CHL07}. The natural action of $F_n$ on 
%$\hat T$ is continuous, and the inclusion 
%$T\to \hat T$ is continuous as well where 
%$T$ is equipped with the metric topology.
The $F_n$-orbits on an indecomposable
tree are dense (Lemma 1.18 of \cite{G08}).
Being indecomposable is invariant under scaling and
hence it is defined for projective trees.

A \emph{graph of actions} is a minimal $F_n$-tree 
${\cal G}$ which consists of 
\begin{enumerate}
\item a simplicial tree $S$, called the \emph{skeleton},
equipped with an action of $F_n$
\item for each vertex $v$ of $S$ an $\mathbb{R}$-tree
$Y_v$, called a \emph{vertex tree}, and
\item for each oriented edge $e$ of $S$ with terminal vertex
$v$ a point $p_e\in Y_v$, called an \emph{attaching point}.
\end{enumerate}

It is required that the projection $Y_v\to p_e$ is equivariant, 
that for $g\in F_n$ one has $gp_e=p_{ge}$ and that
each vertex tree is a minimal very small tree for the action of its
stabilizer in $F_n$.

Associated to a graph of actions ${\cal G}$ is a canonical action
of $F_n$ on an $\mathbb{R}$-tree $T_{\cal G}$ which is
called the \emph{dual} of the graph of actions
 \cite{L94}. Define a pseudo-metric $d$ on 
$\coprod_{v\in V(S)}Y_v$ as follows. If $x\in Y_{v_0},y\in Y_{v_k}$
let $e_1\dots e_k$ be the reduced
edge-path from $v_0$ to $v_k$ in $S$ and define
\[d(x,y)=d_{Y_{v_1}}(x,p_{\overline e_1})+\dots
+d_{Y_{v_k}}(p_{e_k},y).\]
Making this pseudo-metric Hausdorff gives an 
$\mathbb{R}$-tree $T_{\cal G}$. The vertex trees
of the graph of actions ${\cal G}$ are isometrically embedded in 
$T_{\cal G}$ and will be called the vertex trees of
$T_{\cal G}$. 

An minimal very small $F_n$-tree 
$T$ \emph{splits as a graph of actions}
if there
is a graph of actions ${\cal G}$ with dual tree $T_{\cal G}$, and 
there is an equivariant isometry $T\to T_{\cal G}$. We call
${\cal G}$ a structure of a graph of action for $T$.
We also say that the projectivization $[T]$ of an $F_n$-tree $T$ 
splits as a graph of actions  if $T$ splits as a graph of actions.

%The transverse family is a \emph{transverse covering}
%if any finite segment $I\subset T$ is contained in a finite
%union $Y_{v_1}\cup \dots \cup Y_{v_r}$ 
%of components from the family. 
%By Lemma 4.7 of \cite{G04}, $T$ admits a transverse covering
%if and only if $T$ splits as a graph of actions. Moreover, every
%structure of a graph of action for $T$ defines a transverse covering
%which is the covering defined by the vertex trees.

A tree $T\in \partial {cv(F_n)}$ with dense orbits 
\emph{splits as a large graph of actions}
if $T$ splits as a graph of actions with 
the following properties.
\begin{enumerate}
\item There is a single $F_n$-orbit of non-degenerate vertex trees.
\item The stabilizer of a non-degenerate
vertex tree is not contained in any proper
free factor of $F_n$.
\item A non-degenerate vertex tree is indecomposable for its stabilizer.
\end{enumerate}

Note that by the second requirement, the stabilizer of an edge of the
skeleton of a large graph of actions is non-trivial. 

We say that a tree $T\in \partial cv(F_n)$ with dense orbits
\emph{splits as a very large graph of 
actions} if $T$ splits as a large graph of actions and if
the stabilizer of an edge of the skeleton of the 
is not infinite cyclic.

If $T\in \partial{cv(F_n)}$ is a tree with a dense action of
$F_n$ then the union 
$\hat T$ of the metric completion 
$\overline{T}$ of $T$ with the Gromov boundary $\partial T$ can
be equipped with an \emph{observer's topology} \cite{CHL07}.
This topology only depends on the projective class 
of $T$.

Denote by $\partial {\rm CV}(F_n)$ the boundary of projectivized
Outer space, i.e. the space of projective
classes of trees in $\partial cv(F_n)$.
Let ${\cal S\cal T}\subset \partial {\rm CV}(F_n)$ 
be the ${\rm Out}(F_n)$-invariant
subspace of projective trees which either are indecomposable
or split as large graph of actions. 
Let  $\sim$ be the equivalence relation on
${\cal S\cal T}$ which is defined as follows. The equivalence class
of the projective class $[T]$ of a
tree $T$ consists of all projective classes of 
trees $S$ so that  
the trees $\hat S$ and $\hat T$ equipped with
the observer's topology are $F_n$-equivariantly homeomorphic.
We equip $ {\cal S\cal T}/\sim$ 
with the quotient topology.

\begin{theorem}\label{freesplitbd}
The Gromov boundary of the free splitting graph is the 
space ${\cal S\cal T}/\sim$ of equivalence classes of 
projective very small trees 
which either are indecomposable or which split as large
graph of actions.  
\end{theorem}

Let ${\cal C\cal T}\subset {\cal S\cal T}$ be the 
${\rm Out}(F_n)$-invariant subset of equivalence
classes of projective trees which either are indecomposable 
or split as very large graph of actions. This subspace is 
saturated for the equivalence relation $\sim$.

\begin{theorem}\label{arational}
The Gromov boundary of the cyclic splitting graph is 
the space ${\cal C\cal T}/\sim$ of equivalence classes of
projective trees which either are indecomposable
or which split
as very large graph of actions.
\end{theorem}

An indecomposable projective tree $[T]\in {\cal S\cal T}$ is 
called \emph{arational} if no nontrivial point stabilizer of $[T]$
contains a free factor. The space of arational trees is 
saturated for the equivalence relation $\sim$.

 \begin{theorem}\label{positive}
The Gromov boundary 
of the free factor graph is the space ${\cal F\cal T}/\sim$ of
equivalence classes of 
arational trees equipped with the quotient topology.
\end{theorem}

The strategy of proof for the above theorems
%Theorem \ref{arational} 
%and 
%Theorem \ref{positive} 
builds on the strategy of
Klarreich \cite{K99} who determined the Gromov
boundary of the curve graph of a non-exceptional
surface of finite type. 
%To determine the boundary of the
%free splitting graph we use in addition a result from
%\cite{H11}.

As in \cite{K99}, we begin with describing in Section 2
two electrifications of Outer space which 
are quasi-isometric to the free factor graph and to the 
free splitting graph, respectively.

In Section 3 we introduce folding paths and collect 
those of their properties which are used later on.
%We also formulate some conditions on graphs 
%with an isometric action of ${\rm Out}(F_n)$,
%and we show that they 
%are satisfied for the free splitting graph and
%the free factor graph. These properties
%are used throughout the
%rest of the paper. We call a graph with these properties
%a good hyperbolic ${\rm Out}(F_n)$-graph. 

Section 4 establishes a technical result on folding
paths and their relation to the geometry of 
the free splitting graph. Section 5
gives some information on trees in 
$\partial cv(F_n)$ for which the orbits of the action of 
$F_n$ are not dense.
 
In Section 6 we show that the boundary 
of each of the above three ${\rm Out}(F_n)$-graphs is the image of 
a ${\rm Out}(F_n)$-invariant
subspace of the boundary $\partial {\rm CV}(F_n)$ 
of projectivized Outer space under a continuous map. 
The main technical tool to this end is a detailed
analysis on properties of folding paths.

In Section 7 we investigate the structure of indecomposable
trees, and 
Section 8 contains some information
on trees with point stabilizers containing a free factor. In
Section 9 we show that only indecomposable projective
trees can give rise to points in the boundary of the free factor graph.

The proof of Theorem \ref{positive} is contained in Section 10, 
and the proofs of Theorem \ref{freesplitbd} and 
Theorem \ref{arational} are completed in Section 11.
%and 
%Theorem \ref{arational} is proved in Section 9.
%The appendix summarizes some results from \cite{HM13,BF12} 
%in the form needed in Section 9.

Theorem \ref{positive} was independently and at the same time 
obtained by Mladen Bestvina and Patrick
Reynolds \cite{BR12}. Very recently 
Horbez \cite{H14} extended Theorem \ref{arational} to a more general
class of groups, with a somewhat different proof.

\bigskip

{\bf Acknowledgement:} I am indebted to 
Vincent Guirardel and Gilbert Levitt for inspiring and
insightful discussions.
In particular, I owe an argument 
used in the proof of
Lemma \ref{graphofgroups}  
to Vincent Guirardel. Thanks to Karen Vogtmann
for help with references.
I am grateful to Mladen Bestvina for a valuable
comment and to Lee Mosher 
for some 
helpful explanation related to his yet unpublished work \cite{HM13d}.
I am particularly indebted to Martin Lustig and Camille Horbez 
for pointing out an error in an earlier version of this paper.

\section{Geometric models}\label{geometric}

As in the introduction, we consider a free group $F_n$ of 
rank $n\geq 3$. Our goal is to 
introduce geometric models for the free factor graph and
the free splitting graph. We also construct $n-3$ additional
${\rm Out}(F_n)$-graphs which geometrically
lie between the free factor graph and the free splitting graph; 
they are analogs of the graphs considered in \cite{H11}.
Although these graphs are not used for the proofs of 
the theorems from the introduction, they shed some light 
on the geometry of the free splitting graph and the structure
of its boundary. They will be useful in another context.

The \emph{free splitting graph} 
\cite{KL09} is defined 
to be the graph whose vertices are  
one-edge graph of groups decompositions of $F_n$ with
trivial edge group. Two such vertices are connected
by an edge if up to conjugation,
they have a common refinement.
It is more convenient for our purpose 
to use instead the first barycentric subdivision
${\cal F\cal S}$ of the free splitting graph. 
Its vertices are graph of groups decompositions
of $F_n$ with trivial edge groups. Two vertices 
$\Gamma,\Gamma^\prime$ are
connected by an edge if $\Gamma$ is a collapse or blow-up of
$\Gamma^\prime$. The outer automorphism group 
${\rm Out}(F_n)$ of $F_n$ acts on ${\cal F\cal S}$ as a group
of simplicial isometries.

The \emph{cyclic splitting graph} \cite{Mn12} is the 
graph whose vertices are one-edge graph of groups 
decompositions of $F_n$ with trivial or infinite cyclic 
edge group. Two such vertices are connected by an
edge if up to conjugation, they have a common refinement.
The vertices of the first barycentric subdivision ${\cal C\cal S}$ 
of the cyclic splitting graph 
are graph of groups decompositions of $F_n$ with 
trivial or infinite cyclic edge groups. Two vertices
$\Gamma,\Gamma^\prime$ are connected by an
edge if $\Gamma$ is a collapse or a blow-up of
$\Gamma^\prime$. The graph ${\cal F\cal S}$ is 
an ${\rm Out}(F_n)$-invariant subgraph of 
${\cal C\cal S}$. The inclusion is an ${\rm Out}(F_n)$-equivariant
one-Lipschitz embedding
\begin{equation}\label{psi}
\Psi:{\cal F\cal S}\to {\cal C\cal S}.
\end{equation}

Let $cv(F_n)$ be the unprojectivized Outer space of all 
simplicial trees with minimal free isometric actions of $F_n$,
equipped with the equivariant Gromov-Hausdorff topology.
We refer to \cite{P88} for detailed information on this topology.
The quotient $T/F_n$ of each tree $T\in cv_0(F_n)$ 
defines a graph of groups decomposition of $F_n$ 
with trivial edge groups and hence
a vertex $\Upsilon(T)$ in ${\cal F\cal S}$.
This graph of groups decomposition is invariant
under scaling the metric of $T$.

An $F_n$-tree $T$ is called \emph{very small} 
if $T$ is minimal and if moreover the
following holds.
\begin{enumerate}
\item Stabilizers of non-degenerate segments are at most cyclic.
\item If $g^n$ stabilizes a non-degenerate
segment $e$ for some $n\geq 1$ then
so does $g$.
\item ${\rm Fix}(g)$ contains no tripod for $g\not=1$.
\end{enumerate}
Here a tripod is a compact subset of $T$ which is homeomorphic
to a cone over three points.

The equivariant Gromov Hausdorff topology extends to the 
space $\partial cv(F_n)$ 
of minimal very small $F_n$-trees \cite{BF92,CL95} which 
either are not simplicial or which are not free.
The subspace $cv(F_n)$ is dense in 
$\overline{cv(F_n)}=cv(F_n)\cup \partial cv(F_n)$. 

Let 
\begin{equation}\label{simplicial}
\overline{cv(F_n)}^s\subset \overline{cv(F_n)}
\end{equation}
be the ${\rm Out}(F_n)$-invariant subset of
all minimal very small
simplicial $F_n$-trees.
% with volume one quotient.
Then $\overline{cv(F_n)}^s-cv(F_n)$ consists of simplicial
$F_n$-trees so that the action of $F_n$ is not free.

For each $T\in \overline{cv(F_n)}^s$ the quotient graph
$T/F_n$ defines a graph of groups decomposition 
of $F_n$ with at most
cyclic edge groups (we refer to \cite{CL95} for a detailed
discussion). 
Thus there is a natural 
${\rm Out}(F_n)$-equivariant map
\begin{equation}\label{upsilonc}
\Upsilon_{\cal C}:\overline{cv(F_n)}^s\to {\cal C\cal S}.
\end{equation}
Its image is the vertex set of ${\cal C\cal S}$.

Let 
\begin{equation}\label{bd+}
\overline{cv(F_n)}^+\subset \overline{cv_0(F_n)}^s
\end{equation}
be the ${\rm Out}(F_n)$-invariant subspace of 
simplicial trees with 
%volume one quotient and 
at least one $F_n$-orbit of edges with
trivial edge stabilizer. To each 
tree $T\in \overline{cv(F_n)}^+$ 
we can associate the graph of groups decomposition 
$\Upsilon(T)$ with trivial edge groups 
obtained by collapsing all edges in $T/F_n$ with
non-trivial edge groups to a point. 
Thus there is an ${\rm Out}(F_n)$-equivariant 
coarsely surjective map
\begin{equation}\label{upsilon}
\Upsilon:\overline{cv(F_n)}^+\to 
{\cal F\cal S}.\end{equation}
Note that for $T\in \overline{cv(F_n)}^+$
the distance in ${\cal C\cal S}$ between 
$\Upsilon(T)\in {\cal F\cal S}\subset {\cal C\cal S}$ and 
$\Upsilon_{\cal C}(T)$ is at most one.

The \emph{free factor graph} ${\cal F\cal F}$ 
is the graph whose vertices
are free factors of $F_n$ and where two such vertices 
$A,B$ are connected by an edge 
of length one if and only if up to conjugation,
either
$A<B$ or $B<A$. 
We next observe that there is a (coarsely) 
${\rm Out}(F_n)$-equivariant (coarsely) consistent
way to associate to a point in $\overline{cv(F_n)}^s$ a 
vertex in the free factor graph. 

To this end 
say that a map $f:X\to Y$ between
metric spaces $X,Y$ equipped with an isometric action of 
a group $\Gamma$ is \emph{coarsely $\Gamma$-equivariant}
if there is a number $C>0$ such that
\[d(g(f(x)),f(gx))\leq C\]
for all $x\in X$, all $g\in \Gamma$.
The map $f$ is called \emph{coarsely $L$-Lipschitz}
if 
\[d(f(x),f(y))\leq Ld(x,y)+L\] for all $x,y\in X$.

\begin{lemma}\label{map}
There is a number $k>1$, and there is a
coarsely $k$-Lipschitz coarsely ${\rm Out}(F_n)$-equivariant
map $\Omega:{\cal C\cal S}\to {\cal F\cal F}$.
\end{lemma}
\begin{proof} 
There is a second description of the
cyclic splitting graph as follows. Namely, 
let ${FZ}_n$ be the graph whose vertex set is the set
of one edge free splittings of $F_n$.
Two such splittings $X,Y$ are connected
by an edge if either
\begin{enumerate}
\item[(a)] they are connected in the free splitting graph by
an edge or
\item[(b)] there exists a $\mathbb{Z}$-splitting
$T$ and equivariant \emph{edge folds} $X\to T,Y\to T$.
\end{enumerate}
Roughly speaking, an edge fold of a free splitting
$A*B$ is a splitting of the form
$A*_C\langle B,C\rangle$ where $C$ is any maximal cyclic subgroup of $A$.
We refer to \cite{St65,BF92,Mn12} for details of this
construction.

By Proposition 2 of \cite{Mn12}, the vertex inclusion 
of the set of vertices of 
$ZF_n$ into the set of vertices of the cyclic splitting graph extends to
an ${\rm Out}(F_n)$-equivariant quasi-isometry.

Define a map $R$ from the vertex set of 
$FZ_n$ into the vertex set of the free factor graph
by associating to a vertex $X$ of $FZ_n$ a vertex group of the
corresponding free splitting. 

Since for every free splitting of $F_n$ of the form
$F_n=A*B$ the distance in ${\cal F\cal F}$ between the
free factors $A,B$ is at most three, 
the map $R$ coarsely does not depend on choices and hence
it is coarsely ${\rm Out}(F_n)$-equivariant. 
We claim that it
extends to a coarsely Lipschitz map $ZF_n\to {\cal F\cal F}$.

To this end let $d_{\cal F}$ be the distance in 
${\cal F\cal F}$. Since the metrics on the graphs 
${\cal C\cal S},\,{\cal F\cal F}$ are geodesic,
by an iterated application of the triangle inequality 
it suffices to show that there is a number
$L>1$ so that $d_{\cal F}(R(X),R(Y))\leq L$ whenever
$X,Y\in ZF_n$ are vertices connected by an edge.

Consider first the case that $X,Y$ are connected by
an edge in the free splitting graph. 
Assume furthermore that $X,Y$ are one edge two vertex splittings,
i.e. that these splittings are of the form
$X=A*B$ and $Y=C*D$. Then up to conjugation, these splitting have
a common refinement, i.e. up to exchanging $A,B$ and 
$C,D$ there is a free factor $E$ of $F_n$ which is a subgroup of 
both $A,C$. 

Then $d_{\cal F}(E,A)\leq 1,d_{\cal F}(E,C)\leq 1$, moreover 
the distance in 
${\cal F\cal F}$ between $A$ and $B$ and between
$C$ and $D$ is at most three. Thus the distance between
$R(X)$ and $R(Y)$ is at most $8$. 
The case that one or both of the splittings $X,Y$ is a one-loop
splitting follows in the same way and will be omitted
(see \cite{KR12} for details).

Now assume that $X,Y$ are connected by an edge of type (b) above.
Following Lemma 1 of \cite{Mn12} and the discussion
in the proof of Theorem 5 of \cite{Mn12},
let $\langle w\rangle$ be the edge group of the $\mathbb{Z}$-splitting
to which $X,Y$ fold and let $A$ be the smallest free factor of $F_n$
which contains $\langle w\rangle$.
Then $A$ is a subgroup of a vertex group of both $X$ and 
$Y$ and hence
the distance in the free factor graph between
$R(X),R(Y)$ is at most $8$.
\end{proof}

By Lemma \ref{map}, the map
\begin{equation}\label{ff}
\Upsilon_{\cal F}=\Omega\circ \Upsilon_{\cal C}:
\overline{cv(F_n)}^{+}\to {\cal F\cal F}
\end{equation}
is coarsely ${\rm Out}(F_n)$-equivariant and coarsely
surjective.

For a number $\ell>0$ call a free basis 
$e_1,\dots,e_n$ of $F_n$ \emph{$\ell$-short} for a tree
$T\in \overline{cv(F_n)}^{s}$ if there is a vertex $v$ of 
$T$ so that for each $i$ the distance in $T$ between
$v$ and $e_i v$ is at most $\ell$. 
If $T\in cv(F_n)$ then 
this is equivalent
to stating that there is a marked rose $R$ with $n$ petals of equal
length one representing each one of the basis elements $e_i$,
and there is an $\ell$-Lipschitz map 
$u:R\to T/F_n$ which maps the vertex of $R$ to a vertex of $T/F_n$,
and which maps each petal marked with $e_i$ to a path in $T/F_n$
representing $e_i$.

Let 
\[\overline{cv_0(F_n)}^s\subset \overline{cv(F_n)}^s\]
be the subspace of all simplicial trees $T$ with  
quotient graph $T/F_n$ of volume one. The group 
${\rm Out}(F_n)$ acts on $cv_0(F_n)^s$ by precomposition of marking.
Let $\overline{cv_0(F_n)}^+=\overline{cv(F_n)}^+\cap
cv_0(F_n)^s$ and define 
\begin{equation}
\overline{cv_0(F_n)}^{++}\subset \overline{cv_0(F_n)}^+\end{equation}
to be the set of all 
simplicial very small $F_n$-trees 
with volume one quotient and no non-trivial
edge stabilizer.

\begin{lemma}\label{shortexists}
Every tree $T\in \overline{cv_0(F_n)}^{++}$ admits a 
$3$-short basis.
\end{lemma}
\begin{proof} If $T\in cv_0(F_n)$ then choose any vertex
$v\in G=T/F_n$. Collapse a maximal tree in $G$ to a point.
The resulting graph is a rose $R$ with $n$ petals which determines
a free basis $e_1,\dots,e_n$ of $F_n$. In the graph $G$, each 
basis element $e_i$ is represented by a loop based at $v$ which 
passes through every edge of $G$ at most twice. Since the volume of 
$G$ equals one, the basis $e_1,\dots,e_n$ is 
$2$-short for $T$.

If $T$ has non-trivial vertex stabilizers then each of these
stabilizers is a free factor of $F_n$. This means that there is a
free simplicial $F_n$-tree 
$T^\prime\in cv(F_n)$ with quotient of volume 
smaller than $3/2$,
and there is a one-Lipschitz equivariant map
$T^\prime\to T$ obtained by collapsing 
the minimal subtrees of $T^\prime$ which are invariant
under the vertex stabilizers of $T$ to points.
The tree $T^\prime$ is the universal covering of a graph $G$
obtained from $T/F_n$ by attaching a marked rose with
$m$ 
petals of very small length to the projection of a vertex
in $T$ with stabilizer of rank $m$. This construction is carried
out in detail in the Combination Lemma 8.6 of \cite{CL95}.

As $T^\prime$ has a $3$-short basis by
the discussion in the first paragraph of this proof, the
same holds true for $T$. 
The lemma follows.
\end{proof}

The following lemma shows that 
the statement of Lemma \ref{shortexists}  
can be extended to a class of trees with 
non-trivial edge stabilizers. 
It serves as an illustration for the results
established later on, but
it will not be used directly in the
sequel. We also refer to Lemma 8.6 of \cite{CL95} 
for related and more general constructions.

\begin{lemma}\label{shortexists3}
Let $T\in \overline{cv_0(F_n)}^s$ be a simplicial 
$F_n$-tree such that the graph of groups decomposition
$T/F_n$ has a single edge with non-trivial edge group, 
and this edge is separating.  
Then $T$ admits a $3$-short basis.
\end{lemma}
\begin{proof} 
Let $T\in \overline{cv_0(F_n)}^s$ be a tree as in the lemma.
There is a single separating edge $s$ in $T/F_n=G$ with non-trivial 
edge group.
The edge $s$ defines a one-edge two vertex cyclic splitting of $F_n$.
Let $S<F_n$ be the cyclic edge group.

By assumption, the edge $s$ is separating in $G$.
Let $v_1,v_2\in G$ be the two vertices
on which $s$ is incident and let $H(v_1),H(v_2)$ be the vertex groups.
The infinite cyclic edge group
$S$ is a free factor 
in at least one of the vertex groups $H(v_i)$, say in the group $H(v_2)$
\cite{St65}. Then $H(v_1)$ is a free factor of $F_n$
(this is not true in general for the vertex group $H(v_2)$).

Let $\tilde v_1$ be a preimage of $v_1$ in $T$. As $H(v_1)$ is a free factor
of $F_n$ containing $S$ as a subgroup, 
the subset $A$ of $F_n$ of all elements $g$ so that
the segment in $T$ connecting $\tilde v_1$ to $g\tilde v_1$ does not
cross through a preimage of $s$ is a free factor of $F_n$.
The minimal $A$-invariant subtree $T_A$ of $T$ does not have
edges with non-trivial stabilizer. 
Its quotient graph $T_A/A$ is an embedded
subgraph of $G$ and hence its 
volume is smaller than one.

Choose the vertex $\tilde v_1$ 
as a base point. By Lemma \ref{shortexists} and 
the above discussion, there is a 
$3$-short free basis $e_1,\dots,e_k$ for the free
factor $A$ of $F_n$. 
Extend this basis of $A$ 
to a basis of 
$F_n$ as follows. First 
add a free basis $e_{k+1},\dots,e_u$ for the free factor $B$ of $F_n$ 
generated by all based loops at $v_2$ in $G$ 
which do not cross through $s$. We require that
each such based loop 
passes through every edge of $G$ at most twice.
Up to conjugation of $B$, an element $e_i$ 
$(i\geq k+1)$ translates 
the vertex 
$\tilde v_1$ a distance which equals 
the sum of twice the length of $s$ with the 
length of the defining loop in $G$. In particular,
the translation length of each of the elements $e_i$ is at most three.
This partial basis 
can be extended to a $3$-short free basis
of $F_n$ by adding some elements in the point stabilizer
of $v_2$.
\end{proof}

Lemma \ref{shortexists3} 
does not seem to hold for all points in 
$\overline{cv_0(F_n)}^+$.  
However, there is a weaker statement which holds true.

For its formulation, define a \emph{pure cyclic splitting} to 
consist of a  graph of groups 
decomposition for $F_n$ with all edge groups cyclic.
Define a \emph{free refinement} of a pure cyclic splitting $s$
to be a refinement $s^\prime$ of $s$  
so that each edge group of an edge in $s^\prime-s$ 
is trivial.
We have

\begin{lemma}\label{shortexists2}
For every pure cyclic splitting $s$ of $F_n$ there is a number
$\ell(s)>0$ with the following property. Every
simplicial tree $T\in \overline{cv_0(F_n)}^s$ so that 
$\Upsilon_{\cal C}(T)$ is a free refinement of $s$ 
admits an $\ell(s)$-short basis.
\end{lemma}
\begin{proof} Let $s$ be a pure cyclic splitting of $F_n$. For 
some $\epsilon <<1/3n-4$ let 
$B\subset \overline{cv_0(F_n)}^s$ be the set of all
trees $T$ with the following properties. 
\begin{enumerate}
\item $\Upsilon_{\cal C}(T)\in {\cal C\cal S}$ is a free refinement of $s$.
\item The length of each edge of $T$ is at least $\epsilon$.
\end{enumerate}
Note that $B$ is a closed subset of $\overline{cv_0(F_n)}^s$
(this standard fact is discussed in detail in \cite{CL95}).

As there are only finitely many topological types of quotient graphs
$T/F_n$ for trees $T\in B$, 
the stabilizer in ${\rm Out}(F_n)$ of the 
pure cyclic splitting $s$ 
acts on $B$ properly and cocompactly.
On the other hand, by the definition of the equivariant Gromov
Hausdorff topology \cite{P89}, 
if $k>0$ and if ${\cal A}$ is a 
$k$-short free basis for a tree $T\in \overline{cv_0(F_n)}^s$ then
there is a neighborhood $U$ of $T$ in $\overline{cv_0(F_n)}^s$ such that
${\cal A}$ is a $k+1$-short free basis for every tree $S\in U$.
Thus by invariance under the action of ${\rm Out}(F_n)$
and cocompactness, there is a number $q(s)>0$ so that  
any tree $T\in B$ admits a $q(s)$-short basis. 
 
Now there is a number $b>0$ depending on $s$ and $\epsilon$, and 
for every tree $T\in \overline{cv_0(F_n)}^s$ 
such that $\Upsilon_{\cal C}(T)$ is a free refinement of $s$ 
there is a tree $T^\prime\in B$ and there is an equivariant 
$b$-Lipschitz map $T^\prime\to T$.  This map 
preserves the topological type of the quotient graph and
expands or decreases the lengths of the edges of $T^\prime$.
As $T^\prime$ admits $q(s)$-short basis, the tree
$T$ admits a $bq(s)$-short basis. This shows the lemma.
\end{proof}

Fix a number $k\geq 3$ and a number $\ell\leq n-1$. 
Let $R_\ell$ be the rose with $\ell$ petals of equal length one
and let $\tilde R_\ell$ be its universal covering.
Let $A$ be a free factor of $F_n$ and identify $A$ with the fundamental
group of $R_\ell$. The free factor $A$ is called  
\emph{$k$-short} for 
$T\in \overline{cv_0(F_n)}^s$ 
if there is an $A$-equivariant $k$-Lipschitz map 
\[F:\tilde R_\ell\to T\]
which maps vertices to vertices.
If $T\in cv_0(F_n)$ then this is equivalent to the requirement
that the quotient map $f:R_\ell\to T/F_n$ of $F$ is a 
$k$-Lipschitz map which 
maps the vertex $v$ of $R_\ell$ to a vertex of $T/F_n$ and 
so that $f_*(\pi_1(R_\ell,v))$
is conjugate to $A$. 
By Lemma \ref{shortexists}, 
each $T\in \overline{cv_0(F_n)}^{++}$ admits $3$-short free factors of any rank
$\ell\leq n-1$.

The following observation is a version of Lemma 3.2 of \cite{BF11}
and Lemma A.3 of \cite{BF12} 
(see also \cite{HM13}). For its formulation, 
note that a corank one free factor $A<F_n$ determines a 
one-loop graph of groups decomposition of $F_n$ with 
vertex group $A$ (see Section 4.1 of \cite{HM13a} for details).
In the sequel we often view a corank one free factor of 
$F_n$ as a point in ${\cal F\cal S}$ without further notice.

\begin{lemma}\label{short}
For every $k\geq 3$ 
there is a number $c=c(k)>0$ with the following property.
For every $T\in \overline{cv_0(F_n)}^{++}$, 
the distance in ${\cal F\cal S}$ between
$\Upsilon(T)$ and any corank one free factor 
which is $k$-short for $T$ 
is at most $c$. 
\end{lemma}
\begin{proof} Since the volume of $T/F_n$ equals one 
and since $T/F_n$ has at most $3n-4$ edges, there is 
an edge $e$ in $T/F_n$ of length at least $1/(3n-4)$.

Assume first 
%that $T\in \overline{cv_0(F_n)}^{++}$, i.e.
%that there are no edges in $T$ with trivial stabilizer.
%Assume moreover 
that $e$ is non-separating.
Collapsing  a maximal forest in $T/F_n$ 
not containing $e$ yields an $F_n$-tree
$V$ whose quotient $V/F_n$ is a rose 
and an equivariant one-Lipschitz map
\[F:T\to V.\] 
The induced quotient map
$f:T/F_n\to V/F_n$ maps the edge $e$ isometrically onto a 
petal $e_0$ of $V/F_n$. We refer again to \cite{CL95} for 
details of this construction.

Let $v$ be a vertex of $V$ and 
let $A<F_n$ be the corank one free factor of all elements 
$g$ with 
the property that the segment connecting $v$ to $gv$ does not cross
through a preimage of $e_0$.
The free factor $A$ defines a one-loop free splitting of $F_n$ 
which is obtained from $V/F_n$ by collapsing the complement of 
$e_0$ to a point. In particular, this splitting 
is a collapse of $\Upsilon(T)$. Therefore 
it suffices to show that the 
distance in ${\cal F\cal S}$ between $A$ 
and any corank one free factor of $F_n$ which is 
$k$-short for $T$ is bounded from above by a number
only depending on $k$.

Let as before $R_{n-1}$ be a rose with $n-1$ petals of equal length one
and universal covering $\tilde R_{n-1}$. 
Let $Q:\tilde R_{n-1}\to T$ be a $k$-Lipschitz map
which maps the vertices of $\tilde R_{n-1}$ to vertices in $T$
and which is equivariant with respect to the action of a corank free
factor $C$ of $F_n$, viewed as the fundamental group of $R_{n-1}$. 
Then $Q_0=F\circ Q:
\tilde R_{n-1}\to V$ is an equivariant $k$-Lipschitz map.
Its quotient
$q_0:R_{n-1}\to V/F_n$ maps the 
vertex of $R_{n-1}$ to the vertex of $V/F_n$. 
The image under $q_0$ of 
a petal of $R_{n-1}$ passes through the loop $e_0$ at most 
$k(3n-4)$ times. The claim of the lemma now follows from 
Lemma A.3 and Remark A.8 of \cite{BF12}.

If 
%$T\in \overline{cv_0(F_n)}^{++}$ and if 
the edge
$e$ is separating then the same reasoning applies.
Let $v_1,v_2$ be the two vertices in $T/F_n$ on which 
$e$ is incident. Collapsing maximal trees in the two
components of $T/F_n-e$ yields an $F_n$-tree $V$ whose
whose quotient $V/F_n$ consists of two roses connected
at their vertices by the edge $e$. Lemma A.3 and
Remark A.8 of \cite{BF12} are valid in this situation as well and 
yield the lemma.
%We sketch a proof in the appendix (Lemma \ref{estimate}).
\end{proof}

\begin{corollary}\label{control}
\begin{enumerate}
\item There is a number $m=m(n)>0$, and for every
tree $T\in \overline{cv_0(F_n)}^{++}$ there is a neighborhood $U$ of 
$T$ in $\overline{cv_0(F_n)}^{++}$ such that
\[{\rm  diam}(\Upsilon(U))\leq m.\]
%\item For every tree $T\in \overline{cv_0(F_n)}^+$ 
%there is a neighborhood 
%$U$ of $T$ in $\overline{cv_0(F_n)}^+$ such that the diameter
%of 
%$\Upsilon(U)$ is finite.
\item For every simplicial tree 
$T\in \overline{cv_0(F_n)}^s$ there is a neighborhood $U$ of 
$T$ in $\overline{cv_0(F_n)}^s$ such that
the diameter of $\Upsilon(U\cap \overline{cv_0(F_n)}^{++})$ is 
finite.
\end{enumerate}
\end{corollary}
\begin{proof} By Lemma \ref{shortexists}, every tree $T\in 
\overline{cv_0(F_n)}^{++}$ admits a $3$-short basis.
Let $e_1,\dots,e_n$ be such a basis. By the definition of 
the equivariant Gromov Hausdorff topology, there is a neighborhood 
$V$ of $T$ in $\overline{cv_0(F_n)}^s$, and for every $S\in V$ there
is a vertex $v\in S$ so that for each $i$ 
the distance in $S$ between $v$ and 
$e_iv$ is at most $4$. Thus the corank one free factor
with basis 
$e_1,\dots,e_{n-1}$ is $4$-short for each $S\in V$. 
The first part of the corollary for $U=V\cap \overline{cv_0(F_n)}^{++}$
is now immediate from Lemma \ref{short}.

To show the second part of the corollary, let 
$T\in \overline{cv_0(F_n)}^s$. 
Choose any free 
basis $e_1,\dots,e_n$ for $F_n$. Let $v\in T$ be any vertex and let 
$\ell =\max\{{\rm dist}(v,e_iv)\}$. 
Then the corank one free factor with basis $e_1,\dots,e_{n-1}$ is 
$\ell$-short for $T$. 
As in the first part of this proof, this factor is 
$\ell+1$-short for every tree in some neighborhood
of $T$ in $\overline{cv_0(F_n)}^s$.
%Now the graph of groups decomposition $T/F_n$ has 
%an edge with trivial edge group, and this holds true
%for every nearby tree as well. 
The second part of the corollary now follows from 
Lemma \ref{short}.
%
%The third part of the corollary follows from Lemma \ref{short} and 
%the same argument.
\end{proof}

\begin{remark}\label{noneigh}
By the main result of
\cite{G00}, any neighborhood in 
$\overline{cv(F_n)}$ of a simplicial tree
$T\in \overline{cv_0(F_n)}^s-\overline{cv_0(F_n)}^{++}$ contains
trees with dense orbits. Moreover, it is unclear whether 
all neighborhoods of $T$ contain points $S\not=T$ in 
$\overline{cv_0(F_n)}^+$. 
\end{remark}

%By construction, the set $cv(F_n)$ of free simplicial trees
%is dense in $\overline{cv(F_n)}$ for the equivariant
%Gromov Hausdorff topology. 

\begin{corollary}\label{simpfinite}
Let $T\in \overline{cv_0(F_n)}^s$ be a simplicial tree
and let $T_i\subset \overline{cv_0(F_n)}^{++}$ be a sequence
converging to $T$. Then ${\rm diam}(\Upsilon\{T_i\mid i\})<\infty$.
\end{corollary}

Fix once and for all a number
$k\geq 4$. Quantitative versions of all statements in the sequel
will depend on this choice of $k$,
but for simplicity we will drop this dependence
in our notations. 

Let $cv_0(F_n)$  be the space of all free simplicial $F_n$-trees
with volume one quotient.  
For a number $\ell\leq n-1$ 
call two trees
$T,T^\prime\in cv_0(F_n)$ \emph{$\ell$-tied} if there is a free factor 
$A$ of $F_n$ of rank $\ell$ which is $k$-short
for both $T,T^\prime$.
For trees $T\not=T^\prime\in cv_0(F_n)$ let $d_{ng}^\ell(T,T^\prime)$
be the minimum of all numbers $s\geq 1$ with the 
following property. There is a sequence $T=T_0,\dots,T_s=T^\prime
\subset cv_0(F_n)$ so that for all $i$ 
the trees $T_i,T_{i+1}$ are $\ell$-tied. 
Define moreover $d_{ng}^\ell(T,T)=0$ for all $T\in cv_0(F_n)$.

\begin{proposition}\label{distance}
For all $1\leq \ell\leq n-1$, 
the function 
\[d_{ng}^\ell:cv_0(F_n)\times cv_0(F_n)\to \mathbb{N}\]
is a distance on $cv_0(F_n)$.
\end{proposition}
\begin{proof}
Symmetry of $d_{ng}^\ell$ is immediate 
from the definition, as well as the fact that
$d_{ng}^\ell(T,T^\prime)=0$
if and only if $T=T^\prime$.
The triangle inequality is built into the construction, so all we have
to show is that $d_{ng}^\ell(T,T^\prime)<\infty$ for all $T,T^\prime\in
cv_0(F_n)$. For $\ell=n-1$, this is a consequence of  
Lemma \ref{equivariantfs} below 
and its proof, and the case $\ell\leq n-2$ is established
in Lemma \ref{equivariantfs2}. 
\end{proof}

%In the next lemma, the map $\Upsilon$ is defined as in 
%equation (\ref{upsilon}). 

\begin{lemma}\label{equivariantfs}
The map $\Upsilon:(cv_0(F_n),d_{ng}^{n-1})\to {\cal F\cal S}$ 
is a coarsely ${\rm Out}(F_n)$-equi\-va\-riant quasi-isometry.
\end{lemma}
\begin{proof} Recall that 
the map $\Upsilon$ is coarsely surjective.
We begin with showing that $\Upsilon$
is coarsely $L$-Lipschitz for some $L=L(k)\geq 1$.

Let $d_{\cal F\cal S}$ be the 
distance in ${\cal F\cal S}$.
By definition, 
if $T,T^\prime\in cv_0(F_n)$ 
are $(n-1)$-tied then there is a free factor
$A$ of $F_n$ of rank $n-1$ which is $k$-short for
both $T,T^\prime$. By Lemma \ref{short}, 
\[d_{\cal F\cal S}(A,\Upsilon(T))\leq c(k),\,
d_{\cal F\cal S}(A,\Upsilon(T^\prime))\leq c(k)\]
and hence
$d_{\cal F\cal S}(\Upsilon(T),\Upsilon(T^\prime))\leq 2c(k)$.
Thus $\Upsilon$ is $2c(k)$-Lipschitz by an iterated application of the
triangle inequality (remember that $d_{ng}^{n-1}$ only assumes integral
values).

We are left with showing that the map $\Upsilon$ 
coarsely decreases distances by at most a fixed 
positive multiplicative constant. 
Choose for every 
free factor $A$ of $F_n$ of corank one  
a marked 
rose $R(A)$ with $n$ petals of length $1/n$ each which
represents this factor $A$, i.e. such that $n-1$ petals 
of $R(A)$ generate $A$. We denote 
by $\tilde R(A)\in cv_0(F_n)$ 
the universal covering of $R(A)$.
By construction, $d_{\cal F\cal S}(A,\Upsilon(\tilde R(A)))=1$.

Let $\hat{\cal F\cal S}$ be the the complete subgraph 
of the free splitting graph (not of its first barycentric subdivision)
whose vertex set consists of one-loop free splittings of $F_n$.
It is well known that the vertex inclusion extends to a 
coarsely ${\rm Out}(F_n)$-equivariantly quasi-isometry
$\hat{\cal F\cal S}\to {\cal F\cal S}$ (see e.g. \cite{H14b} for 
a detailed proof in the case of the curve graph which carries
over word by word).

Since $\hat {\cal F\cal S}$ is a metric graph and, in particular,
a geodesic
metric space,
and since for all corank one free factors 
$A$ the distance beween $A$ and
$\Upsilon(\tilde R(A))$ equals one,  
it now suffices to show the following.
Whenever two one-loop graph of groups decompositions
$A,C$ of $F_n$ are of distance one in $\hat{\cal F\cal S}$,
i.e. if they 
have a common refinement, then 
 $d_{ng}^{n-1}(\tilde R(A),\tilde R(C))\leq 2$.

Thus 
let $A<F_n,C<F_n$ be corank one free factors defining
one-loop graph of groups decompositions with a common refinement.
Then up to replacing 
one of these free factors  by a conjugate,
the intersection $B=A\cap C$ is a free factor
of $F_n$ of rank $n-2$, and 
there is a free splitting 
$F_n=U*B*D$ with $U*B=A$ and $B*D=C$.

Let $G$ be a metric rose with $n$ petals of length $1/n$ 
which represents this splitting, with 
univeral covering $\tilde G\in cv_0(F_n)$. Then both corank one 
free factors $A,C$ of $F_n$ are $2$-short for $\tilde G$.
This implies that the distance with respect to 
the metric $d_{ng}^{n-1}$ 
between $\tilde R(A)$ and $\tilde G$ is at most one, and 
the same holds true for the 
distance between $\tilde G$ and $\tilde R(C)$. 
Therefore we have $d_{ng}^{n-1}(\tilde R(A),\tilde R(C))\leq 2$
which is what we wanted to show.
The lemma is proven.
\end{proof}

The final goal of this section is to 
give a geometric
interpretation of the free factor graph.
We will take a slightly more
general viewpoint and introduce $n-3$ 
additional intermediate ${\rm Out}(F_n)$-graphs.

Fix a number $\ell\leq n-2$. Define a graph
${\cal F\cal F}_\ell$ as follows. Vertices of 
${\cal F\cal F}_\ell$ are free factors of rank $n-1$.
Two such vertices $A,B$ are connected by an edge of length one
if and only if up to conjugation, the intersection $A\cap B$ 
contains a free factor of rank $\ell$. Note that
the graphs ${\cal F\cal F}_\ell$ all have the same set of vertices, and
for each $\ell\leq n-2$ the graph ${\cal F\cal F}_{\ell}$
can be obtained from the graph ${\cal F\cal F}_{\ell-1}$ 
by deleting some edges. In particular,
the vertex inclusion extends to a one-Lipschitz
map ${\cal F\cal F}_\ell\to {\cal F\cal F}_{\ell-1}$.
The group ${\rm Out}(F_n)$ acts as a group of 
simplicial automorphisms on each of the graphs
${\cal F\cal F}_\ell$.
We have

\begin{lemma}\label{freefactor}
The graph ${\cal F\cal F}_1$ is coarsely ${\rm Out}(F_n)$-equivariantly
quasi-isometric to the free factor graph.
\end{lemma}
\begin{proof} Associate to a free factor $A$ in $F_n$ a free factor
$h(A)>A$ of corank one. If the free factors $A,B$ are connected by an
edge in the free factor graph, then up to conjugation
and exchanging $A$ and $B$ we have $A<B$. But this just
means that up to conjugation, $h(A)$ and $h(B)$ intersect
in the free factor $A$ and hence $h(A)$ and $h(B)$ are connected
by an edge in ${\cal F\cal F}_1$. In other words,
the map $h$ which maps 
the vertex set of the free factor graph into the vertex set of the 
graph ${\cal F\cal F}_1$ 
is one-Lipschitz with respect to the metric of ${\cal F\cal F}$ 
and the metric of ${\cal F\cal F}_1$.
Moreover, $h$ is coarsely ${\rm Out}(F_n)$-equivariant.

Now if $A,B$ are free factors of rank $n-1$ which are
connected by an edge in ${\cal F\cal F}_1$ then up to conjugation,
they intersect in some free factor $C$. Then the distance between
$A,C$ and between $B,C$ in ${\cal F\cal F}$ equals one and hence
the distance between $A,B$  
in ${\cal F\cal F}$ is at most two. The lemma follows.
\end{proof}

\begin{remark}
The dual of a graph $G$ is a graph $G^\prime$ 
whose vertex set is the set of edges of $G$ and where two 
distinct vertices of $G^\prime$ 
are connected by an edge if the corresponding edges of $G$ are
incident on the same vertex. For $\ell\leq n-2$, 
the dual ${\cal F\cal F}_\ell^\prime$ 
of ${\cal F\cal F}_\ell$ has a simple description. 
Its vertex set consists of free factors of rank $\ell$, and two such 
free factors are connected by an edge if up to conjugation, they
are subgroups of the same free factor of rank $n-1$.
Although we do not use these dual graphs directly, they
are more closely related to the geometric description of the
graphs ${\cal F\cal F}_\ell$ given in Lemma \ref{equivariantfs2}
below.
\end{remark}

For each $\ell\leq n-2$ let 
\[\Upsilon_\ell:cv_0(F_n)\to {\cal F\cal F}_\ell\]
be a map which associates to a tree $T$ a $k$-short
corank one free factor- in fact we can take the same
map for every $\ell$ and view its images as vertices 
in any one of the graphs ${\cal F\cal F}_\ell$.

The statement of the following lemma is only needed in the
case $\ell=1$. 

\begin{lemma}\label{equivariantfs2}
For $\ell\leq n-2$ the map 
$\Upsilon_\ell:(cv_0(F_n),d_{ng}^\ell)\to {\cal F\cal F}_\ell$ 
is a coarsely  ${\rm Out}(F_n)$-equivariant quasi-isometry.
\end{lemma}
\begin{proof} The proof of the lemma is similar to the
proof of Lemma \ref{equivariantfs}. Note first that the
map $\Upsilon_\ell$ is coarsely surjective.

Let $\ell\leq n-2$ and let
$A<F_n$ be a free factor of rank $\ell$. 
Let $T\in cv_0(F_n)$ be such that $A$ is $k$-short for $T$.
Recall from (\ref{upsilon}) the definition of the
map $\Upsilon:cv_0(F_n)\to {\cal F\cal S}$.
We claim that there is a 
free factor $B>A$ of rank $n-1$
whose distance to $\Upsilon(T)$ in the free splitting graph
is bounded from above by a number only depending on $k$.

Let as before $R_\ell$ be a rose with $\ell$ petals of equal length one and 
let $f:R_\ell\to T/F_n$ be a $k$-Lipschitz map such that
$f_*(\pi_1(R_\ell))$ is conjugate to $A$.
Let $e\subset T/F_n$ be an edge of length at least
$1/(3n-4)$. Then $e$ is 
covered
by $f(R_\ell)$ at most $k(3n-4)\ell$ times.

As in the proof of
Lemma \ref{short}, assume first that the edge $e$ 
of $T/F_n$ is non-separating.
Collapse a maximal forest in $T/F_n$ not containing $e$ 
to a point and let $V$ be the resulting rose. Let $e_0\subset V$
be the image of $e$ under the collapsing map.

Embed $V$ 
into the manifold $M=\sharp_n(S^2\times S^1)$ to define a marked isomorphism of
fundamental groups. 
Let $S_0\subset M$ be an embedded 
sphere which intersects the interior of 
the petal $e_0$ of $V\subset M$ in a single
point and has no other intersection with $V$.
Then $S_0$ defines a one-loop free splitting
of $F_n$ obtained by collapsing the complement
of $e_0$ in $V$ to a point. Its 
vertex group is up to conjugation the 
group of all homotopy classes of 
loops in $M$ based at a point $x\in M-S_0$ which do not
intersect $S_0$. The 
distance in the graph ${\cal F\cal S}$ 
between $\Upsilon(T)$ and the 
one-loop free splitting defined by $S_0$ equals one.

Let $G\subset M$ be an embedded rose 
with $\ell$ petals, with vertex $x\in M-S_0$ 
and with fundamental group the free factor $A$.
We may assume that the rose $G$ intersects 
the sphere $S_0$ in at most 
$k(3n-4)\ell$ points.  
Let $\Sigma\subset M$ be a sphere which is disjoint from 
the rose $G$. Such a sphere exists since the rank of $A$ 
equals $\ell\leq n-2$. 
The free factor $A$ is a subgroup of the corank one 
free factor of $F_n$ which is the vertex group $B$ 
of the splitting defined by $\Sigma$.

If 
$\Sigma$ is disjoint from $S_0$ then $S_0,\Sigma$ define free
splittings of $F_n$ of distance two in ${\cal F\cal S}$. 
The distance in 
the free splitting graph 
between $B$ and $\Upsilon(T)$ is at most four, so the
above claim follows in this case.

If $\Sigma$ intersects $S_0$ then 
use $\Sigma$
to surger $S_0$ so that the intersection number with the rose $G$ is
decreased as follows. 
Choose a disc component $D$ of $\Sigma-S_0$.
Its boundary divides $S_0$ into two discs $D_1,D_2$.
Assume that $D_1$ is the disc with fewer intersections with $G$.
Then $D\cup D_1=S_1$ is a sphere disjoint from $S_0$ with 
at most $k(3n-4)\ell/2$ intersections with $G$. The
sphere $S_1$ defines a free splitting of $F_n$ whose 
distance in ${\cal F\cal S}$ to the
splitting defined by $S_0$ equals two 
(we refer to \cite{HH12} for
details of this well known construction).

Repeat this construction with $S_1$ and $\Sigma$. 
After at most $\log_2 k(3n-4)\ell$ such surgeries the 
resulting sphere $Q$  is disjoint from $G$. 
The sphere $Q$ determines a corank one free factor
of $F_n$ containing $A$ whose distance to $\Upsilon(T)$ in 
${\cal F\cal S}$ is at most $1+\log_2k(3n-4)\ell$.

As a consequence, if $e$ is non-separating then 
in the free splitting graph, $\Upsilon(T)$ is at uniformly
bounded distance from a one-loop graph of groups 
decomposition $F_n=C*$ with trivial edge group and 
vertex group $C>A$. 
The same reasoning also applies if the edge $e$ is separating- in
this case the sphere $S_0$ is chosen to be separating as well.

To summarize, if the free factor 
$A<F_n$ is $k$-short for $T\in cv_0(F_n)$ then
the distance in ${\cal F\cal S}$ between $\Upsilon(T)$ and 
some corank one free factor $B>A$ is uniformly bounded.
By the definition of the metric $d_{ng}^\ell$ and of
the graph ${\cal F\cal F}_\ell$, 
this implies that
whenever $T,T^\prime\in cv_0(F_n)$ are such that
$d_{ng}^\ell(T,T^\prime)=1$ then the distance in ${\cal F\cal F}_\ell$
between $\Upsilon(T)$ and $\Upsilon(T^\prime)$ is uniformly bounded.
By an iterative application of the triangle inequality, 
this shows that the map 
$\Upsilon_\ell:(cv_0(F_n),d_{ng}^\ell)\to {\cal F\cal F}_\ell$ 
is coarsely Lipschitz.

To show that the map $\Upsilon_\ell$ coarsely decreases distances
at most by a fixed positive multiplicative constant, note that 
if two corank one free factors $A,A^\prime$ of $F_n$ are connected
by an edge in ${\cal F\cal F}_\ell$ then there are free bases
$e_1,\dots,e_n$ and $e^\prime_1,\dots,e_n^\prime$ of $F_n$ such that
$e_i=e_i^\prime$ for $1\leq i\leq \ell$ and that
$A=<e_1,\dots,e_{n-1}>,A^\prime=<e_1^\prime,\dots,e_{n-1}^\prime>$
up to conjugation.
The universal coverings $R,R^\prime\in cv_0(T)$ of the marked roses 
with petals of length $1/n$ determined by these bases are 
$\ell$-tied. As in the proof of Lemma \ref{equivariantfs}, 
this yields the lemma.
\end{proof}

\begin{remark} 
Purists may find the simultaneous use 
of the sphere graph 
and of Outer space 
as a model for the free splitting
graph  unsatisfactory. However 
we felt that using spheres makes the proof of 
Lemma \ref{equivariantfs2}  more transparent.
We invite the reader to find another proof along the 
lines of the proofs of Lemma A.3 and Lemma 4.1 of \cite{BF12}
\end{remark}

\section{Skora paths}\label{skora}

The goal of this section is to introduce folding paths 
and establish some of their properties needed later on.

A \emph{morphism} between $F_n$-trees 
$S,T$ is an equivariant map $\phi:S\to T$ such that every segment
of $S$ can be subdivided into 
finitely many subintervals on which $\phi$ is 
an isometric embedding. If the tree $S$ is simplicial then
we require that the restriction of $\phi$ to any edge of $S$
is an isometric embedding.

The following (well known) construction is taken 
from Section 2 of \cite{BF11}. 
Let for the moment $U$ be an arbitrary $F_n$-tree.
A \emph{direction}
at a point $x\in U$ is a germ of non-degenerate
segments $[x,y]$ with $y\not=x$. At each interior point of 
an edge of $U$ there are exactly two directions.
A collection of directions at $x$ is called a \emph{gate}
at $x$. A \emph{turn} at $x$ is an unordered pair of distinct directions 
at $x$. It is called \emph{illegal} if the directions belong to the
same gate, and it is called \emph{legal} otherwise.
A \emph{train track structure} 
on $U$ is an $F_n$-invariant family of gates at the points of $U$
so that at each $x\in U$ there are at least two gates.

A morphism $\phi:S\to T$ determines a collection of gates
as follows. Define a turn in $S$ to be 
illegal if it is given by 
two directions which are identified by the morphism $\phi$.
Otherwise the turn is defined to be legal.
Two directions $d,d^\prime$ at the same point
belong to the same gate if either $d=d^\prime$ or if the turn
$d,d^\prime$ is illegal.
If these gates determine a train track structure on $S$ and if 
moreover there is a train track structure on $T$ so that 
legal turns are sent to legal turns, then
$\phi$ is called a \emph{train track map}
(see p.110 of \cite{BF11}).

Let as before $cv(F_n)$ be unprojectivized
Outer space, with its boundary $\partial cv(F_n)$ of
minimal very small actions of $F_n$ on $\mathbb{R}$-trees
which are either not simplicial or which are not free.
As in Section \ref{geometric}, denote by $cv_0(F_n)\subset cv(F_n)$ the 
subspace of trees with quotient of volume one.
Let  ${\rm CV}(F_n)$ be the projectivization of 
$cv(F_n)$, with its boundary 
$\partial {\rm CV}(F_n)$.
The union 
\[\overline{{\rm CV}(F_n)}={\rm CV}(F_n)\cup
\partial{\rm CV}(F_n)\] is a compact space. 
The outer automorphism group ${\rm Out}(F_n)$ of 
$F_n$ acts on $\overline{{\rm CV}(F_n)}$ as a group of 
homeomorphisms.
The projection map restricts to an ${\rm Out}(F_n)$-equi\-va\-riant
homeomorphism $cv_0(F_n)\to {\rm CV}(F_n)$.

There is a natural bijection between
conjugacy classes of 
free bases of $F_n$ and roses 
(=marked metric roses with $n$ petals of length $1/n$ each).
Define the
\emph{standard simplex} of a free basis of $F_n$ to consist
of all simplicial trees 
\[U\in \overline{cv_0(F_n)}^{++}\subset 
\overline{cv_0(F_n)}^+\subset \overline{cv(F_n)}
=cv(F_n)\cup \partial cv(F_n)\] 
with quotient of volume one
which are universal coverings of 
graphs obtained from the rose $R$
corresponding to the basis 
by changing the lengths of the edges.
We allow that $U$ is contained in the 
boundary of unprojectivized Outer space, i.e. that 
the rank of the  
fundamental group of the graph $U/F_n$ is
strictly smaller than $n$.
A standard simplex is a compact subset of the space 
$\overline{cv_0(F_n)}^{++}$
of simplicial $F_n$-trees with trivial 
edge stabilizers and quotient of volume one.

In the proof of Lemma \ref{morphismsex} below
and several times later on we will use
the main result of \cite{P88} which we report here for 
easy reference. We do not define the equivariant 
Gromov Hausdorff topology, but we note that this is 
the topology used for $\overline{cv(F_n)}$ (see \cite{P89} for 
a precise statement).

\begin{theo}\label{paulin}
Let $(X_i)$ be a sequence of complete $\mathbb{R}$-trees.
Let $\Gamma$ be a countable group acting
by isometries on all $X_i$. Suppose 
that for each $i$ there exists a point $x_i\in X_i$ such that
the following holds true.
For every finite subset $P$ of $\Gamma$, the closed convex hulls
of the images of $x_i$ under $P$ admit for all $\epsilon >0$ a covering
by balls of radius $\epsilon$ of uniformly bounded cardinality.
Then there exists a subsequence converging in the equivariant
Gromov Hausdorff topology to an $\mathbb{R}$-tree.
\end{theo}

The construction in Lemma \ref{morphismsex} below 
is discussed in detail in Section 2 of \cite{BF11}.
We also refer to this paper for references
to earlier works where this construction is introduced. 
In the rest of the paper, we always denote by 
$[T]\in \overline{{\rm CV}(F_n)}$ the projectivization of a 
tree $T\in \overline{cv(F_n)}$.

\begin{lemma}\label{morphismsex}
For every $[T]\in \overline{{\rm CV}(F_n)}$
and every standard simplex 
$\Delta\subset \overline{cv_0(F_n)}^{++}$ there is a
tree $U\in \Delta$ and 
a train track map $f:U\to T$
where $T$ is some representative of $[T]$.
\end{lemma}
\begin{proof} In the case that $[T]\in {\rm CV}(F_n)$ 
is free simplicial 
a detailed argument is given in the proof of 
Proposition 2.5 of \cite{BF11} (see also \cite{FM11}).
A limiting argument then yields the result for 
trees
$[T]\in \partial {\rm CV}(F_n)$.

Let $\Sigma\subset \overline{cv(F_n)}$ be the set
of all trees $U\in \overline{cv(F_n)}$ 
such that the minimum over all trees $S\in \Delta$ 
of the smallest Lipschitz constant for
equivariant maps $S\to U$ equals one.
Let $[T_i]\subset {\rm CV}(F_n)$ be a sequence of 
free simplicial projective $F_n$-trees which converge
in $\overline{{\rm CV}(F_n)}$ 
to a projective $F_n$-tree $[T]$.
By Proposition 2.5 of \cite{BF11}, for each $i$ 
there is a point $S_i\in \Delta$, a 
representative $T_i\in \Sigma$ of $[T_i]$ and
a train track map  
\[f_i:S_i\to T_i\] of Lipschitz
constant one.

Since $\Delta\subset \overline{cv_0(F_n)}^{++}$ is compact,
the set $\Sigma$ 
is compact with respect to the equivariant Gromov 
Hausdorff topology \cite{P88,P89} and hence 
\[{\cal P}=\{U\times V\mid U\in \Delta,V\in \Sigma\}\]
is compact in the equivariant Gromov Hausdorff topology
for the diagonal action of $F_n$.
Then $S_i\times T_i\in {\cal P}$ for each $i$
and therefore up to passing to a subsequence, we may
assume that $S_i\times T_i$ converges in the equivariant
Gromov Hausdorff topology to $S\times T\in {\cal P}$ where
$S\in \Delta$ and where 
$T\in \Sigma$ is a representative of $[T]$
(see Theorem \ref{paulin} and \cite{P88}).

For each $i$ the graph $A_i$ of $f_i$ is a closed 
$F_n$-invariant subset
of $S_i\times T_i$.
By passing to another subsequence we may assume
that the sequence $A_i$ converges in the equivariant 
Gromov Hausdorff topology
to a closed 
$F_n$-invariant subset $A$ of $S\times T$
(see once more Theorem \ref{paulin} and \cite{P88}).

Let $d_S,d_T$ be the distance on $S,T$.  
For each $i$ the set $A_i\subset S_i\times T_i$ 
is the graph of a one-Lipschitz map
$S_i\to T_i$. Thus by continuity, for 
$(x_1,y_1),(x_2,y_2)\in A\subset S\times T$ we have
$d_T(y_1,y_2)\leq d_S(x_1,x_2)$. In particular,
for each $x\in S$ there is a unique point $f(x)\in T$ so that
$(x,f(x))\in A$, and the assignment $x\to f(x)$ is 
an equivariant map $S\to T$ with Lipschitz constant one.
Now $T\in \Sigma$ and therefore  the Lipschitz constant one is optimal.
The map $f$ is an isometry on edges since this
is the case for each of the maps $f_i$.
Thus $f$ is a morphism which induces a collection of gates on $S$.
Since the Lipschitz constant one for equivariant maps
$S\to T$ is optimal, there are 
at least two gates
at each preimage of the unique vertex of $S/F_n$
(compare the end of the proof of Proposition 2.5 of 
\cite{BF11} for details). In other words, 
$f$ has the required properties.
\end{proof}

\begin{remark} \label{optimal}
The train track map $f:U\to T$ constructed in the
proof of Lemma \ref{morphismsex} minimizes the Lipschitz
constant among all equivariant maps from points in 
the standard simplex $\Delta$ to $T$. We call such a map
\emph{optimal}.
\end{remark}

\begin{remark}\label{short2} 
The proof of Lemma \ref{short} shows the following
special case of  Lemma A.3 of \cite{BF12}. Let $x,y\in cv(F_n)$ and assume that
there is a train track map $f:x\to y$ which maps a non-degenerate
segment $c\subset x$ isometrically to a non-degenerate segment
$f(c)$ so that $f^{-1}(f(c))=c$; then the distance in ${\cal F\cal S}$ 
between $\Upsilon(x)$ and $\Upsilon(y)$ is uniformly bounded.
%This also holds true for simplicial trees
%$x,y\in \overline{cv(F_n)}^{++}$.
%(see Lemma \ref{estimate} for a more precise 
%statement).
\end{remark}

Let $S\in \overline{cv(F_n)}^s$ 
be a simplicial tree, let 
$T\in \overline{cv(F_n)}$ and let 
$f:S\to T$ be a train track map.  By the definition of a 
train track map, 
$f$ isometrically embeds every edge. 

Let $\epsilon >0$ be half of the smallest length of 
an edge of $S$. Let  
$e,e^\prime$ be edges with the same initial
vertex $v$ which define an illegal turn. 
Assume that $e,e^\prime$ are parametrized by
arc length on compact intervals $[0,a],[0,a^\prime]$
where $e(0)=e^\prime(0)=v$.
Then 
there is some
$t\in (0,\epsilon]$ so that $f e[0,t]=
f e^\prime[0,t]$. 
For $s\in [0,t]$ 
let $S_s$ be the quotient of $S$ by the equivalence
relation $\sim_s$ which is defined by 
$u\sim_s v$ if and only if 
$u=ge(r)$ and $v=ge^\prime(r)$ for some $r\leq s$
and some $g\in F_n$.
The tree $S_s$ is called a \emph{fold} of $S$ 
obtained by folding the illegal turn
defined by $e,e^\prime$ (once again, compare the
discussion in Section 2 of \cite{BF11}).
Note that for $s>0$ the volume of the graph 
$S_s/F_n$ is strictly smaller than the volume of $S/F_n$.
There also is an obvious notion of a maximal
fold at the illegal turn defined by $e,e^\prime$.

Using the terminology of the previous paragraph, 
the assignment $s\to S_s$  
$(s\in [0,t])$ 
is a path in $\overline{cv(F_n)}^s$ through $S_0=S$
which is called a \emph{folding path}. We call the train 
track map $f:S\to T$ used to construct the path a \emph{guide}
for the path.
The semigroup property holds for folding paths.
For each $s\in [0,t]$ the train track map $f:S\to T$ decomposes as 
\[f=f_s\circ \phi_s\] where 
$f_s:S_s\to T$ is a train track
map and $\phi_s:S\to S_s$ is a train track map for the train track
structures on $S,S_s$ defined by $f$ and $f_s$. 
We refer to \cite{HM11} and to 
Section 2 of \cite{BF11} for details of this construction.
We insist that
we view the initial train track map $f:S\to T$ as part of the data
defining a folding path.

Repeat this construction with $S_t$ and a 
perhaps different pair of edges. 
The path constructed in this way by successive foldings
terminates if 
$T$ is free simplicial
(Proposition 2.2 of \cite{BF11}). 

We next observe that with perhaps the exception of 
the endpoint, 
such folding paths are in fact contained
in the subspace 
$\overline{cv(F_n)}^+$ 
of all simplicial trees
with at least one orbit of edges with trivial stabilizer.

\begin{lemma}\label{someedgetriv}
Let $U\in \overline{cv(F_n)}^s$ be a simplicial
tree and assume that there is a non-trivial folding path
$(x_t)\subset \overline{cv(F_n)}$ 
connecting $U=x_0$ to a tree $U\not=T\in \overline{cv(F_n)}$.
Then $U\in \overline{cv(F_n)}^+$. If all stabilizers of 
non-degenerate segments in $T$ are trivial 
then $U\in \overline{cv(F_n)}^{++}$.
\end{lemma}
\begin{proof}
Let $f:U\to T$ be a train track map which 
guides the folding path $(x_t)$. Let $e_1,e_2$ be two edges
in $U$ incident on the same vertex $v$ 
with the property that the map $f$ identifies
two non-degenerate initial proper
subsegments $e_1^\prime,e_2^\prime$ 
of $e_1,e_2$. Assume that the stabilizers
$A_1,A_2<F_n$ of both edges $e_1,e_2$ are non-trivial. 
Since $U$ is very small, the groups $A_1,A_2$ are maximal cyclic.
By equivariance, the non-degenerate segments $f(e_i)\subset T$
are stabilized by $A_i$ $(i=1,2)$ 
and hence $f(e_1^\prime)=f(e_2^\prime)
\subset f(e_1)\cap f(e_2)$
is stabilized by the subgroup of $F_n$ generated by $A_1$ and $A_2$.
Since $T$ is very small and hence 
stabilizers of non-degenerate 
segments in $T$ are maximal cyclic,
we conclude that $A_1=A_2=A$.

Fold the tree $U$ and equivariantly identify 
$ge_1^{\prime}$ and
$ge_2^{\prime}$ for all $g\in F_n$. 
The resulting tree $U^\prime$ 
contains a tripod which consists of the images of 
the identified 
segments $e_1^{\prime},e_2^{\prime}$ and
the images of the segments $e_i-e_i^{\prime}$
$(i=1,2)$. As both $e_1,e_2$ are stabilized by $A$, by 
equivariance 
this tripod is stabilized by $A$ as well. But
$U^\prime$ is contained in a folding path connecting
$U$ to $T$ and hence $U^\prime$ is very small. In particular,
stabilizers of tripods are trivial. This is a contradiction which 
implies that the stabilizer of at least one of the edges
$e_1,e_2$ is trivial. The 
first part of the lemma follows.

To show the second part, note that
if there is some $t\geq 0$ so that $x_t$ has a non-trivial
edge stabilizer then by equivariance and the fact that 
the train track map $x_t\to T$ maps edges isometrically,
there is a non-degenerate segment in $T$ with non-trivial 
edge stabilizer. Thus  indeed
$(x_t)\subset \overline{cv(F_n)}^{++}$
if all stabilizers of non-degenerate
segments in $T$ are trivial.
\end{proof}

\begin{remark}\label{upsilonex}
Lemma \ref{someedgetriv} shows in particular that
the map $\Upsilon$ as defined in (\ref{upsilon}) of Section 2
is defined on every folding path, perhaps 
with the exception of its endpoint. 
\end{remark}

We can also fold all illegal turns with unit speed at once
\cite{BF11}. The resulting path is unique (Proposition 2.2 of \cite{BF11}).
In the same vein 
we can rescale all trees along the path 
to have volume one quotient and rescale the endpoint tree
$T$ accordingly while folding with unit speed all 
illegal turns at once. Using Proposition 2.2 and 
Proposition 2.5 of \cite{BF11} (compare
\cite{BF11} for references), a path constructed in this
way from a train track map $f:S\to T$ 
is unique and will be called a \emph{Skora path}
in the sequel (however, the path
depends on the guiding train track map $f:S\to T$).
If $T\in cv(F_n)$ then 
this path has finite length, otherwise its length
may be infinite. By convention, even if the path has finite
length
we do not consider its endpoint to be a point on the path.

%Let 
%\[\overline{cv_0(F_n)}^{++}\subset \overline{cv_0(F_n)}^+\]
%be the set of all simplicial trees with volume one quotient and
%without edges with non-trivial stabilizer. 

We say that a (normalized or unnormalized) folding path 
or Skora path $(x_t)_{0\leq t <\xi}$ $(\xi\in [0,\infty])$ 
\emph{converges} to a projective 
tree $[T]\in \overline{{\rm CV}(F_n)}$ if as $t\to \xi$
the projective trees $[x_t]$ converge in $\overline{{\rm CV}(F_n)}$
to $[T]$. If $(x_t)$ is a Skora path then this is equivalent to stating
that $(x_t)$ is guided by a train track map $x_0\to T$.  
We summarize the discussion as follows (see also Proposition 2.5 of 
\cite{BF11}).

\begin{lemma}\label{pathsexist}
For every standard simplex $\Delta$ and every 
tree $[T]\in\overline{{\rm CV}(F_n)}$ 
there is a Skora path $(x_t)\subset
\overline{cv_0(F_n)}^+$ with 
$x_0\in \Delta$ which
converges to $[T]$.
%If the stabilizers of all non-degenerate
%arcs of $T$ are trivial 
%then $(x_t)\subset \overline{cv_0(F_n)}^{++}$.
%With perhaps the exception of its endpoint, this path is 
%entirely contained in $\overline{cv_0(F_n)}^+$.
\end{lemma}
\begin{proof} By Lemma \ref{morphismsex}, there is a
representative $T$ of $[T]$, a tree $U\in \Delta$ and
a train track map $f:U\to T$.
This train track map then determines a unique
Skora path $(x_t)$ issuing from $U$. 
The projectivizations
$[x_t]$ of the trees $x_t$ converge as $t\to \infty$
in $\overline{{\rm CV}(F_n)}$  to 
the projectivization $[T]$ of $T$ (see \cite{FM11,BF11}).
\end{proof}

In the sequel we will often use volume renormalization
to define a Skora path $(x_t)\subset cv_0(F_n)$.
However, most of the time we consider 
\emph{unnormalized} Skora paths, i.e. we scale 
the trees along the path $(x_t)$ 
in such a way that the train track 
maps along the path are edge isometries onto
a fixed endpoint tree $T$. Note that our parametrization
of such unnormalized Skora path does not coincide with the 
parametrization used in \cite{BF11} although the paths coincide
as sets.

\section{Alignment preserving maps}\label{alignment}

For a number $L>1$, an \emph{$L$-quasi-geodesic} 
in a metric space $(X,d)$ is a map $\rho:J\subset \mathbb{R}\to X$
such that 
\[\vert s-t\vert/L-L\leq d(\rho(s),\rho(t))\leq L\vert s-t\vert +L\]
for all $s,t\in J$. The path $\rho$ is called a \emph{reparametrized
$L$-quasi-geodesic} if there is a homeomorphism
$\psi:I\to J$ such that $\rho\circ \psi:I\to X$ is an $L$-quasi-geodesic. 

Recall from (\ref{upsilon}), (\ref{upsilonc}), (\ref{ff})
of Section 2 the definition of the maps $\Upsilon,\Upsilon_{\cal C},
\Upsilon_{\cal F}$. 
A \emph{liberal folding path} is a folding path where folding
occurs with any speed, and there may also be rest intervals.
Building on the work \cite{HM13}, it was observed in  
\cite{BF12} that there is a number $L>1$ such that
the image under $\Upsilon$
of any liberal folding path in $\overline{cv(F_n)}^{++}$ 
is an $L$-unparametrized quasi-geodesic,
and the same holds true
for the maps $\Upsilon_{\cal C}$ \cite{Mn12}
and $\Upsilon_{\cal F}$ \cite{BF11}.
Moreover, 
following the construction in the proof of 
Proposition 2.2 of \cite{BF11},
if $(x_t)$ is a liberal folding path converging to a 
projective tree $[T]$, i.e. such that 
$[x_t]\to [T]$ in $\overline{\rm CV}(F_n)$, 
then there is a Skora path
$(y_t)$ converging to $[T]$ such that  
the Hausdorff distance between $\Upsilon(x_t)$ and 
$\Upsilon(y_t)$ is uniformly bounded.

By Lemma \ref{someedgetriv}, \emph{all} 
folding paths (with perhaps the exception of their endpoints) 
are contained in 
$\overline{cv(F_n)}^+$ and therefore 
the map $\Upsilon$ is defined on any
folding path. Out next goal is to show that the image
under $\Upsilon$ of \emph{any} folding path is a reparametrized 
$L$-quasi-geodesic and that the same holds true for
the maps 
$\Upsilon_{\cal C},\Upsilon_{\cal F}$.
The main result of this section (Proposition \ref{onelipproj} below)
will be used several times again in later sections.

An \emph{alignment preserving map} 
between two $F_n$-trees 
$T,T^\prime\in \overline{cv(F_n)}$ is an equivariant   
map  $\rho:T\to T^\prime$ with the property
that $x\in [y,z]$ implies 
$\rho(x)\in [\rho(y),\rho(z)]$. 
An equivariant map $\rho:T\to T^\prime$ is alignment preserving
if and only if the preimage of every point in $T^\prime$ is convex
(\cite{G00} and Def. 10.7 of \cite{R10}).
The map $\rho$ is then continuous on segments.

An example of an alignment preserving map can be obtained as follows.
Let $G$ be a finite metric graph with fundamental group $F_n$ 
and without univalent vertices and let 
$G^\prime$ be obtained from $G$ by collapsing a forest. The collapsing map
$G\to G^\prime$ lifts to a one-Lipschitz alignment preserving map 
$T\to T^\prime$ where $T,T^\prime$ is the universal covering of $G,G^\prime$.
An alignment preserving morphism is an equivariant isometry.

In the next lemma, $d$ denotes the distance in ${\cal F\cal S}$.
Versions of the lemma can be found in \cite{HM13} and
\cite{BF12}.

\begin{lemma}\label{nearby}
Let $S,S^\prime\in \overline{cv_0(F_n)}^+$ 
and assume that
there is an alignment preserving map $\rho:S\to S^\prime$;
then
$d(\Upsilon(S),\Upsilon(S^\prime))\leq 2$.
\end{lemma}
\begin{proof} 
By equivariance, an edge in $S$ with non-trivial stabilizer
is mapped by $\rho$ to a (perhaps degenerate) edge in 
$S^\prime$ with non-trivial stabilizer.
Thus $\rho$ induces an alignment preserving map 
$\rho_0:S_0\to S_0^\prime$ where 
$S_0,S_0^\prime$ 
is obtained from $S,S^\prime$ by collapsing all edges with 
non-trivial stabilizers to points. 
The map $\rho_0$ projects to a quotient map
$\hat \rho_0:S_0/F_n\to S_0^\prime/F_n$.

Let $e^\prime$ be an edge in $S_0^\prime/F_n$
and let $V$ be the graph obtained from 
$S_0^\prime/F_n$
by collapsing the complement of $e^\prime$ to a point. 
Let $\zeta:S_0^\prime/F_0\to V$ be 
the collapsing map. Then 
$\zeta\circ \hat\rho_0:S_0/F_n\to V$ collapses $S_0/F_n$ to $V$. 
As $S_0/F_n$ is the graph of groups decomposition 
defining $\Upsilon(S)$ and $S_0/F_n^\prime$ is the
graph of groups decomposition defining 
$\Upsilon(S^\prime)$, this shows that
$\Upsilon(S)$ and $\Upsilon(S^\prime)$ collapse to the same
vertex in ${\cal F\cal S}$. Thus  
the distance in ${\cal F\cal S}$ between 
$\Upsilon(S)$ and $\Upsilon(S^\prime)$ is at most two.
\end{proof}

\begin{proposition}\label{onelipproj}
There is a number $\chi>0$ with the following property.
Let $T,T^\prime\in \partial cv(F_n)$ 
and assume that there is a one-Lipschitz alignment
preserving map $\rho:T\to T^\prime$.
%Then $\psi([T])=\psi([T^\prime])$.
Let $\Delta\subset \overline{cv_0(F_n)}^{++}$ be a standard simplex
and let $(x_t)$ be any folding path connecting 
a point in $\Delta$ to $T$.
Then there is a liberal folding path $(y_t)$ connecting a point
in $\Delta$ to $T$ such that for all $t$ we have
\[d(\Upsilon(x_t),\Upsilon(y_t))\leq \chi.\]
%Then the Hausdorff distance between 
%the paths $\Upsilon(x_t)$ and 
%$\Upsilon(y_t)$ is uniformly bounded.
\end{proposition}
\begin{proof}  Let $T,T^\prime\in \partial cv(F_n)$ and 
assume that there is a one-Lipschitz alignment preserving
map \[\rho:T\to T^\prime.\]

As this property is invariant under scaling the metric
on both trees by the same constant,
given a standard simple 
$\Delta\subset \overline{cv_0(F_n)}^{++}$ 
we may assume that
there is a point 
$S\in \Delta$ and there is a train
track map \[\phi:S\to T.\] 
Then the map 
$\rho\circ \phi:S\to T^\prime$ is equivariant and one-Lipschitz.

We claim that  there is a tree 
$S^\prime\in \overline{cv(F_n)}^{++}$ which can be obtained from $S$ by
decreasing the lengths of some
edges of $S$, there is a one-Lipschitz alignment
preserving map $\alpha:S\to S^\prime$ 
and there is a morphism $\phi^\prime:S^\prime\to T^\prime$ such that
\[\phi^\prime\circ\alpha=\rho\circ \phi.\]
The map $\alpha$ may collapse
some edges of $S$ to points. 

The tree $S^\prime$ and the map $\alpha$ are constructed as follows.
Let $e$ be an edge in $S$. By definition of a train track
map, the restriction of $\phi$ to $e$ is an
isometric embedding and hence
$\phi(e)$ is a segment in $T$  whose length
equals the length $\ell(e)$ of $e$. 
The alignment preserving map $\rho$ maps $\phi(e)$ to 
a segment $\rho(\phi(e))$ of length $\ell^\prime(e)\leq \ell(e)$.
To construct $S^\prime$ we reduce the length of 
the edge $e$ to $\ell^\prime(e)\geq 0$. The 
natural map $\alpha:S\to S^\prime$ 
is one-Lipschitz, equivariant and simplicial, 
and it associates to $e$ the (possibly degenerate) 
edge in $S^\prime$ 
obtained in this way (compare with the discussion
in \cite{G00,HM13,BF12}).

%We call $f\circ \phi$ a \emph{morphism up to homotopy}. 

%Let $S^\prime$ be the tree obtained from 
%$S$ by reducing the length
%of the edges in this way. Note that $S^\prime$ is 
%uniquely determined by $S$ and $\phi,f$.

Let $a,b\subset S$ be subsegments of edges incident on 
the
same vertex $v$ which are identified by the map $\phi$. 
Then the segments $a,b$ are also identified by
$\rho\circ \phi$. 
This means the following. Let $U$ be the simplicial
tree obtained from $S$ by equivariantly 
folding the $F_n$-translates of the segments $a,b$. 
Then there is a morphism $\chi:U\to T$ which maps the identified
segments $a,b$ to a segment in $T$, and this segment is 
mapped by $\rho$ to a segment in $T^\prime$.
By equivariance of the map $\rho$, 
there is a tree $U^\prime$ which 
is be obtained from $S^\prime$ by equivariantly 
identifying the $F_n$-translates
of $\alpha(a)$ and $\alpha(b)$, i.e. by  
a (perhaps trivial) fold,
and there is a one-Lipschitz alignment preserving map 
$\beta:U\to U^\prime$ and a morphism
$\chi^\prime:U^\prime\to T^\prime$
such that
\[\chi^\prime\circ \beta= \rho\circ \chi.\]

Now let 
$(\hat x_t)$ be an (unnormalized)  Skora path
connecting $S$ to $T$. Following 
the discussion in Proposition 2.2 of \cite{BF11} and its proof,
there is an (unnormalized) folding path $(x_t)$ so that  
only a single fold is
performed at the time, and there are sequences 
$t_i\to \infty,s_i\to \infty$ such that $\hat x_{s_i}=x_{t_i}$. 
In particular, we have 
$(x_{t_i})\to T$ in the 
equivariant Gromov Hausdorff topology.
For each $t$ there is a 
train track map $h_{t}:x_{t}\to T$. The
Hausdorff distance between the path $\Upsilon(\hat x_t)$ and
the path $\Upsilon(x_t)$ is at most $\chi$ for a 
number $\chi>0$ not depending on the paths. 

The discussion in the beginning of this proof shows that 
there is a (suitably parametrized) liberal 
folding path
$(y_t)$ connecting $S^\prime$ to $T^\prime$ 
with the following property.  
For each 
$t$, there is a one-Lipschitz alignment preserving map
$\zeta_t:x_t\to y_t$, and there is a 
morphism $h_{t}^\prime:y_{t}\to T^\prime$ such that
the diagram
\[ 
\begin{CD}
x_{t}           @>h_{t}>>      T  \\  
@V\zeta_{t}VV                       @V\rho VV  \\ 
y_t           @>h_t^\prime>>  T^\prime   
\end{CD}
\]
commutes. As $t\to \infty$, 
$y_t$ converges in the equivariant Gromov Hausdorff topology
to $T^\prime$.
We refer to \cite{HM13,BF12} for
a detailed proof of this fact in the 
case that $T\in cv(F_n)$.

By Lemma \ref{nearby}, for each $t$ the distance
in ${\cal F\cal S}$ between $\Upsilon(x_t)$ and 
$\Upsilon(y_t)$ is at most two. Consequently for each
$t$ there is some $s(t)$ so that the distance
between $\Upsilon(\hat x_t)$ and
$\Upsilon(y_{s(t)}$ is at most $\chi+2$.
This shows the proposition.
%Since the 
%assignments $t\to \Upsilon(x_t)$ and 
%$t\to \Upsilon(y_t)$ are uniform reparametrized
%quasi-geodesics in ${\cal F\cal S}$, the proposition 
%now follows from Proposition \ref{coarseindependence} and
%Remark \ref{hierarchy2}.
 \end{proof}

Recall 
the definitions of 
the maps 
$\Upsilon:\overline{cv(F_n)}^+\to {\cal F\cal S}$
and $\Upsilon_{\cal C}:\overline{cv (F_n)}^s\to {\cal C\cal S}$ and 
$\Upsilon_{\cal F}=\Omega\circ 
\Upsilon_{\cal C}:\overline{cv(F_n)}^s\to 
{\cal F\cal F}.$

\begin{corollary}\label{arbitrary}
There is a number $L>1$ such that the image under $\Upsilon$
(or under $\Upsilon_{\cal C},\Upsilon_{\cal F}$) 
of every folding path is a reparametrized $L$-quasi-geodesic in 
${\cal F\cal S}$ (or ${\cal C\cal S},{\cal F\cal F}$).
\end{corollary}
\begin{proof} It suffices to show the corollary for bounded subsegments
of folding paths. Now if $S,T\in \overline{cv(F_n)}^+$ and if there is 
a folding path $(x_t)$ connecting $S$ to $T$, then 
Proposition \ref{onelipproj} shows that there are 
trees $S^\prime,T^\prime\in \overline{cv(F_n)}^{++}$ 
with $\Upsilon(S)=\Upsilon(S^\prime)$, 
$\Upsilon(T)=\Upsilon(T^\prime)$ 
and 
there is a liberal folding path $(y_t)$ connecting 
$S^\prime$ to $T^\prime$ 
such that $d(\Upsilon(x_t),\Upsilon(y_t))\leq \chi$
for all $t$.
Then $(y_t)\subset \overline{cv(F_n)}^{++}$ by Lemma \ref{someedgetriv}.
The corollary now
follows from the fact that the image under $\Upsilon$ of a liberal
folding path in $\overline{cv(F_n)}^{++}$ 
is an reparametrized $m$-quasi-geodesic for a universal
number $m>1$ \cite{HM13}.
\end{proof}

\section{Trees in $\partial {\rm CV}(F_n)$ 
without dense orbits}\label{nodense}

A tree $T\in \partial cv(F_n)$ 
decomposes 
canonically into two disjoint $F_n$-invariant subsets
$T_d$ and $T_c$. Here $T_d$ is the set of all points $p$ 
such that the orbit $F_np$ is discrete, and 
$T_c=T-T_d$. The set $T_c\subset T$ is closed.
Each of its connected components is a subtree $T^\prime$
of $T$. The stabilizer of $T^\prime$ 
acts on $T^\prime$ with
dense orbits. We have $T_d=\emptyset$ if and only 
if the group $F_n$ acts on $T$ with dense orbits.
This property is invariant under scaling and hence
it is defined for projective trees.

Let $T\in \partial cv(F_n)$ be a very small $F_n$-tree 
with $T_d\not=\emptyset$.
The quotient 
$T/F_n$ admits a natural pseudo-metric.
Let $\widehat{T/F_n}$ be the associated
metric space. 

Since $T$ is very small, 
by Theorem 1 of \cite{L94} the
space $\widehat{T/F_n}$ is a finite graph. Edges
correspond to orbits of the action of $F_n$ on 
$\pi_0(T-\overline{B})$ where $B\subset T$ is the set of branch
points of $T$. The graph $\widehat{T/F_n}$ defines
a graph of groups decomposition for $F_n$, with
at most cyclic edge groups.

Denote as before by ${\cal F\cal S}$ and 
${\cal C\cal S}$
the first barycentric subdivision of the free splitting graph and the
cyclic splitting graph. 
%In Section 4 we
%showed that there is an ${\rm Out}(F_n)$-equivariant
%continuous map 
%$\psi:\partial {\rm CV}(F_n)\to \partial {\cal F\cal S}\cup \Theta$ and
%$\psi_{\cal C}:\partial {\rm CV}(F_n)\to 
%\partial {\cal C\cal S}\cup \Theta$ where
%$\Theta$ is an isolated point.
The proof of the
following result uses an argument which 
was shown to me by Vincent Guirardel.

\begin{proposition}\label{graphofgroups}
Let $[T]\in \partial {\rm CV}(F_n)$ be such that
$T_d\not=\emptyset$. Let 
$(x_t)\subset \overline{cv(F_n)}^+$ be a Skora path converging to
$[T]$.
\begin{enumerate}
\item 
%$\psi_{\cal C}([T])=\Theta$.
${\rm diam}(\Upsilon_{\cal C}(x_t))<\infty$.
\item If $T$ contains an edge with trivial edge stabilizer
then 
%$\psi([T])=\Theta$.
${\rm diam}(\Upsilon(x_t))<\infty$.
\end{enumerate}
\end{proposition}
\begin{proof} Let $[T]\in \partial {\rm CV}(F_n)$ 
be such that $T_d\not=\emptyset$.
Let $(x_t)$ be an unnormalized Skora path connecting 
a point $x_0$ in a standard simplex $\Delta$ to 
a representative $T$ of $[T]$.
By this we mean that there is a folding path
$(y_t)\subset cv_0(F_n)$ 
guided by a train track map $f:x_0=y_0\to T$
such that for each $t$ we have $x_t=u_ty_t$. Here $u_t>0$ is such that
for $s<t$ there is a morphism 
\[g_{st}:x_s\to x_t\] and a 
train track map $f_t:x_t\to T$ with 
$f_s=f_t\circ g_{st}$.

Let $a>0$ be the smallest length
of an edge in $\widehat{T/F_n}$ and let $c\geq  a$ be
the volume of $\widehat{T/F_n}$. 
The volume of $x_t/F_n$ is a decreasing function in $t$ which
converges as $t\to \infty$ 
to some $b\geq c$. Thus there is a number $t_0>0$
such that for each $t\geq t_0$ 
the volume of $x_{t}/F_n$ is smaller than
$b+a/8$. 

Let $e_0$ be an edge of $T$. The length of $e_0$ is at least $a$.  
By equivariance and the
fact that for each $t$ the map $f_{t}$ is an edge isometry, 
there is a non-degenerate subarc $\rho$ of $e_0$ 
of length $a/2$ (the central subarc of length $a/2$)
such that for each $t\geq t_0$ the 
preimage of $\rho$ under $f_{t}$ consists of $k(t)\geq 1$ pairwise
disjoint subarcs of edges of $x_{t}$ (i.e. the
preimage does not contain any vertex). 
Namely, since $f_{t}$ is an edge isometry, otherwise
there are two segments in $x_t$ of length at least 
$a/4$ which are identified by the 
unnormalized Skora path connecting
$x_t$ to $T$. By equivariance, this implies that there is some
$u>t$ such that 
the volume of $x_u/F_n$ is strictly smaller than 
the volume of $x_t/F_n$ minus $a/8$. This violates the 
assumption on the volume of $x_{t}/F_n$.

We claim that the number $k(t)$ of preimages of $\rho$ 
in $x_t$ does not
depend on $t\geq t_0$. Namely, by equivariance and the 
definition of a Skora path, otherwise there
is some $t>t_0$ and a vertex of $x_t$ which is mapped
by $f_t$ into $\rho$. By the discussion in the
previous paragraph, this is impossible.

Write $k=k(t_0)$. 
(It is not hard to see that $k=1$, however
this information is not needed in the sequel).
For each $u\geq t_0$ the preimage of $\rho$ in $x_{u}$ consists
of exactly $k$ subsegments of edges of $x_u$. 
The morphism 
$g_{t_0u}$ maps the $k$ components of the preimage
of $\rho$ in $x_{t_0}$ 
isometrically onto the $k$ components of the
preimage of $\rho$ in $x_u$.
Thus by equivariance,  
there is an edge $h_{t_0}$ of $x_{t_0}/F_n$, 
and there is a subsegment $\rho_{t_0}$ of $h_{t_0}$ 
of length $a/2$ 
which is mapped by the quotient map 
$\hat g_{t_0u}:x_{t_0}/F_n\to x_u/F_n$ 
of the train track map $g_{t_0u}$ 
isometrically
onto a subsegment $\rho_u$ of an edge $h_u$ in 
$x_u/F_n$, and $\hat g_{t_0u}^{-1}(\rho_u)=\rho_{t_0}$.

By 
Lemma \ref{estimate} (see also 
Lemma A.3 of \cite{BF12}), 
this implies that
the distance in ${\cal C\cal S}$ 
between $\Upsilon_{\cal C}(x_{t_0})$ and 
$\Upsilon_{\cal C}(x_u)$ is uniformly bounded, independent of $u\geq t_0$.
As a consequence, we have
%$\psi_{\cal C}([T])=\Theta$ as claimed.
${\rm diam}(\Upsilon_{\cal C}(x_t))<\infty$ as claimed.

Now if the stabilizer of the edge $e_0$ of $T$ is trivial
then by equivariance, the same holds true for the
stabilizer of $h_u$ for all $u\geq t_0$. 
Then Lemma A.3 of \cite{BF12} shows that 
%$\psi([T[)=\Theta$.
${\rm diam}(\Upsilon(x_t))<\infty$.
\end{proof}

From Corollary \ref{hierarchy} and Proposition \ref{graphofgroups}
we obtain as an immediate consequence

\begin{corollary}\label{forallok}
If $[T]\in \partial {\rm CV}(F_n)$ is such that
$T_d\not=\emptyset$ and if $(x_t)$ is a Skora path
converging to $[T]$ then 
%$\psi_{\cal F}([T])=\Theta$. 
${\rm diam}(\Upsilon_{\cal F}(x_t))<\infty$.
%for all $\ell\leq n-1$.
\end{corollary}

For the map $\Upsilon$, the analog of the first
part of Proposition \ref{graphofgroups}
may not hold. 
To obtain information in this case as well 
we formulate first a technical observation which will be used
again in Section 6.

\begin{lemma}\label{arcstabilizer}
Let $(x_t)$ be an unnormalized Skora path connecting
$x_0\in \overline{cv_0(F_n)}^{++}$ to a tree
$[T]\in \partial{\rm CV}(F_n)$. For each $t\geq 0$ let 
$f_t:x_t\to T$ be the corresponding train track map.
If there is some 
$t>0$ such that $x_t$ has  an edge $e$ with non-trivial
stabilizer then $f_t(e)$ is
contained 
in the closure of the discrete set $T_d\subset T$. 
\end{lemma}
\begin{proof} For some 
$t>0$ let $e$ be an edge in $x_t$ with 
non-trivial stabilizer $A<F_n$. By equivariance, 
a train track map $f_t:x_t\to T$ maps $e$ isometrically onto a
segment $f(e)$ 
in $T$ stabilized by $A$.

For $s>t$ let $g_{ts}:x_t\to x_s$ be the train track map
induced by the Skora path, i.e. such that
$f_s=f_t\circ g_{ts}$ for all $s>t$.
The image $g_{ts}(e)$ of $e$ 
is isometrically embedded in the fixed point set of $A$ 
in $x_s$.
Since the tree $x_s$ is very small, this fixed point set is a 
line segment $h_s$ in $x_s$ whose length is not smaller than
the length of $e$.

We claim that the projection of $h_s$ into
$x_s/F_n$ is an immersed arc with only finitely many
double points. Namely, 
otherwise there is an element $g\in F_n-A$ such that
$gh_s\cap h_s$ contains a non-degenerate segment $a$.
Now the group $A$ fixes $h_s$, and by equivariance, the 
group $gAg^{-1}$ fixes $gh_s$. Then $a$ is fixed by
$\langle A,gAg^{-1}\rangle$. Since $a$ is non-degenerate
and $x_s$ is very small, this implies that $g$ normalizes $A$ and hence
$g\in A$. This contradiction shows the claim.

As a consequence, for 
all $s>t$ the volume of $x_s/F_n$ is not smaller than
the length of $e$. In particular, we have
$T_d\not=\emptyset$. 

We next observe that 
the segment $f(e)$ is contained in 
the closure $\overline{T_d}$ of $T_d$. 
Namely, otherwise $f(e)\cap \overline{T_d}$ is a non-trivial subsegment
of $f(e)$. Let $T^\prime$ be the tree obtained from $T$ by
equivariantly collapsing the segments in the $F_n$-orbit
of $f(e)\cap \overline{T_d}$ to points and let  
\[\rho:T\to \hat T\]
be the collapsing map.
By the construction in the
proof of Proposition \ref{onelipproj}, there is a tree
$\hat x_t$, there is an alignment preserving map 
$\alpha:x_t\to \hat x_t$ and there is a train track map
$\phi:\hat x_t\to T^\prime$ such that
\[\phi\circ \alpha=\rho\circ f.\] 
Since $f(e)\not\subset \overline{T_d}$, the segment 
$\alpha(e)\subset \hat x_t$ is non-degenerate, and 
by equivariance, it is
stabilized by $A$. The above discussion then shows that
$\phi(\xi(e))$ intersects $T^\prime_d$ in a non-degenerate segment
which is impossible. This contradiction yields the lemma.
\end{proof}

\begin{corollary}\label{plusplus}
\begin{enumerate}
\item
Let $[T]\in \partial {\rm CV}(F_n)$ be such that $T_d=\emptyset$. 
Then any Skora path $(x_t)$ converging to $[T]$ is contained in 
$\overline{cv(F_n)}^{++}$.
\item For every tree $T\in \partial {\rm CV}(F_n)$
there is a number $k([T])>0$ with the following property.
Let $(x_t)\subset cv_0(F_n)^+$ be a folding path converging to $[T]$.
Then for all sufficiently large $t$ there is a neighborhood $W$ of 
$x_t$ in $\overline{cv_0(F_n)}^+$ such that
${\rm diam}
(\Upsilon((W\cap \overline{cv_0{F_n}}^{++})\cup x_t))\leq k([T])$.
\end{enumerate}
\end{corollary}
\begin{proof} The first part of the corollary is 
immediate from Lemma \ref{arcstabilizer}.
To show the second part, 
let $T$ be such that $T_d\not=\emptyset$ and let 
$(x_t)\subset \overline{cv_0(F_n)}^+$  be any 
normalized folding path which converges to $[T]$.
For all $t$ let $f_t:x_t\to T$ be the corresponding 
train track map.

As in the beginning of this Section,
let $\widehat{T/F_n}$ be the graph of groups decomposition of $F_n$
defined by $T$. 
Lemma \ref{arcstabilizer} shows that for every $t$ and every
edge $e$ of $x_t$ with non-trivial stabilizer $A$, the segment
$f_t(e)\subset T$ is contained in 
an edge $h$ of $T_d$ with stabilizer $A$.  
The edge $h$ defines
a cyclic splitting $c$ of $F_n$ which is 
a collapse of $\widehat{T/F_n}$, and this splitting
coincides with the
splitting defined by the edge $e$ of $x_t$, i.e. which is 
determined by the tree obtained 
by equivariantly collapsing all edges in $x_t/F_n$
which are not contained in the $F_n$-orbit of $e$ to points.

As $\widehat{T/F_n}$ only has finitely many edges, 
we deduce that there 
are only finitely many possibilities for the
pure cyclic splitting defined by the union 
of all edges of $x_s$ with
non-trivial stabilizer. 
Lemma \ref{shortexists2} shows that for every $t$ 
there is a neighborhood
$W$ of $x_t$ in $\overline{cv_0(F_n)}^+$ so that
the diameter of $\Upsilon(W\cap \overline{cv_0(F_n)}^{++})$ is 
uniformly bounded, independent of $t$.
The construction in the proof of Lemma \ref{shortexists2} also
yields that 
$\Upsilon((W\cap \overline{cv_0(F_n)}^{++})\cup x_t)$ is 
uniformly bounded in diameter as well.
\end{proof}

We use Lemma \ref{arcstabilizer} to show

\begin{lemma}\label{dead}
Let $(x_t)$ be a Skora path connecting a point $x_0\in
\overline{cv_0(F_n)}^{++}$ to a simplicial 
tree $[T]$. 
Then ${\rm diam}(\Upsilon(x_t))<\infty$.
\end{lemma}
\begin{proof} Proposition \ref{graphofgroups} yields the
lemma for projective 
trees $[T]$ with at least one 
edge with trivial stabilizer. Thus  
assume that $[T]\in \partial{\rm CV}(F_n)$ is 
simplicial, and that $T/F_n$ is a graph of groups
with each edge group maximal cyclic.

Let $(x_t)$ be an unnormalized Skora path
connecting $x_0\in \overline{cv_0(F_n)}^{++}$ to 
a representative $T$ of $[T]$. Since
$(x_t)\subset \overline{cv(F_n)}^+$, 
for each $t$ we can consider the
tree $y_t$ obtained from 
$x_t$ by equivariantly collapsing all edges with
non-trivial stabilizers to points. 
By definition, we have $\Upsilon(x_t)=\Upsilon(y_t)$ for all $t$.

As $T$ is simplicial, the quotient $T/F_n$ is a finite graph of 
positive volume. The volumes of the 
quotient graphs $x_t/F_n$ are decreasing.
The volumes of the graphs $y_t/F_n$ are decreasing as well.
%The stabilizers of the edges of 
%$T$ are maximal cyclic subgroups $A_1,\dots,A_k$ of $F_n$.
We distinguish two cases.

{\sl Case 1:} The volumes $\beta_t$ of the graphs $y_t/F_n$ are 
bounded from below by a number $b>0$.

Let $e_1,\dots,e_n$ be any free basis of $F_n$.
Choose a vertex $v\in T$
and let $k=\sup\{d(v,e_iv)\mid i\}$. 
As $x_t\to T$, for sufficiently large $t$ 
the basis $e_1,\dots,e_n$ is $(k+1)$-short for $x_t$
and hence it is $(k+1)$-short for $y_t$.
Then $e_1,\dots,e_n$ is $(k+1)/a$-short for
$y_t/\beta_t\in \overline{cv_0(F_n)}^{++}$. 
By Lemma \ref{short},
this means that 
$\Upsilon(x_t)=\Upsilon(y_t/\beta_t)$ 
is contained in a uniformly 
bounded neighborhood of the corank one free factor
of $F_n$ with basis $e_1,\dots,e_{n-1}$. 
The lemma follows in this case.

{\sl Case 2:} The volumes of $y_t/F_n$ converge to zero.

Let $c>0$ be the smallest length of an edge in $T$. 
Choose $t>0$ sufficiently small that the volume 
of $y_t/F_n$ is smaller than $c/10$. 
Let $f_t:x_t\to T$ be the guiding train track map. 
Then for 
each edge $e$ of $T$, there is a subsegment $\rho$ of $e$
(the middle subsegment of length at least $c/2$)
so that $f_t^{-1}(\rho)$ consists of a 
finite number $k\geq 1$ of subarcs of edges
(compare the proof of Proposition \ref{graphofgroups}).

In fact, we have $k=1$. Namely, otherwise there is an
embedded edge path $\alpha\subset x_t$ with endpoints in the center of 
distinct edges with non-trivial stabilizer which is mapped
to a loop in $T$. As $T$ is simplicial, this means that 
$f(\alpha)$ is a compact subtree of $T$. 
If $v\in x_t$ is such that $f_t(v)$ is 
a leaf of this subtree then there are two 
proper subsegments of edges incident on $v$
which are indentified by the map $f_t$.
Then $x_t$ can be
folded in such a way that these segments are identified.
In finitely many such folding steps, we obtain a new 
tree $\hat x_t$ and train track maps
$g:x_t\to \hat x_t$, $h:\hat x_t\to T$ so that $f=h\circ g$ and that
$g(\alpha)$ consists of a single point. Then the volume of 
$\hat x_t$ is smaller than the volume of $x_t$ minus the length of an
edge with non-trivial stabilizer.
This violates the assumption on the volume of $x_t$.

Let $z_t$ the tree obtained
from $x_t$ by collapsing all edges with trivial 
stabilizers to points.  By equivariance and the above discussion, 
the graph of groups decomposition 
$\Upsilon_{\cal C}(z_t/F_n)$ coincides with 
$\Upsilon_{\cal C}(T)$. This means that there is a 
simplicial tree $S$ with an edge with non-trivial stabilizer,
and there is a one-Lipschitz alignment preserving
map $S\to T$.
The claim of the lemma now follows  
from Proposition \ref{onelipproj} and
Proposition \ref{graphofgroups}.
\end{proof}

\section{Boundaries and the boundary of Outer space}
\label{boundaries}

This section contains the main geometric results of this work.
We establish some properties 
which are valid for all three graphs ${\cal F\cal S},
{\cal C\cal S}, {\cal F\cal F}$. 
To ease notation we will prove the results we need 
only for the graph ${\cal F\cal S}$. It will be
clear that the argument is also valid for 
any graph ${\cal G}$ which can be obtained from ${\cal F\cal S}$ 
by a surjective coarsely ${\rm Out}(F_n)$-equivariant
Lipschitz map $\Psi$ so that the image under
$\Psi\circ \Upsilon$ of any folding path is a uniform
reparametrized quasi-geodesic.
Recall that 
by Corollary \ref{arbitrary}, 
the image under $\Upsilon$ of any folding path
is a reparametrized $L$-quasi-geodesic for a universal
number $L>1$.

Let $\partial {\cal F\cal S}$ be the Gromov 
boundary of ${\cal F\cal S}$ and let $\Theta$ be an
isolated point. 
For every standard simplex
$\Delta$ define a map
$\phi_\Delta:\partial {\rm CV}(F_n)\to 
\partial{\cal F\cal S}\cup \Theta$
as follows.
For a projective tree 
$[T]\in \partial{\rm CV}(F_n)$ 
choose a Skora path $(x_t)$ connecting
a point in $\Delta$ to $[T]$.  
If the reparametrized quasi-geodesic 
$t\to \Upsilon(x_t)\in {\cal F\cal S}$ 
has finite diameter then define
$\phi_\Delta([T])=\Theta$.
If the diameter of the reparametrized 
quasi-geodesic $\Upsilon(x_t)$ 
is infinite then define $\phi_\Delta([T])$ to be  
the unique
endpoint in $\partial {\cal F\cal S}$ of the path
$\Upsilon(x_t)$.

Note that a priori, the map $\phi_\Delta$ depends on choices
since a Skora path connecting 
a tree $S$ in $\Delta$ to a representative $T$ of $[T]$ depends
on the choice of a train track map $S\to T$.
The next lemma shows that the map $\phi_\Delta$
does not depend on $\Delta$ nor on any choices made.
It is related to a construction 
of Handel and Mosher \cite{HM13}.

For the purpose of the proof, 
for a number $c>0$ we say that  
a path $\alpha:[0,\infty)\to {\cal F\cal S}$
is a \emph{$c$-fellow traveler} 
of a path $\beta:[0,\infty)\to {\cal F\cal S} $ 
if  there is a nondecreasing
function $\tau:[0,\infty)\to [0,\infty)$ 
such that for all $t\geq 0$ we have
$d(\alpha(t),\beta(\tau(t)))\leq c$.
We allow that
the function $\tau$ has bounded image.

\begin{proposition}\label{coarseindependence}
The 
map $\phi_\Delta:\partial {\rm CV}(F_n)\to \partial {\cal F\cal S}\cup
\Theta$ 
does not depend on the choice of
Skora paths, nor does it depend on the 
standard simplex $\Delta$. Moreover, it is 
equivariant with respect to the action of
${\rm Out}(F_n)$ on $\partial {\rm CV}(F_n)$ and on 
$\partial {\cal F\cal S}$.
\end{proposition}
\begin{proof} If $[T]\in \partial{\rm CV}(F_n)$ is 
simplicial then Lemma \ref{dead} shows that
the diameter of the image under $\Upsilon$ of any Skora path
converging to $[T]$ is finite. Thus it suffices to show the
proposition for trees $[T]\in \partial {\rm CV}(F_n)$ with
$T_c\not=\emptyset$.

Let $T$ be such a tree. As $T$ is not simplicial, the
tree $T^\prime$ obtained from $T$ by equivariantly 
collapsing all edges of $T$ to points is not trivial, and
there is a one-Lipschitz alignment preserving map
$T\to T^\prime$. By Proposition \ref{onelipproj}, 
for every Skora path $(x_t)$ connecting a point
$x_0$ in a standard simplex $\Delta$ to $T$ there is a
Skora path $(y_t)$ connecting a point $y_0\in \Delta$ to 
$T^\prime$ so that the Hausdorff distance between
$\Upsilon(x_t)$ and $\Upsilon(y_t)$ is uniformly bounded.
Thus it suffices to consider trees $[T]$ with $T_d=\emptyset$.

Let $[T]\in \partial {\rm CV}(F_n)$ with $T_d=\emptyset$
and let 
$(x_t)_{t\geq 0}\subset \overline{cv_0(F_n)}$ 
be a Skora path guided by an optimal train track map 
$x_0\to \hat T$ where 
$x_0\in \Delta$ and where 
$\hat T$ is a representative of $[T]$ 
(see Remark \ref{optimal} for the existence of such a map). 
By Lemma \ref{arcstabilizer}, with perhaps the exception of 
its endpoint, the path $(x_t)$ is contained
in $\overline{cv_0(F_n)}^{++}$.

Choose another (not necessarily distinct)
standard simplex $\Lambda$ 
and 
%an optimal train track map 
%\[f:S\to T\]  
%where $S\in \Lambda$ and where $T$ is a representative
%of $[T]$   
a Skora path 
$(y_s)$ guided by an optimal train track map 
$y_0\to T$ where $y_0\in \Lambda$ and where 
$T$ is a representative of $[T]$
(in general we expect that $T\not=\hat T$).
%Skora path defined by $f$.
If the diameters of both $\Upsilon(x_t)$ and 
$\Upsilon(y_s)$ are finite then there is nothing to show, so by 
perhaps exchanging $(x_t)$ and $(y_s)$ we may assume that
the diameter of $\Upsilon(y_s)$ is infinite. 

%By property (5) in Definition \ref{hypout}, 
%there is a number $k>0$ and for every 
%$u>b$ and every $v>u$ there is a point
%$\hat x_v\in \overline{cv_0(F_n)}^{++}$ which is contained
%in an arbitrarily small neighborhood of 
%$x_t$ and such that $d(\Upsilon(x_t),\Upsilon(\hat x_t))\leq k$.
%In particular, we may assume that $[\hat x_t]\to [T]$.

Let $f:y_0\to T$ be an optimal train track map which
guides the folding path $(y_t)$.
We claim that 
there is a sequence $t_i\to \infty$, a sequence 
$S_i\subset \Lambda$ of points in the standard simplex $\Lambda$,
a sequence $a_{t_i}>0$ such that
$a_{t_i}x_{t_i}\to T$ in $\overline{cv(F_n)}$ and 
a sequence of train track maps 
$f_i:S_i\to a_{t_i} x_{t_i}$ 
with the following properties.
\begin{enumerate}
\item For every $U\in \Lambda$ 
the Lipschitz constant of any equivariant map 
$U\to a_{t_i}x_{t_i}$ is at least $1/2$.
\item The maps $f_i$ 
converge as $i\to \infty$ to the
map $f$ in the following sense: The graphs of 
$f_i$ in $S_i\times a_{t_i} x_{t_i}$ converge in the
equivariant Gromov Hausdorff topology to the graph of $f$.
\end{enumerate}

Let $B\subset \overline{cv(F_n)}$ be the set of all 
trees $V$ such that
the minimum of 
the smallest Lipschitz constant for equivariant maps
$U\to V$ where $U$ runs through the points in the
standard simplex $\Lambda$ equals one.
Then $B$ is a compact subset of $\overline{cv(F_n)}$, 
and $T\in B$.
For $t>0$ let 
$b_t>0$ be such that $b_tx_t\in B$. Then 
we have  
$b_tx_t\to T$
in $\overline{cv(F_n)}$ $(t\to\infty)$
with respect to  
the equivariant Gromov Hausdorff 
topology.

Let $\tilde v\in S$ be a preimage of the unique vertex 
$v$ of the graph 
$S/F_n$. Let $\{g_1,\dots,g_n\}$ 
be a free basis of $F_n$ which determines 
the simplex $\Lambda$ and such that for each $i$ the
axis of $g_i$ passes through $\tilde v$ (in particular, 
if $g_i$ has a fixed point in $S$ then this fixed point equals $\tilde v$).
Let  $K$ be the compact subtree of $T$ which is
the convex hull of the points $f(\tilde v),g_j f(\tilde v),
g_j^{-1} f(\tilde v)$ $(j=1,\dots,n)$.
Then for each $t$ there is 
a point $v_t\in b_tx_t$, and there is 
a compact subtree
$K_t$ of $b_tx_t$ which is the convex hull of the points 
$v_t,g_j v_t,g_j^{-1}v_t$  $(j=1,\dots,n)$ and such 
that the trees $K_t$ converge 
to $K$ in the usual Gromov Hausdorff topology
(see \cite{P89} for details).

Since for each $U\in \Lambda$ the 
quotient graph $U/F_n$ 
has a single vertex, for each $t$ 
there is a tree $U_t\in \Lambda$, 
there is a number $a_t\leq b_t$ and a
morphism $U_t\to a_tx_t$ 
which maps a preimage of a vertex of $U_t/F_n$ 
to $v_t$. Note that in general we have $a_t\not=b_t$, 
however by construction, $a_t/b_t\to 1$ $(t\to \infty)$.
In particular, for sufficiently large $t$ 
the trees $a_tx_t$
have property (1) above. 

Using again the fact that $U_t$ has a single vertex, 
we may in fact assume that the morphism
$U_t\to a_tx_t$ is a train track map, however it may
not be optimal. Namely, the elements $g_i$ $(1\leq i\leq n)$
generate $F_n$ and therefore if there was only
one gate for the morphism at a vertex of $U_t$ then 
the tree $a_tx_t$ can not be minimal. 
By Theorem \ref{paulin},
applied to the products 
$U_t\times a_tx_t$ (see \cite{P88} for more details), 
there is a sequence $t_i\to \infty$ so that the 
sequence 
\[f_i:S_i=U_{t_i}\to a_{t_i}x_{t_i}\] 
of  train track maps 
has property (2) above as well. 

The subspace $\Sigma\subset \overline{cv(F_n)}$ of all
trees $x$ for which the minimum of the 
smallest Lipschitz constant
for equivariant maps from trees $U\in \Lambda$ to $x$ is 
contained in the interval $[1/2,1]$ is compact, and it contains $B$.
For each $i$ 
connect $S_i$ to $a_{t_i}x_{t_i}$  
by a Skora path 
$(y_s^i)$ guided by the train track map $f_i$. For sufficiently large $i$ 
the corresponding unnormalized Skora paths are entirely contained in 
$\Sigma$. 

Normalized 
Skora paths are geodesics for the 
one-sided Lipschitz metric on the locally compact
space $\overline{cv_0(F_n)}^s$ \cite{FM11} (note however that
the one-sided Lipschitz distance between an ordered pair of points  
$(x,y)\in \overline{cv_0(F_n)}^2$ may be infinite).
Thus we can apply
to these paths the Arzela Ascoli theorem.
As a consequence, up to passing to a subsequence, the paths
$(y_s^i)$ converge as $i\to \infty$ to a Skora path $(y_s)$. 
By the discussion in the previous paragraph and construction, this path
is guided by the train track map $f$, and it connects $S$ to $[T]$. 
The image under the map $\Upsilon$ of the family of paths 
$(y_s^i),(y_s)$ is 
a family of reparametrized $L$- quasi-geodesics 
in the hyperbolic graph ${\cal F\cal S}$. 
Moreover, since $T_d=\emptyset$, 
Corollary \ref{plusplus} shows that 
the path 
$(y_s)$ is contained in 
$\overline{cv_0(F_n)}^{++}$, and the same holds true
for each of the paths $(y_s^i)$.

The diameter 
of $\Upsilon(\Lambda)\subset {\cal F\cal S}$ 
is bounded independent of the standard simplex $\Lambda$. 
Let $b\geq 0$ be such that $\Upsilon(x_b)$ is a coarsely well 
defined shortest distance projection of 
$\Upsilon(\Lambda)$
into the reparametrized $L$-quasi-geodesic 
$\Upsilon(x_t)$.

By hyperbolicity of ${\cal F\cal S}$,
for every $u>b$ and every $v>u$, the image under $\Upsilon$ of 
any Skora path connecting a point in $\Lambda$ to $x_v$ 
passes through a uniformly bounded
neighborhood of $\Upsilon(x_u)$.  
Thus for all $u>b$ and all 
$i$ such that
$t_i>u$, the reparametrized $L$-quasi-geodesic 
$\Upsilon(y_s^i)$ in ${\cal F\cal S}$ 
passes through 
a uniformly bounded neighborhood of $\Upsilon(x_u)$. 

We claim that the same holds true for 
the reparametrized 
$L$-quasi-geodesic $\Upsilon(y_s)$.
To this end recall that we assumed that the diameter of 
$\Upsilon(y_s)$ is infinite.
%Let $r\geq 0$ be the distance in ${\cal H\cal G}$ between
%$\Upsilon(\Lambda)$ and $\Upsilon(x_u)$.
Let $\tau>0$ be arbitrary. 
%a sufficiently large number so that
%the distance between $\Upsilon(\Lambda)$ and
%$\Upsilon(y_\tau)$ is at least $2r$. 
%By property (5) in Definition \ref{hypout} and 
Since $(y_s)\subset \overline{cv_0(F_n)}^{++}$, 
Corollary \ref{control} shows that 
for all sufficiently large $i$ the reparametrized
quasi-geodesic $\Upsilon(y_s^i)$ connecting 
a point in $\Upsilon(\Lambda)$ to $\Upsilon(x_{t_i})$ 
passes through the $L$-neighborhood 
of $\Upsilon(y_\tau)$.

On the other hand, it also passes through a uniformly bounded
neighborhood of $\Upsilon(x_u)$. 
Since the diameter of $\Upsilon(y_s)$ is infinite, for
sufficiently large $\tau$ 
the point $\Upsilon(x_u)$ coarsely lies between 
$\Upsilon(y_0)$ and $\Upsilon(y_\tau)$. 
This shows that 
the reparametrized quasi-geodesic 
$\Upsilon(y_s)$ passes through a uniformly bounded
neighborhood of $\Upsilon(x_u)$ as claimed above.

We use this fact to show that the diameter of $\Upsilon(x_t)$ is 
infinite. Namely, assume to the contrary
that the diameter of $\Upsilon(x_t)$ is finite. 
Then by construction, for each
$i$ the diameter of $\Upsilon(y_s^i)$ is bounded from above by
a number $C>0$ not depending on $i$. 
Choose a number $t>0$ so that the distance
between $\Upsilon(\Lambda)$ and 
$\Upsilon(y_t)$ is at least $C+10m$ where 
$m>0$ is as in Corollary \ref{control}.
For sufficiently
large $i$ the path $\Upsilon(y_s^i)$ passes through the 
$m$-neighborhood of $\Upsilon(y_t)$ which is a contradiction.

By symmetry, we conclude that the diameter of 
$\Upsilon(x_t)$ is 
infinite if and only if the diameter of
$\Upsilon(y_t)$ is infinite. Moreover, if this holds true then 
there is a number $p>0$ only depending on 
$\Lambda,\Delta$ such that $\Upsilon(y_s)$ is a $p$-fellow 
traveller of $\Upsilon(x_t)$. As a consequence, 
the map $\phi_\Delta$  indeed 
does not depend on the choice of $\Delta$ or 
on the choice of Skora paths.
This then implies that  the map $\phi_\Delta$ 
is moreover ${\rm Out}(F_n)$-equivariant.
\end{proof}

\begin{remark}\label{hierarchy2}
The proof of Proposition \ref{coarseindependence} 
shows more generally the following.
Let $\Delta$ be a standard simplex and let 
$(x_t)\subset \overline{cv_0(F_n)}^{+}$ be any normalized folding path.
If $[x_t]\to [T]\in \partial {\rm CV}(F_n)$ in the
equivariant Gromov Hausdorff topology then the reparametrized
quasi-geodesic 
$\Upsilon(x_t)$ in ${\cal F\cal S}$ is of infinite diameter if
and only if 
$\phi_\Delta([T])\in \partial {\cal F\cal S}$, and  in this case we
have $\Upsilon(x_t)\to \phi_\Delta([T])$ in ${\cal F\cal S}\cup
\partial {\cal F\cal S}$. 
\end{remark}

\begin{remark}\label{theta}
The proof of Proposition \ref{coarseindependence} also 
gives some information on Skora paths whose
images under $\Upsilon$ have finite diameter. Namely, 
if $\phi_\Delta([T])=\Theta$ then there is a 
point $\xi\in {\cal F\cal S}$ and
there is a number $r>0$ so that for every Skora path $(x_t)\subset 
\overline{cv_0(F_n)}^{+}$
with $[x_t]\to [T]$ in 
${\rm CV}(F_n)\cup \partial {\rm CV}(F_n)$ and for all large enough $t$ the
point $\Upsilon(x_t)$ is contained in the $r$-neighborhood of $\xi$ in 
${\cal F\cal S}$.
\end{remark}

By Proposition \ref{coarseindependence}, we can define
\[\psi=\phi_\Delta:\partial {\rm CV}(F_n)\to \partial {\cal F\cal S}\cup \Theta\]
for some (and hence every) standard simplex $\Delta$. 
Similarly we define 
\[\psi_{\cal C}:\partial{\rm CV}(F_n)\to 
\partial{\cal C\cal S}\cup \Theta,\,
\psi_{\cal F}:\partial{\rm CV}(F_n)\to 
\partial {\cal F\cal F}\cup \Theta.\] The maps
$\psi,\psi_{\cal C },\psi_{\cal F}$ do not depend on choices, 
and they are ${\rm Out}(F_n)$-equivariant.

Our next goal is to show 
that the map $\psi$ is onto $\partial {\cal F\cal S}$.
As a main preparation we establish the following lemma which 
is motivated by the work of Klarreich
(Proposition 6.4 of \cite{K99}). For convenience of notation we extend the 
map $\psi$ to $\overline{{\rm CV}(F_n)}$ by defining 
$\psi([T])=\Theta$ for every simplicial free 
projective $F_n$-tree 
$[T]\in {\rm CV}(F_n)$.

%For $\epsilon >0$ let ${\rm Thick}_\epsilon(F_n)\subset cv_0(F_n)$ be the
%set of all free simplicial $F_n$-trees with quotient of volume one and
%such that the length of every loop in $T/F_n$ is at least $\epsilon$.
%For sufficiently small $\epsilon >0$, the space 
%${\rm Thick}_\epsilon(F_n)$ is path connected, moreover it is
%${\rm Out}(F_n)$-invariant and ${\rm Out}(F_n)$ acts properly 
%and cocompactly.

\begin{lemma}\label{sequence}
Let  $(T_i)\subset 
cv_0(F_n)$ 
%{\rm Thick}_\epsilon(F_n)$ 
be a sequence whose projectivization
converges in ${\rm CV}(F_n)\cup \partial {\rm CV}(F_n)$
to a point $[T]\in \overline{ {\rm CV}(F_n)}$.
If $\psi([T])=\Theta$
then $(\Upsilon(T_i))$ does not converge to a point
in $\partial {\cal F\cal S}$. 
\end{lemma}
\begin{proof} 
We follow the reasoning in the proof of Proposition 6.4 of 
\cite{K99}. The case that $[T]\in {\rm CV}(F_n)$ is immediate from
Corollary \ref{control}, so 
let $([T_i])\subset {\rm CV}(F_n)$ 
be a sequence 
which converges to a point $[T]\in 
\partial{\rm CV}(F_n)$ with $\psi([T])=\Theta$.

%Assume without loss of generality that $T_i$ is
%the normalized Cayley tree of some free basis of $F_n$ for all $i$.
%This is possible since every point in $cv_0(F_n)$ 
%is the image of a normalized Cayley tree under 
%a map of uniformly bounded Lipschitz constant and
%hence also of uniformly bounded 
%no-gap distance.

%
%there is a measured lamination $\mu$ supported in 
%a proper free factor $H$ of $F_n$,
%and there is an alignment preserving map
 %$\phi:T\to S$ where 
%$\langle S,\mu\rangle =0$. 
%This is the case if $T$ splits as a graph of 
%actions or if $T$ is indecomposable with large
%point stabilizer.

For each $i$ let $T_i\in cv_0(F_n)$ be a representative of 
$[T_i]$. 
We argue by contradiction and we 
assume that the sequence $(\Upsilon(T_i))$ converges
to a point in the Gromov boundary of ${\cal F\cal S}$.

For each $i$ 
there is a standard simplex
$\Delta_i\subset \overline{cv_0(F_n)}^{++}$  
so that the distance in ${\cal F\cal S}$ between
$\Upsilon(T_i)$ and $\Upsilon(\Delta_i)$ is at most $k$
where $k>0$ is a universal constant. Such a simplex
can be constructed as follows. Collapse 
all edges of $T_i$ with non-trivial stabilizer to a point.
Let $S$ be the resulting tree. 
By definition, we have 
$\Upsilon(T_i)=\Upsilon(S)$. 

Collapse all but one edge in $S/F_n$ to a point. 
If the resulting one-edge graph of groups 
decomposition $G$ of $F_n$ is 
a two-vertex decomposition then the vertex groups are
free factors of $F_n$. Replace each such vertex by a
marked rose representing a basis of the free factor.
The resulting graph of groups decomposition collapses
to both $G$ and a decomposition defined by a rose of rank $n$.
This rose determines a standard simplex as required.
The case that $G$ is a one-loop decomposition 
follows in the same way.

Let $\Sigma_i\subset \overline{cv(F_n)}$ be the set 
of all trees for which the minimum of 
the optimal Lipschitz constants for equivariant maps from 
points in $\Delta_i$ equals one.

For $j>i$ let $T_j^i=b_jT_j\in \Sigma_i$ be a representative
of $[T_j]$. Choose an optimal 
train track map $f_j:r_0^{ij}\to T_j^i$ 
where 
$r_0^{ij}\in \Delta_i$ and let 
$(r^{ij}_t)$ be a Skora path connecting $r_0^{ij}$ to $T_j^i$ which is
guided by $f_j$. As the endpoints of the paths are contained
in $\overline{cv(F_n)}^{++}$, the same holds true for the
entire paths. 
 
The initial
points of the paths $(r_t^{ij})$  $(j>i)$ 
are contained in the compact subset $\Delta_i$ 
of $\overline{cv_0(F_n)}^{++}$. Hence by the Arzela Ascoli theorem, 
applied to the geodesics $r_t^{ij}$ for the one-sided Lipschitz metric, 
up to passing to a subsequence
we may assume that the paths $(r_t^{ij})$ converge
as $j\to \infty$  
locally uniformly 
to a Skora path $t\to r_t^i$ issuing from 
a point $r_0^i \in \Delta_i$. By the reasoning 
in the proof of Lemma \ref{morphismsex}, since $[T_j]\to [T]$ 
we may assume that 
the path $(r_t^i)$ is defined by a train track map 
$r_0^i\to T$ where $T$ is a representative of $[T]$ and hence
it connects $r_0^i$ to $[T]$.

By 
%Property (5) in Definition \ref{hypout} and 
Remark \ref{theta}, 
there is a point $\eta\in {\cal F\cal S}$, 
and there is a number $b>0$ only depending on $T$ 
with the following property.
For any Skora path 
$(x_t)\subset \overline{cv_0(F_n)}^+$ guided 
by some train track map $x_0\to T$ where $T$ is 
a representative of 
$[T]$ and for all large enough $t$,
the point $\Upsilon(x_t)$ is contained in the 
$b$-neighborhood of $\eta$ in ${\cal F\cal S}$.
In particular, for large enough $t$ the point 
$\Upsilon(r_t^i)$ 
is contained in the $b$-neighborhood 
of $\eta$.
Fix a number $t_0$ with this property.

By the second part of Corollary \ref{plusplus}, 
%By property (5) in Definition \ref{hypout}, 
for large enough $j$
the distance between $\Upsilon(r^{ij}_{t_0})$ and
$\Upsilon(r_{t_0}^i)$ is bounded by $k([T])$.
As a consequence, 
if $(\mid )_\xi$ is the Gromov product
based at $\eta\in {\cal H\cal G}$ then 
we have 
\[(\Upsilon(T_i)\mid \Upsilon(T_j))_\eta\leq q\] 
for infinitely many $i,j$
where $q>0$ is a constant depending on $T$ but not on $i,j$.
This is a contradiction to the assumption that
the sequence
$(\Upsilon(T_i))$ converges to a point in the Gromov boundary of 
${\cal F\cal S}.$
%
%
%In the case that $T$ is dense and that there is no measured
%lamination supported in the zero lamination
%of $T$ which intersects a proper free factor, 
%using the fact that $T\not\in {\cal F\cal T}$,
%by Proposition \ref{projection} we may assume that
%$T$ admits an alignment preserving map onto
%a tree $S$ which splits as a graph of actions.
%By modifying $T_i$ within the simplex $\Delta_i$ 
%defined as above we may assume that
%$T_i$ admits a morphism $T_i\to T$. Let $S_i$ be a point in the
%simplex which admits a morphism $S_i\to S$. Then the distance
%in ${\cal F\cal F}$ between $\Upsilon(T_i)$ 
%and $\Upsilon(S_i)$ is 
%uniformly bounded, moreover the sequence
%$[S_i]\subset \overline{{\rm CV}(F_n)}$ 
%converges to $[S]$. The above discussion
%shows that $\Upsilon(S_i)$ does not converge to a point in the Gromov
%boundary of ${\cal F\cal F}$ and hence the 
%same is true for $\Upsilon(T_i)$.
\end{proof}

\begin{remark} Lemma \ref{sequence} does \emph{not} state
that if $(x_t)\subset cv_0(F_n)$ is any sequence
converging to a point $[T]\in \partial {\rm CV}(F_n)$ with
$\psi([T])=\Theta$ then $\Upsilon(x_t)$ is bounded.
In fact, we believe that there should be sequences for which
this is not true.
\end{remark}

Next we have

\begin{lemma}\label{sequence3}
If $[T_i]\subset {\rm CV}(F_n)$ is 
any sequence which converges to $[T]\in \partial {\rm CV}(F_n)$ 
with $\psi([T])\in \partial {\cal F\cal S}$ 
then $\Upsilon(T_i)$ converges
to $\psi([T])$ in ${\cal F\cal S}\cup \partial {\cal F\cal S}$.
\end{lemma}
\begin{proof} Let $\Delta$ be a standard simplex. For each
$i$ let $(r_t^i)$ be a 
Skora path connecting a point
in $\Delta$ to a tree $T_i\in cv_0(F_n)$ representing the class
$[T_i]$. By the Arzela Ascoli theorem, up to passing
to a subsequence we may assume that the paths
$(r_t^i)$ converge as $i\to \infty$ to a Skora path $(x_t)$ 
connecting a point in $\Delta$ to a representative of $[T]$.

Since the paths $\Upsilon(r_t^i)$ 
are uniform reparametrized quasi-geodesics in ${\cal F\cal S}$ and 
since the path $(x_t)$ connects a point in 
$\Delta$ to $[T]$, the path $\Upsilon(x_t)$ 
is a (reparametrized) 
quasi-geodesic ray connecting $\Upsilon(x_0)$ to 
$\psi([T])$. The lemma now follows from the second part
of Corollary \ref{plusplus} 
and hyperbolicity
of ${\cal F\cal S}$.
\end{proof}

As a corollary we obtain

\begin{corollary}\label{surjective}
$\psi(\partial {\rm CV}(F_n))\supset \partial {\cal F\cal S}$.
\end{corollary}
\begin{proof} Since 
the map $\Upsilon:cv_0(F_n)\to {\cal F\cal S}$ is coarsely surjective,
for every $\eta\in \partial {\cal F\cal S}$ there is a
sequence $(T_i)\subset cv_0(F_n)$ so that 
$\Upsilon(T_i)\to \eta$.

By compactness of $\overline{{\rm CV}(F_n)}$, by passing to a
subsequence we may assume that 
$[T_i]\to [T]$ for some $[T]\in \overline{{\rm CV}(F_n)}$.
By Lemma \ref{sequence}, if
$\psi([T])=\Theta$ then
the sequence $\Upsilon(T_i)$ does not converge to $\eta$. 
Thus $\psi([T])=\chi\in \partial {\cal F\cal S}$, 
and it follows from Lemma \ref{sequence3}
that $\chi=\xi$.
\end{proof}

Let now 
\[{\cal F\cal T}=\psi^{-1}(\partial {\cal F\cal S})\subset
\partial {\rm CV}(F_n).\] 
By ${\rm Out}(F_n)$-equivariance of the map $\psi$, the
set ${\cal F\cal T}$ is an ${\rm Out}(F_n)$-invariant subset of 
$\partial {\rm CV}(F_n)$.
We showed above that the restriction of 
$\psi$ maps ${\cal F\cal T}$ onto $\partial {\cal F\cal S}$.
We also have

\begin{lemma}\label{cont}
The restriction of the map $\psi$ 
to ${\cal F\cal T}$ is continuous and closed.
\end{lemma}
\begin{proof}
To show that the restriction to ${\cal F\cal S}$ of the map $\psi$  
is continuous, note that
if $[T_i]\to [T]\in {\cal F\cal T}$ then there
is a sequence $(r_t^i)$ 
of Skora paths starting at a point in a standard
simplex $\Delta$ so that $r^i_t\to [T_i]$ $(t\to \infty)$
and $r^i \to r$ locally uniformly where $r$ is a Skora 
path connecting $\Delta$ to $[T]$. 
The images under $\Upsilon$ of these paths
are reparametrized quasi-geodesics in 
${\cal F\cal S}$.  
Continuity now follows from 
hyperbolicity of ${\cal F\cal S}$.

Since the Gromov topology on $\partial {\cal F\cal S}$ is 
metrizable, 
to show that the restriction of $\psi$ is closed
it suffices to show the following. 
If $[T_i]\subset {\cal F\cal T}$ is any sequence
and if $\psi([T_i])\to \eta\in \partial{\cal F\cal S}$ then
up to passing to a subsequence, we have
$[T_i]\to [U]$ where $\psi([U])=\eta$.

Assume to the contrary that this is not the case. 
By compactness of $\partial{\rm CV}(F_n)$ 
there is then 
a sequence $[T_i]\subset {\cal F\cal T}$ 
so that $\psi ([T_i])\to \eta$ and such that
$[T_i]\to [S]\in \partial{\rm CV}(F_n)$ where
either $[S]\not\in {\cal F\cal T}$ or 
$\psi([S])\not=\eta$.

Now if $[S]\in {\cal F\cal T}$ then 
by continuity of the map $\psi$, we have 
\[\psi([S])=\lim_{i\to \infty}\psi([T_i])=\eta.\] 
Since we assumed that 
$\psi([S])\not=\eta$ this is 
impossible. 

Thus $[S]\not\in {\cal F\cal T}$. 
However, it follows from the hypotheses that
there is a sequence
$(S_i)\subset cv_0(F_n)$ with $[S_i]\to [S]$ and 
such that $\Upsilon(S_i)\to \eta$
in ${\cal F\cal S}\cup \partial {\cal F\cal S}$.
Here $S_i$ can be chosen to be a point on the Skora
path $r_t^i$ which is such that $\Upsilon(S_i)$ is
sufficiently close to $\Upsilon(T_i)$ in 
${\cal F\cal S}\cup \partial {\cal F\cal S}$.
This violates
Lemma \ref{sequence}.
The corollary follows.
\end{proof}

\begin{remark}\label{notspecific}
The results in this section do not use any specific property of 
Outer space and can easily be formulated in a more abstract
setting which is valid for example for Teichm\"uller space and
the curve graph as in \cite{K99} and for various disc graphs in a 
handlebody \cite{H11}.
\end{remark}

Recall the definition of the 
hyperbolic ${\rm Out}(F_n)$-graphs 
$({\cal C\cal S},\Upsilon_{\cal C})$
and $({\cal F\cal F},\Upsilon_{\cal F})$ and of the maps
%where
%we put ${\cal F\cal F}_{n-1}={\cal F\cal S}$ for convenience
%of notation. 
\begin{equation}
\psi_{\cal C}:
\partial {\rm CV}(F_n)\to \partial{\cal C\cal S}\cup
\Theta
\text{ and } \psi_{\cal F}:\partial{\rm CV}(F_n)
\to \partial {\cal F\cal F}\cup \Theta.\notag\end{equation}
 %Recall that
% for every $\ell\leq n-2$ 
 %the vertex set of the graph 
%${\cal F\cal F}_1$ is the set of conjugacy classes of free
%factors of $F_n$ of rank $n-1$ which is 
%a coarsely dense subset of the vertex set of ${\cal C\cal S}$.
%Moreover, for $1\leq \ell\leq m\leq n-1$
Since for any tree $T\in \overline{cv_0(F_n)}^{+}$ 
the distance in ${\cal C\cal S}$ between
$\Upsilon(T)\in {\cal F\cal S}\subset {\cal C\cal S}$ 
and $\Upsilon_{\cal C}(T)$ is at most one and since 
$\Upsilon_{\cal F}=\Omega\circ \Upsilon$ for 
a coarsely Lipschitz map $\Omega:{\cal C\cal S}\to
{\cal F\cal F}$, 
we obtain 
as an immediate consequence

\begin{corollary}\label{hierarchy}
\begin{enumerate}
\item If $\psi([T])=\Theta$ then $\psi_{\cal C}([T])=\Theta$.
\item 
If
% $1\leq \ell\leq m\leq n-1$ and 
$\psi_{\cal C}([T])=\Theta$ then
$\psi_{\cal F}([T])=\Theta$.
\end{enumerate}
\end{corollary}

As a consequence of the results in this section, 
%for each $\ell$
the boundaries of ${\cal F\cal S},$  
${\cal C\cal S}$ and ${\cal F\cal F}$ 
can be identified as a set with
the quotient of an
${\rm Out}(F_n)$-invariant subspace of $\partial {\rm CV}(F_n)$ by an 
equivalence relation defined by the map $\psi$ and 
the map $\psi_{\cal C}$ and 
the map $\psi_{\cal F}$.

%{\bf Remark:} One needs to show that the
%image of the path coarsely only depends on its endpoints.

%\begin{remark}\label{hierarchy3}
%The graphs ${\cal F\cal F}_\ell$ $(1\leq \ell\leq n-1)$ 
%are not the only good hyperbolic ${\rm Out}(F_n)$-graphs.
%In a forthcoming work we will introduce another such graph 
%which admits an equivariant coarse Lipschitz map onto 
%the free splitting
%graph and which gives more information on the geometry of ${\rm Out}(F_n)$. 
%\end{remark}

We complete the section with some first easy information
on the fibres of $\psi$.
From now on 
we only consider trees 
$T\in \partial cv(F_n)$ (or projective trees
$[T]\in \partial {\rm CV}(F_n)$)
with dense orbits. 
For simplicity we call such a tree 
\emph{dense}.

The following definition is due to Paulin (see \cite{G00}).

\begin{definition}\label{lengthmeasure}
A \emph{length measure} $\mu$ on $T$ is an
$F_n$-invariant collection
\[\mu=\{\mu_I\}_{I\subset T}\] of locally finite Borel measures
on the finite arcs $I\subset T$; it is required that
for $J\subset I$ we have $\mu_J=(\mu_I)\vert J$.
\end{definition}
The Lebesgue measure $\lambda$ defining the metric on 
$T$ is an example of a length measure on $T$ with full
support.

Denote by $M_0(T)$ the set of 
all non-atomic length measures on $T$.
By Corollary 5.4 of \cite{G00}, 
$M_0(T)$ is 
a finite dimensional convex set which is projectively
compact. Up to homothety, there are at most
$3n-4$ non-atomic ergodic length measures. 
In particular, $M_0(T)$ is a cone over a compact convex 
polyhedron with finitely many vertices. 
Each non-atomic length measure
$\mu\in M_0(T)$ defines 
an $F_n$-tree $T_\mu\in \partial cv(F_n)$ as follows
\cite{G00}. Define a pseudo-metric $d_\mu$ on $T$ by
$d_\mu(x,y)=\mu([x,y])$. Making this pseudo-metric
Hausdorff gives an $\mathbb{R}$-tree $T_\mu$.  
%If the support of $\mu$ is all of $T$ then 
%$T_\mu$ is equivariantly homeomorphic to 
%$T$.

\begin{corollary}\label{lengthmeasurepro}
Let $T\in \partial cv(F_n)$ be a tree with dense orbits; then 
$\psi([T_\mu])=\psi([T])$ for all 
$\mu\in M_0(T)$.
%\item If $T$ admits an invariant atomic length measure or
%two measures $\mu,\nu\in M_0(T)$ with distinct
%non-degenerate support then $\psi([T])=\Theta$.
%\end{enumerate} 
\end{corollary}
\begin{proof} Let $\nu$ be a length measure on $T$  
which is contained in the interior of the
convex set $M_0(T)$.
Let $\zeta$ be a point in 
$M_0(T)$ which projects to a vertex in the
projectivization of $M_0(T)$; this is an ergodic measure in 
$M_0(T)$.  
Up to rescaling, there is a
one-Lipschitz alignment preserving 
map $T_\nu\to T_\zeta$. 
Proposition \ref{onelipproj} shows that 
$\psi([T_\nu])=\psi([T_\zeta])$.

Now if $\xi\in M_0(T)$ is arbitrary 
then there is an ergodic
measure $\beta\in M_0(T)$, and there is   
a one-Lipschitz alignment preserving
map $T_\xi\to T_\beta$. 
Using once more Proposition \ref{onelipproj},  we deduce that
$\psi([T_\xi])=\psi([T_\beta])=\psi([T_\nu])$.
Since the Lebesgue measure on $T$ defines a point 
in $M_0(T)$ this shows the corollary.
\end{proof}

%By Lemma 10.2 of \cite{R10}, if $T$ admits an
%invariant atomic length measure then $T$ splits as a
%graph of actions. The same holds true if 
%$T$ admits two invariant non-atomic measures 
%whose supports are non-degenerate and distinct
%(Lemma 12.1 of \cite{R10}). 
%Thus the second part of the corollary is a consequence of 
%Proposition \ref{graphofaction2}.
%\end{proof}

%In particular, for the length structure induced by 
%$\mu,\mu^\prime$ the map $f$ is one-Lipschitz. 
We use Proposition \ref{onelipproj} 
and Corollary \ref{lengthmeasurepro}ï¿œ to show

\begin{corollary}\label{projectionalign}
%\begin{enumerate}
%\item 
Let $T,T^\prime\in \partial cv(F_n)$ 
be trees with dense orbits 
and let 
$\rho:T\to T^\prime$ be alignment preserving. 
Then $\psi([T])= \psi([T^\prime])$.
%\item If the tree $T\in \partial cv(F_n)$ 
%admits an alignment preserving map onto  
%a tree which splits
%as a graph of actions then $\psi([T])=\Theta$.
%\end{enumerate}
\end{corollary}
\begin{proof} Let $T,T^\prime\in \partial cv(F_n)$
be trees with dense orbits and assume  
that there is an alignment preserving map $\rho:T\to T^\prime$.

As is shown in \cite{G00}, if 
$\mu^\prime$ is a non-atomic length measure on 
$T^\prime$ then there is a length measure $\mu$ on 
$T$ such that $\rho_*\mu=\mu^\prime$.
This means that for every segment $I\subset T$ we have 
$\mu(I)=\mu^\prime(\rho I)$. 
As a consequence, 
there is a one-Lipschitz alignment preserving map 
$\hat \rho:T_\mu\to T_{\mu^\prime}$.
Proposition \ref{onelipproj} shows that
$\psi([T_\mu])= \psi([T_{\mu^\prime}])$.
The 
corollary now follows from Corollary \ref{lengthmeasurepro}.
\end{proof}

\section{Indecomposable trees}\label{indecom}

%By Corollary \ref{surjective}, 
%Lemma \ref{cont},
%Corollary \ref{projectionalign} 
%and Corollary \ref{indecomposable3}, there is an 
%${\rm Out}(F_n)$-invariant subspace ${\cal U}$ of
%${\cal S\cal T}$ and a continuous closed  
%surjective ${\rm Out}(F_n)$-equivariant map 
%${\cal U}\to \partial {\cal C\cal S}$.
%
The goal of this section is to obtain some information on the
structure of indecomposable trees 
which is needed for the proof of the theorems from the introduction.

Let $\Delta\subset\partial F_n\times \partial F_n$ be the diagonal.
The \emph{zero lamination} $L^2(T)$ of 
an $\mathbb{R}$-tree $T$ with an isometric
action of $F_n$ is the 
closed $F_n$-invariant subset of 
$\partial F_n\times \partial F_n-\Delta$ which
is the set of all 
accumulation points of pairs of fixed points
of any family of conjugacy
classes with translation length on $T$ that 
tends to $0$. The zero lamination
of a tree $T\in cv(F_n)$ is empty. For 
$T\in \partial cv(F_n)$,  it only depends on the 
projective
class $[T]\in \partial {\rm CV}(F_n)$ of $T$.

In \cite{CHL07} the following topological interpretation of 
the zero lamination of an $\mathbb{R}$-tree $T\in \partial cv(F_n)$ is
given. 

The union 
$\hat T=\overline{T}\cup \partial T$ of the metric completion
$\overline{T}$ of $T$ with the Gromov boundary 
$\partial T$ of $T$ can be equipped with
an \emph{observer's   topology}. With respect to this
topology, $\hat T$ is a compact $F_n$-space, and the
inclusion $T\to \hat T$ 
is continuous \cite{CHL07}. Isometries of $T$ induce
homeomorphisms of $\hat T$ (see p.903 of \cite{CHL09}).
 
There is an explicit description of $\hat T$ 
as follows. Namely, 
let again $L^2(T)$ be the zero lamination of $T$.
There is an $F_n$-equivariant \cite{LL03}
continuous (Proposition 2.3 of \cite{CHL07}) map
\[Q:\partial F_n\to \hat T\]
such that 
$L^2(T)=\{(\xi,\zeta)\mid Q(\xi)=Q(\zeta)\}.$
This map determines
an equivariant homeomorphism 
\[\partial F_n/L^2(T)\to \hat T\]
(Corollary 2.6 of \cite{CHL07}), i.e. the tree
$\hat T$ is the quotient of 
$\partial F_n$ by the equivalence relation
obtained by 
identifying all points $\xi,\xi^\prime\in \partial F_n$ 
with $Q(\xi)=Q(\xi^\prime)$, and a pair
of points $(\xi,\xi^\prime)$ is 
identified if and only if it is
contained in $L^2(T)$.

The \emph{limit set} of $\overline{T}$ is defined to be
the set
\[\Omega= Q^2(L^2(T))\subset \overline{T}\subset \hat T\]
where the notation $Q^2$ refers to applying the map $Q$ to 
a pair of points with the same image.

A finitely generated subgroup $H<F_n$ 
is free, and its 
boundary $\partial H$ is naturally
a closed subset of the boundary $\partial F_n$ of 
$F_n$. If $H$ fixes a point in 
an $\mathbb{R}$-tree $T$ then the  
set $\partial H\times \partial H-\Delta$ 
of pairs of distinct points in $\partial H$, 
viewed as a subset of $\partial F_n\times \partial F_n-\Delta$, 
is contained in $L^2(T)$.
We say that a leaf $\ell\in L^2(T)$
is \emph{carried} by a finitely generated subgroup 
$H$ of $F_n$ if it is a point in $\partial H\times \partial H-\Delta$.

Define a leaf
$\ell\in L^2(T)$ to be \emph{regular} if 
$\ell$ is not carried by the stabilizer of a point in 
$\overline{T}$ (this makes sense since point stabilizers
of $\overline{T}$ are finitely generated) 
and if moreover there exists
a sequence $\ell_n\subset L^2(T)$ of leaves 
converging to $\ell$ such that the $x_n=Q^2(\ell_n)$ are
distinct. The set of regular leaves 
of $L^2(T)$ is the \emph{regular sublamination}
$L_r(T)$. It is $F_n$-invariant but in general not closed.

The space ${\cal M\cal L}$ of \emph{measured laminations}
for $F_n$
is a closed subspace of the
space of all locally finite $F_n$-invariant Borel measures on 
$\partial F_n\times \partial F_n-\Delta$,
equipped with the weak$^*$-topology.
Dirac measures on pairs of fixed points of all elements
in some primitive conjugacy class of $F_n$ are dense in 
${\cal M\cal L}$ \cite{Ma95}.
There is a continuous length pairing \cite{KL09}
\[\langle ,\rangle:\overline{cv(F_n)}\times {\cal M\cal L}\to
[0,\infty).\]
For $T\in \overline{cv(F_n)}$ and $\mu\in {\cal M\cal L}$ we have
$\langle T,\mu\rangle =0$ if and only if $\mu$ is supported
in $L^2(T)$ \cite{KL09}. 

The following is the main result of this section.

\begin{proposition}\label{fillisfull}
Let $[T]\in \partial {\rm CV}(F_n)$ be indecomposable
and let $\mu$ be a 
measured lamination with 
$\langle T,\mu\rangle =0$ and $\mu(L^2(T)-L_r(T))=0$.
If $[S]\in \partial {\rm CV}(F_n)$ is indecomposable and if 
$\langle S,\mu\rangle=0$ then 
$\hat S$ and $\hat T$ equipped with
the observer's topology are
$F_n$-equivariantly homeomorphic.
\end{proposition}

Denote by ${\cal I\cal T}\subset \partial {\rm CV}(F_n)$ 
the ${\rm Out}(F_n)$-invariant subset of 
indecomposable projective trees. 
For an indecomposable projective tree $[T]\in 
{\cal I\cal T}\subset \partial{\rm CV}(F_n)$ define
${\cal H}([T])\subset {\cal I\cal T}$ 
to be the set of all indecomposable 
projective trees $[S]$ so that
$\hat S$ is $F_n$-equivariantly homeomorphic to $\hat T$.

\begin{corollary}\label{sometimesclosed}
Let $[T]\in {\cal I\cal T}$ be such that there is a measured
lamination $\mu$ with $\langle T,\mu\rangle =0$ and 
$\mu(L^2(T)-L_r(T))=0$. Then ${\cal H}([T])$ is a closed
subset of ${\cal I\cal T}$.
\end{corollary}
\begin{proof} If $\hat S,\hat T$ are
$F_n$-equivariantly homeomorphic then for every
 measured lamination $\mu$ we have
$\langle T,\mu\rangle =0$ if and only if 
$\langle S,\mu\rangle =0$ \cite{CHL07,KL09}.  
Thus by Proposition \ref{fillisfull}, 
if $[T]\in {\cal I\cal T}$ and if there is a measured lamination
$\mu$ with $\mu(L^2(T)-L_r(T))=0$ then 
we have 
\[{\cal H}([T])=\{[S]\in {\cal I\cal T}\mid \langle S,\mu\rangle =0\}.\]
The corollary now follows from continuity of the length
pairing $\langle ,\rangle$ \cite{KL09}. 
\end{proof}

\begin{remark}\label{donotknow}
We do not know whether every indecomposable tree
$[T]\in {\cal I\cal T}$ admits a measured lamination 
$\mu$ with $\langle T,\mu\rangle =0$ and 
$\mu(L^2(T)-L_r(T))=0$.
\end{remark}

The remainder of this section is devoted to the proof of 
Proposition \ref{fillisfull}.

Recall that a minimal subset of 
an $F_n$-space is an invariant set with each orbit dense.
In the case that $T$ is an indecomposable tree with a 
free action of $F_n$, Proposition 5.14 of \cite{CH10} together
with Proposition 3.7 and Lemma 4.10 of \cite{CHR11} show
that the regular sublamination $L_r(T)$ is minimal.   
The main technical tool for the proof of Proposition \ref{fillisfull}
is the following extension of this 
result to indecomposable trees which are not necessarily free.
%In its formulation, 
%we assume that the action of $F_n$ on 
%an $\mathbb{R}$-tree $T$ is free, but
%we allow fixed points in the metric completion $\overline{T}$.
%Thus there may be elliptic elements for the action of 
%$F_n$ on $\overline{T}$. 

\begin{lemma}\label{minimal}
Let $[T]\in \partial {\rm CV}(F_n)$ be indecomposable;
then the regular sublamination $L_r(T)$ of $L^2(T)$ is minimal.
\end{lemma}
\begin{proof} Choose a basis ${\cal A}$ for $F_n$. Elements of $F_n$ are
reduced finite words in the alphabet ${\cal A}^{\pm}$. 
A point $X$ in the boundary $\partial F_n$ of $F_n$ is an infinite reduced word
in ${\cal A}^{\pm}$ whose first letter will be denoted by $X_1$.
The cylinder 
\[C_{\cal A}(1)=\{(X,Y)\in (\partial F_n)^2\mid X_1\not=Y_1\}\]
is compact, and the same holds true for the \emph{relative limit
set} 
\[\Omega_{\cal A}=Q^2(L^2(T)\cap C_{\cal A}(1))\subset \overline{T}.\]
The \emph{compact heart} $K_{\cal A}$ of $\overline{T}$ is the 
convex hull of $\Omega_{\cal A}$. It is a compact subtree of 
$\overline{T}$ (see Section 5.3 of \cite{CH10}).

The tree $T$ is called of \emph{Levitt type} if 
the limit set $\Omega$ is a totally disconnected subset of $\overline{T}$.
That this definition is equivalent to earlier definitions in the literature
is Theorem 5.11 of \cite{CH10}. 
We first show the lemma for indecomposable trees $T$ of Levitt type.
For such trees the relative limit set $\Omega_{\cal A}$ is a Cantor set
\cite{CH10}.

For any element $a\in {\cal A}^\pm $ consider the partial isometry
defined by the restriction of the action of $a^{-1}$.
 \[K_{\cal A}\cap aK_{\cal A}\to K_{\cal A}\cap a^{-1}K_{\cal A},
 \,x\to a^{-1}x.\]
 The thus defined 
system of partial isometries determines a directed labeled
graph $\Gamma$ whose
 vertex set is the set of components of $K_{\cal A}$ 
(the compact heart is connected, but we will use this construction
below also for compact subsets of $K_{\cal A}$ which may be disconnected)
and where 
an edge labeled with $a\in {\cal A}$ 
is a partial isometry connecting a component of $K_{\cal A}$ which 
intersects the domain of definition of $a$ 
to the component containing its image. As $K_{\cal A}$ is 
connected, the graph $\Gamma$ is a rose.
 
Apply the \emph{Rips machine} (see Section 3 of \cite{CH10}) to 
 the system of isometries $(K_{\cal A},{\cal A}^{\pm})$. 
 Its first step consists in choosing $K_{\cal A}(1)$ to be the subset of 
 $K_{\cal A}$ of all points which are in the domain of at least two 
 distinct partial isometries from ${\cal A}^{\pm}$ and restricting the 
 partial isometries from ${\cal A}^\pm$ to $K_{\cal A}(1)$.
This determines a new directed labeled graph $\Gamma_1$.
 There is a natural morphism $\tau_1$ from the directed 
 graph $\Gamma_1$ into 
 the directed
 graph $\Gamma$ (see Section 3 of \cite{CH10} for details).
 
 Since $T$ is of Levitt type, the Rips machine never stops  \cite{CH10}.
The output is a sequence $\Gamma_i$ $(i>0)$ 
of graphs 
without vertices of valence zero or one (Proposition 5.6 of \cite{CH10})
so that every leaf of 
$L^2(T)$ can be represented by a path in each of the $\Gamma_i$. 
For each $i$ there is a homotopy equivalence
$\tau_{i+1}:\Gamma_{i+1}\to \Gamma_i$ \cite{CH10}.

Define the \emph{size} of an element of $F_n$ to be its length with
respect to the generating set ${\cal A}^{\pm}$.
There are only finitely many $F_n$-orbits of points in $T$ with nontrivial
point stabilizers \cite{CH10}, and  
each point stabilizer for $T$ is finitely generated \cite{GL95}. 
This implies that 
there are numbers $i_0>0,p>0,u\geq 0$ with the following properties.
For each $i\geq i_0$ the graph $\Gamma_i$ contains 
precisely $u\geq 0$ labeled loops $\alpha_1^i,\dots,\alpha_u^i$ 
which define elements of $F_n$ 
of size at most $p$. Each of these loops corresponds to
a conjugacy class
in a point stabilizer of $T$. 
The point stabilizers of $T$ are
generated by elements in the conjugacy classes
of these loops, and the 
homotopy equivalence
$\tau_{i+1}:\Gamma_{i+1}\to \Gamma_i$ \cite{CH10} maps
the loop $\alpha_j^{i+1}$ to the loop $\alpha_j^i$.

By Theorem 5.3 of \cite{CH10}, there are only finitely
many $F_n$-orbits of 
points in $Q^2(L^2(T))$
whose preimage under the map $Q$ have
cardinality at least three. These points contain the points
with non-trivial stabilizer.
Thus  the union of 
all $F_n$-orbits of points in the limit set $\Omega=Q^2(L^2(T))$ 
whose preimage under $Q$ have cardinality at least three is a
countable subset of $\Omega$.

We follow the proof of Lemma 3.6 of \cite{CHR11} and we assume 
to the contrary that
there is a closed proper subset $L_0\subset L_r(T)$.
Since $\Omega_{\cal A}$ is uncountable and since $L_r(T)$ does not have
isolated points, by the previous paragraph 
there is a leaf $\ell\in L_r(T)-L_0$
such that $X=Q^2(\ell)\in \Omega_{\cal A}$ and that moreover
the preimage of $X$ under $Q$ consists precisely of the two endpoints of $\ell$.

The leaf
$\ell$ is not contained in the closure of $L_0$ and therefore 
there is a compact subarc $\rho=\ell[-k,k]$ of $\ell$ 
which is not shared by any leaf in $L_0$. We may assume that the size
$2k$ of the arc is at least $4p$ and that none of its prefixes coincides with
any of the words defining a loop $\alpha_j^i$ or its inverse. 
Note that this makes sense since the loops $\alpha_j^i$ define only
finitely many conjugacy classes in $F_n$.

For sufficiently large $i$, every (non-oriented) loop of size at most $p$ in 
the graph $\Gamma_i$ is one of the
loops  $\alpha_j^i$ (see the proof of Lemma 3.6 of \cite{CHR11}).
As a consequence, the initial segment of 
the subarc  $\rho$ of $\ell$ is distinct from any of these loops.

Choose now a 
sequence $\ell_j\subset L_r(T)$ of pairwise
distinct leaves which converge to $\ell$. 
We may assume that $\ell_j[-k,k]=\rho$ for all $j$.
Since $\Omega$
is a Cantor set, for every $u$ and for 
sufficiently large $m$ the leaves
$\ell_j$ $(0\leq j\leq u)$ 
are contained in distinct components of $K_{\cal A}(m)$.
On the other hand, 
there exists a number $i(k)>i_0$ such that
for $i>i(k)$ the size of any loop in $\Gamma_i$ which is 
distinct from one of the loops $\alpha_j^i$ 
is at least $2k$.
This implies that there are $n>0,m>0$ such that
$\ell_n[-k,k]$ does not cross through any vertex of $\Gamma_m$ of 
valence strictly bigger than two
(compare the proof of Lemma 3.6 of \cite{CHR11} for details).

As in the proof of Lemma 3.6 of \cite{CHR11} we conclude that 
there is 
an edge of $\Gamma_m$ which contains a subsegment of the
arc $\ell_n[-k,k]$ and 
which is
missed by any leaf of $L_0$. Hence every leaf of $L_0$ is contained 
in some free factor of $F_n$. 
Since $T$ is indecomposable,
by Theorem 4.5 of 
\cite{R11}, for every free factor $H$ of $F_n$ the minimal
$H$-invariant subtree of $T$ is discrete and therefore
a leaf of $L^2(T)$ contained in $H$ is contained in a point
stabilizer of $T$. Then $L_0\not\subset L_r(T)$ which is 
a contradiction. This shows 
the lemma in the case that the tree $T$ is of Levitt type.

By definition, a tree $T$ is of \emph{surface type} if the 
Rips machine stops. 
By Proposition 5.14 of \cite{CH10}, an indecomposable tree
$T$ either is of Levitt type or of  surface type.
Thus we
are left with showing the lemma for indecomposable trees of surface type.

Let $({\cal K},{\cal A})$ be a system of isometries 
which is not modified further. 
The set  ${\cal K}$ is a compact forest, in particular
it has finitely many connected components. For each
$a\in {\cal A}$ let ${\cal K}(a)\subset {\cal K}$ be the 
union of those components which intersect the
domain of definition for $a$.
We claim that there is some $a\in {\cal A}$ and some
component $K\in {\cal K}(a)$ such that
the domain of definition for $a$ intersects $K$ 
in a proper subtree.

Namely, 
otherwise each $a\in {\cal A}$ maps each component $K$ of 
${\cal K}(a)$ isometrically onto a component 
$a(K)$ of ${\cal K}$. Since ${\cal K}$ has only finitely many components,
this implies that there is a component $K$ of ${\cal K}$, and there
is a finite cyclically reduced nontrivial word $w=a_1\dots a_s$ in 
${\cal A}^{\pm}$ such that $w(K)=K$.  
Then $w$ is an isometry of the compact tree $K$.
Since $K$ is compact, $w$ has a fixed point $p\in K$ and 
permutes the components of $K-\{p\}$.
Using again compactness of $K$ we conclude that 
there is some $k\geq 1$ such that $w^k$ stabilizes a non-degenerate
segment in $K$. Since
$K\subset \overline{T}$ the tree 
$\overline{T}$ has a segment with non-trivial stabilizer.
This contradicts the assumption that 
$T$ is indecomposable.

As a consequence, there is some $a\in {\cal A}^{\pm}$, and there is a component
$K\in {\cal K}(a)$ containing  an extreme point 
$x$ of the domain of definition of $a$.
By definition, such an extreme point
$x$ is contained in the domain of definition of $a$. 
Moreover, this domain of definition 
contains a segment abutting at $x$, the complement 
$K-x$ of $x$ has 
at least two connected components, 
and there is at least one segment
abutting at $x$ whose interior is \emph{not} contained in the domain of $a$.

Following the proof of Lemma 4.10 of \cite{CHR11}, 
we can now \emph{split} the suspension surface by 
cutting $K_x$ at $x$ and cutting the suspension intervals containing $x$
accordingly. We obtain a new system of isometries. Repeat this construction.
Since by the above discussion the splitting process does not terminate,
we obtain 
the conclusion of the lemma as in the proof of Lemma 4.10 of \cite{CHR11}.
\end{proof}

Let again $[T]\in {\cal I\cal T}\subset
\partial {\rm CV}(F_n)$ be indecomposable and 
let $\mu$ be an ergodic measured lamination
with support ${\rm Supp}(\mu)\subset L^2(T)$
and $\mu(L^2(T)-L_r(T))=0$. We call such a measured lamination
\emph{regular}.  
By \cite{KL09}, a measured lamination $\mu$ is 
regular if and only if  
$\langle T,\mu\rangle =0$ and $\mu(L^2(T)-L_r(T))=0$.
%We do not claim that for every indecomposable tree there
%exists such a regular measured lamination. Rather we restrict our
%investigation to indecomposable trees with this property.
%For any point stabilizer $H$ of $\overline{T}$, a
%measured lamination supported in $H$ is supported in the
%zero lamination of $T$. 
%Thus 
%by Lemma \ref{notfactor} and the definition of an
%arational tree, 
%${\rm Supp}(\mu)$ is not contained
%in the stabilizer of a point in $\overline{T}$. 
%Then by ergodicity, 
%$\mu$ gives full measure to $L_r(T)$.

Define an
equivalence relation $\sim_\mu$ on 
$\partial F_n$ as the smallest 
equivalence relation with the following
property. 
The equivalence class of a point $\xi$ 
contains all points $\xi^\prime$
with $(\xi,\xi^\prime)\in {\rm Supp}(\mu)$.  
Let  $\sim$ be the closure of $\sim_\mu$.
By invariance of ${\rm Supp}(\mu)$ under the action of 
$F_n$, the equivalence relation $\sim$ is $F_n$-invariant,
and its quotient $\partial F_n/\sim$ is a 
compact $F_n$-space. 
Note that it
is unclear whether $\partial F_n/\sim$ is a tree.
The action of $F_n$ on $\partial F_n/\sim$ is minimal
and dense since this holds true for the action of 
$F_n$ on $\partial F_n$. 

By Corollary 2.6 of \cite{CHL07}, the 
tree $\hat T$ with the observer's topology
is $F_n$-equivariantly homeomorphic to 
$\partial F_n/L^2(T)$. Now the zero lamination
$L^2(T)$ contains the support 
 ${\rm Supp}(\mu)$ of $\mu$ 
and hence the compact $F_n$-space
$\partial F_n/\sim$ admits an $F_n$-equivariant
continuous surjection onto the tree $\hat T$. 

The following observation is also used in Section 11
and completes the proof of Proposition \ref{fillisfull}
under the assumption that $\partial F_n/\sim$ is a tree.
In its statement, we do not assume that the
tree $S$ is very small.

\begin{lemma}\label{closure}
Let $S$ 
be a minimal dense $F_n$-tree which admits an alignment
preserving map onto indecomposable trees
$T,T^\prime$. Then $\hat T$ is $F_n$-equivariantly 
homeomorphic to $\hat T^\prime$.
%
%If both $\partial F_n/\sim$ and 
%$[T]\in {\rm CV}(F_n)$  
%are indecomposable 
%then $G$ is an $F_n$-equivariant homeomorphism.
\end{lemma}
\begin{proof} Let $S$ be a minimal dense $F_n$-tree,
let $T\in \partial cv(F_n)$ be indecomposable and
assume that  there is 
an alignment preserving map 
\[f:S\to T.\]

A \emph{transverse family} for an $F_n$-tree
$U$ with dense orbits is an $F_n$-invariant 
family $\{Y_v\}$ of non-degenerate subtrees
$Y_v\subset U$ with the property
that if $Y_v\not=Y_{v^\prime}$ then 
$Y_v\cap Y_{v^\prime}$ contains at most one point.

Since $S$ is dense we can apply
Theorem 12.12 of \cite{R10}. It shows 
that there is an ergodic length measure 
$\nu$ for $S$ so that
with the notations from Section 5 we have
$T=S_\nu$.
Moreover, there is a transverse family
${\cal F}$ for $S$ with the following
properties.

Each component of ${\cal F}$ is a subtree $Y$ of $S$.
Its stabilizer is a vertex group of a very small
splitting of $F_n$, and it acts on $Y$ with
dense orbits (see the explicit construction in Section 10 of 
\cite{R10} for this fact which is attributed to 
Guirardel and Levitt).
The tree $T$ is obtained by collapsing each 
component of the transverse family ${\cal F}$ to 
a point. In particular, the family ${\cal F}$ is 
trivial (i.e. it is empty or its components are 
points) 
if and only if $\hat S$ is equivariantly
homeomorphic to $\hat T$.

Let now $T^\prime$ be a second indecomposable
tree for which there is an
alignment preserving map $f^\prime:S\to T^\prime$. 
Then 
the tree $T^\prime$ is obtained from 
$S$ by collapsing each tree from 
a second transverse family ${\cal F}^\prime$ to a point.
If $\hat T$ is not $F_n$-equivariantly 
homeomorphic to $\hat T^\prime$
then the transverse families
${\cal F},{\cal F}^\prime$ do not coincide.
In particular, up to exchanging $T$ and $T^\prime$ we
may assume that 
there is a component $Y^\prime$ of ${\cal F}^\prime$
which is not mapped to a point by the 
alignment preserving map 
$f:S\to T$.

As 
the stabilizer $H$ of $Y^\prime$ is 
a vertex group of a very small splitting of $F_n$, it is a finitely
generated subgroup of $F_n$ \cite{GL95} 
of infinite index, and it acts on $Y^\prime$ with 
dense orbits. 
By equivariance and the
assumption that $f$ is alignment preserving, 
the image under $f$ of the minimal
$H$-invariant subtree $Y^\prime$ of $S$ equals the minimal
$H$-invariant subtree $Z$ of $T$. Since $Y^\prime$ is not mapped
to a point, 
the tree $Z$ is non-trivial, and by equivariance, 
$H$ acts on $Z$ with dense orbits.
However, by the main result of \cite{R11}, since
$T$ is indecomposable and $H$ is a finitely
generated subgroup of $F_n$ of infinite index, the minimal
$H$-invariant subtree of $T$ is simplicial.
This is a contradiction which shows the lemma.
\end{proof}

{\textit{Proof of Proposition \ref{fillisfull}:}}
Let $T\in \partial cv(F_n)$ be indecomposable
and let $\mu\in 
{\cal M\cal L}$ be a regular
measured lamination for $T$, i.e. such that
$\langle T,\mu\rangle =0$ and $\mu(L^2(T)-L_r(T))=0$.

Recall that the zero  lamination $L^2(T)$ for $T$ 
defines a closed equivalence relation 
on $\partial F_n$, and  
$\hat T=\partial F_n/L^2(T)$ (Corollary 2.6 of \cite{CHL07}). 
Let $\sim$ be the closure of the equivalence
relation on $\partial F_n$ 
defined by ${\rm Supp}(\mu)\subset L^2(T)$.
There is a natural $F_n$-equivariant continuous surjection  
\[G:\partial F_n/\sim\to \hat T.\]

A leaf $\ell$ of $L^2(T)$ is \emph{diagonal}
over a sublamination $L_0$ 
if $\ell=(x_1,x_n)$ and if 
there are points $x_2,\dots,x_{n-1}\in \partial F_n$ so that
$(x_1,x_2),\dots,(x_{n-1},x_n)\in L_0$.
A leaf of $L^2(T)$ which is diagonal
over ${\rm Supp}(\mu)$ is contained in the equivalence
relation $\sim$. 
If the tree $T$ is free then 
$L^2(T)-{\rm Supp}(\mu)$ consists of finitely
many $F_n$-orbits of diagonal leaves 
\cite{CHR11} and therefore $G$ is an 
equivariant homeomorphism. This implies the
proposition in the case that the tree $T$ is free.

Now let $T$ be an arbitrary indecomposable tree,
let $\mu$ be a regular measured lamination for $T$ and 
let $S\in \partial cv(F_n)$ be another indecomposable
tree with $\langle S,\mu\rangle =0$. 
We claim that a subgroup $H$ of $F_n$ stabilizes
a point in $T$ if and only if $H$ stabilizes a point in $S$.

To this end recall that the stabilizer 
$H$ of a point $x\in T$ is finitely generated
\cite{GL95} and therefore 
the Gromov boundary $\partial H$ of $H$ 
is a closed $H$-invariant subset of $\partial F_n$.
By equivariance and continuity, 
the preimage in $\partial F_n/\sim$ 
of the fixed point $x$ of $H$ 
under the surjection $G$ 
is a compact
$H$-invariant subset $A$ of $\partial F_n/\sim$.
This set is just the projection of $\partial H\subset
\partial F_n$ to $\partial F_n/\sim$.
The union $\cup_{gH\in F_n/H}gA$ is disjoint and defines 
a transverse family 
${\cal F}$ of compact
spaces. Each component of ${\cal F}$ is invariant under a conjugate
of $H$ (these spaces may be reduced to points or may not be 
trees).

There also is an equivariant
projection $Q:\partial F_n/\sim \to \hat S$. If $H$ 
does not fix a point in $S$
then the transverse family 
${\cal F}$ of $\partial F_n/\sim$ is not collapsed in $\hat S$ 
to a countable
set of points. However, as in the proof of Lemma \ref{closure},
since the interior of $\hat S$ is $F_n$-equivariantly homeomorphic
to the tree $S$ \cite{CHL07},  
this implies that the minimal $H$-invariant subtree of 
$S$ is not reduced to a point, and $H$ acts on it with 
dense orbits. As in the proof of 
Lemma \ref{closure}, we conclude that this is impossible.

By the definition of the regular lamination $L_r(T)$, the 
zero lamination $L^2(T)$ is the union of $L_r(T)$, the
sets $\partial H\times \partial H-\Delta$ where $H$ runs through
the point stabilizers of $T$ and some isolated diagonal leaves.

As diagonal leaves are contained in the closure of the
equivalence relation defined by ${\rm Supp}(\mu)$ and 
the point stabilizers of $T$, the zero lamination
$L^2(T)$ is determined uniquely by ${\rm Supp}(\mu)$ and 
the point stabilizers of $T$. The above discussion
then shows that  
indeed $L^2(T)=L^2(S)$ as claimed. 
This completes the proof of Proposition \ref{fillisfull}.
\qed

\section{Trees with point stabilizers containing a free factor}

The goal of this section is to analyze trees in 
$\partial {\rm CV}(F_n)$ with dense orbits which
have some point stabilizers containing a free factor.

Let $A<F_n$ be a free factor of rank $r\geq 1$.
A \emph{free basis extension} of $A$ is a free basis 
$e_1,\dots,e_{n-r},e_{n-r+1},\dots,e_n$ of $F_n$ such that
$e_{n-r+1},\dots,e_n$ is a free basis of $A$. 
Define the \emph{standard simplex relative to $A$} 
of a free basis extension of $A$ to be the set 
\[\Delta(A)\subset \overline{cv_0(F_n)}^{++}\]
of all trees $T$ with volume
one quotient
with the following property. Either $T$ 
is contained in the boundary of 
the standard simplex $\Delta$ for the basis and admits
$A$ as a point stabilizer, 
or
$T/F_n$ can be obtained as follows.

Let  
$R_0,R_{n-r+1},\dots,R_n$ be roses 
of rank $n-r$ and $1$, respectively, with petals of length
$1/(n+r)$ which are
marked by the subgroup of $F_n$ generated by $e_1,\dots,e_{n-r}$ 
and by $e_j$ $(n-r\leq j\leq n)$. 
Connect the roses $R_j$ in an arbitrary way by $r$ edges
of length $1/(n+r)$ 
so that these edges form a forest in the resulting 
connected graph $G_1$.
For $j\geq n-r+1$ collapse each rose 
$R_j$ in $G_1$ to a point and let
$G_2$ be the resulting graph. 
The graph $T/F_n$ is obtained from $G_2$  
by changing the
lengths of the remaining edges, allowing some of the edges to shrink to 
zero length. 

If $A<F_n$ is a free factor which fixes a point in a tree
$[T]\in \partial{\rm CV}(F_n)$ then we say that the action of $A$ on $[T]$
is elliptic.

The next observation is a relative version of Lemma 
\ref{morphismsex}.

\begin{lemma}\label{relativemorphism}
Let $A<F_n$ be a free factor and let 
$[T]\in \partial{\rm CV}(F_n)$ be 
such that the action of $A$ on 
$[T]$ is elliptic. Then for every
standard simplex $\Delta(A)$ relative to $A$ 
there is a tree $U\in \Delta(A)$ and a train track
map $f:U\to T$ where $T$ is some representative of $[T]$.
\end{lemma}
\begin{proof}
Let $A<F_n$ be a free factor of rank $r\leq n-1$ and let
$[T]\in \partial {\rm CV}(F_n)$ be a 
tree containing a point $x$ which 
is stabilized by $A$. Let $e_1,\dots,e_n$ be a free basis 
extension of $A$
and let $\Delta\subset \overline{cv_0(F_n)}^{++}$ 
be the standard simplex for this basis.

By Lemma \ref{morphismsex} there is some 
$S\in \Delta$, and there is a train track map  $f:S\to T$ where 
$T$ is a representative of $[T]$. The map $f$ is determined by the
image of a point $\tilde v\in S$ in the preimage of the unique vertex $v$ of 
$S/F_n$. 

There are now two possibilities. In the first case, $f(\tilde v)$ is 
stabilized by a conjugate of $A$. 
Now up to conjugation, 
for $i\geq n-r+1$ a train track map $f:S\to T$ maps 
the edge of $S$ with endpoints $\tilde v,e_i\tilde v$ 
isometrically onto 
the segment in $T$ connecting $f(\tilde v)$ to 
$e_if(\tilde v)$. 
Since $f(\tilde v)=e_if(\tilde v)$ and since $f$ is an edge isometry, 
the tree
$S$ is contained in $\overline{cv_0(F_n)}-cv_0(F_n)$, and the vertex
$\tilde v$ is fixed by a conjugate of $A$. 
As a consequence, we have $S\in \Delta(A)$ and we are done.

If $f(\tilde v)$ is not stabilized by a conjugate of $A$ then 
by equivariance, the vertex $\tilde v$ of $S$ is not 
stabilized by a conjugate of $A$. Thus there is a 
petal of $S/F_n$ which represents an element $a$ of $A$.
This petal lifts to a line in $S$ through $\tilde v$ which is stabilized 
by a conjugate of $a$ (but not pointwise). 
Let us suppose that this conjugate is just $a$.
The edge in $S$ with endpoints  $\tilde v$ and  $a\tilde v$ is 
mapped by $f$ isometrically to a segment in $T$.

By our assumption on  $T$,  
there is a fixed point for $a$ on $T$.
Since $T$ is very small, the fixed point set 
${\rm Fix}(a)$ of $a$ is 
a point or a closed segment in $T$.
Let $\tilde x$ be the point of smallest distance
to $f(\tilde v)$ in ${\rm Fix}(a)$.

Let $s$ be the unique segment in $T$ connecting $f(\tilde v)$ to $\tilde x$. 
The segment meets ${\rm Fix}(a)$ only in $\tilde x$.
Then $a(s)$ is the segment in $T$ connecting $af(\tilde v)$ to 
$a(\tilde x)=\tilde x$, and the segment $s\cap a(s)$ is fixed by $a$.
Since $s$ meets the fixed point set of $a$ only in $\tilde x$,
we conclude that 
$s\cap a(s)=\tilde x$. 
Thus the geodesic segment in $T$ connecting $f(\tilde v)$ to
$af(\tilde v)$ passes through $\tilde x$, and the geodesic segment
connecting $f(\tilde v)$ to $a^{-1}f(\tilde v)$ 
passes through $\tilde x$ as well. 

As a consequence, 
for the train track structure on $S$ defined by $f$, the turn at $\tilde v$
containing the two directions of the axis of $a$ is illegal.
Folding this illegal turn collapses the petal in $S/F_n$ 
defining $a$ to a single
segment without changing the rest of $S/F_n$. 
After volume
renormalization, the resulting tree $S_1$  
is contained in $\overline{cv_0(F_n)}$.
Its quotient $S_1/F_n$
contains a segment which is a collapse of the petal 
of $S/F_n$ defining $a$. There 
is a train track map $f_1:S_1\to T_1$ where 
$T_1$ is a rescaling of $T$.

The graph $S_1/F_n$ has precisely two vertices, and one 
of these vertices is univalent. There are at most $r-1$ loops
in $S_1/F_n$ defining elements of $A$. 
Repeat this construction with a petal of 
$S_1/F_n$ defining an element of 
the free factor $A$.
After a total of $r$ such steps we obtain a simplicial tree 
$S_r\in \overline{cv_0(F_n)}$ and
a train track map $f_r:S_r\to T_r$ where $T_r$ is a rescaling of $T$. 
The graph 
$S_r/F_n$ consists of a rose with $n-r$ petals 
(some of them may be degenerate to points) with a collection
of edges attached.
By construction, we have $S_r\in \Delta(A)$ as claimed.
\end{proof}

%For $1\leq \ell\leq n-1$ 
Recall from Section 2 the definition of the
%graph ${\cal F\cal F}_\ell$ and  of the 
map $\Upsilon_{\cal F}:
\overline{cv_0(F_n)}^+\to {\cal F\cal F}$. 
Recall also from
Section 4 the definition of the map 
$\psi_{\cal F}:\partial {\rm CV}(F_n)\to 
\partial {\cal F\cal F}\cup \Theta$.
We use Lemma \ref{relativemorphism}
to show

\begin{corollary}\label{pointstab}
Let $A<F_n$ be a free factor of rank $\ell\geq 1$ and let 
$[T]\in \partial{\rm CV}(F_n)$ be a tree such that
the action of $A$ on $[T]$ is elliptic.
 Then  $\psi_{\cal F}([T])=\Theta$.
\end{corollary} 
\begin{proof} Let $\Delta(A)$ be a standard simplex relative to $A$.
By Lemma \ref{relativemorphism} there is a tree $S\in \Delta(A)$ and
a train track map $f:S\to T$ where $T$ is a representative of $[T]$.

Any point on a Skora path $(x_t)$ 
guided by $f$ 
contains a point stabilized by a fixed free factor contained in 
$A$. By Lemma \ref{equivariantfs2}, 
this implies that for every $t>0$ the
free factor $\Upsilon_{\cal F}(x_t)$ is contained in a uniformly 
bounded neighborhood of $A$.
 In particular, the diameter of 
 $\Upsilon_{\cal F}(x_t)$ is finite.
The claim of the lemma now follows from 
Proposition \ref{coarseindependence}. 
\end{proof}

\section{Trees which split as graphs of actions}

In this section 
we investigate the structure of
trees in $\partial {\rm CV}(F_n)$ with dense orbits which resemble
trees $T$ with $T_d\not=\emptyset$. 
The description of such trees is as follows 
\cite{G08,L94}.

\begin{definition}\label{graphofaction}
A \emph{graph of actions} is a minimal $F_n$-tree which consists of 
\begin{enumerate}
\item a simplicial tree $S$, called the \emph{skeleton},
equipped with an action of $F_n$
\item for each vertex $v$ of $S$ an $\mathbb{R}$-tree
$Y_v$, called a \emph{vertex tree}, and
\item for each oriented edge $e$ of $S$ with terminal vertex
$v$ a point $p_e\in Y_v$, called an \emph{attaching point}.
\end{enumerate}
\end{definition}

It is required that the projection $Y_v\to p_e$ is equivariant, 
that for $g\in F_n$ one has $gp_e=p_{ge}$ and that
each vertex tree is a minimal very small tree for its stabilizer.

Associated to a graph of actions ${\cal G}$ is a canonical action
of $F_n$ on an $\mathbb{R}$-tree $T_{\cal G}$ which is
called the \emph{dual} of the graph of actions
 \cite{L94}. Define a pseudo-metric $d$ on 
$\coprod_{v\in V(S)}Y_v$ as follows. If $x\in Y_{v_0},y\in Y_{v_k}$
let $e_1\dots e_k$ be the reduced
edge-path from $v_0$ to $v_k$ in $S$ and define
\[d(x,y)=d_{Y_{v_1}}(x,p_{\overline e_1})+\dots
+d_{Y_{v_k}}(p_{e_k},y).\]
Making this pseudo-metric Hausdorff gives an 
$\mathbb{R}$-tree $T_{\cal G}$. Informally, the tree
$T_{\cal G}$ is obtained by 
first inserting the vertex trees into the skeleton $S$ and then 
collapsing each edge of the
skeleton to a point.

We say that an $F_n$-tree 
$T$ \emph{splits as a graph of actions}
if either $T_d\not=\emptyset$ or if there
is a graph of actions ${\cal G}$ with dual tree $T_{\cal G}$, and 
there is an equivariant isometry $T\to T_{\cal G}$.
We also say that the projectivization $[T]$ of an $F_n$-tree $T$ 
splits as a graph of actions  if $T$ splits as a graph of actions.

A transverse family for an $F_n$-tree
$T$ with dense orbits is an $F_n$-invariant 
family $\{Y_v\}$ of non-degenerate subtrees
$Y_v\subset T$ with the property
that if $Y_v\not=Y_{v^\prime}$ then 
$Y_v\cap Y_{v^\prime}$ contains at most one point.
The transverse family is a \emph{transverse covering}
if any finite segment $I\subset T$ is contained in a finite
union $Y_{v_1}\cup \dots \cup Y_{v_r}$ 
of components from the family. The vertex trees 
of a graph of actions with dual tree 
$T_{\cal G}$ define a transverse covering of $T_{\cal G}$.
More precisely,  
by Lemma 4.7 of \cite{G04}, a minimal very small $F_n$-tree
$T$ admits a transverse covering
if and only if $T$ splits as a graph of actions.

A dense tree $T$ splits as a \emph{large graph of actions} if
it is equivariantly isometric to the dual tree of 
a graph of actions ${\cal G}$ with 
the following
additional properties.
\begin{enumerate}
\item There is a single $F_n$-orbit of vertex trees.
\item A stabilizer
of a vertex tree is not contained in a proper free factor of $F_n$.
\item 
A vertex tree is indecomposable for its stabilizer.
\item The skeleton of ${\cal G}$ 
does not have an edge with trivial stabilizer.
\end{enumerate}
Note that there is some redundancy in the above definition.

The tree $T$ splits as a \emph{very large graph of actions}
if it splits as a large graph of actions and if moreover
the skeleton does not
have an edge with infinite cyclic stabilizer.

The following  result is Corollary 11.2 of \cite{R10}.

\begin{proposition}\label{projection}
Let $T\in \overline{cv(F_n)}$ have dense orbits, and
assume that $T$ is neither indecomposable nor splits
as a graph of actions. Then there is an alignment 
preserving map $f:T\to T^\prime$ such that
%\begin{enumerate}
%\item 
either $T^\prime$ is indecomposable or $T^\prime$
splits as a graph of actions.
%\item the image under $f$ of the zero lamination 
%of $T$ is contained in the zero lamination of $T^\prime$.
%\end{enumerate}
\end{proposition}

We use Proposition \ref{projection} to show

\begin{proposition}\label{verylarge}
\begin{enumerate}
\item Let $[T]\in \partial {\rm CV}(F_n)$ be such that
$\psi([T])\in \partial {\cal F\cal S}$. Then $T$ admits
an alignment preserving map onto a tree which either
is indecomposable or splits as a large graph of actions.
\item 
Let $[T]\in \partial {\rm CV}(F_n)$ be such that
$\psi_{\cal C}([T])\in \partial {\cal C\cal S}$. Then $T$ admits
an alignment preserving map onto a tree which either
is indecomposable or splits as a very large graph of actions.
\end{enumerate}
\end{proposition}
\begin{proof} Let $[T]\in \partial{\rm CV}(F_n)$ be such that
$\psi([T])\in \partial {\cal F\cal S}$
(or $\psi_{\cal C}([T])\in \partial {\cal C\cal S}$). 
By Corollary \ref{dead}, $[T]$ is not simplicial. 
A tree which is neither dense nor simplicial
admits a one-Lipschitz alignment preserving map onto a dense tree. 
By Proposition \ref{onelipproj}, we may therefore assume 
without loss of generality that $T$ is dense.
By Proposition \ref{projection} and 
Corollary \ref{projectionalign}, 
it now suffices to show
the following. If $T$ 
splits as a graph of actions,
then $T$ admits an alignment preserving
map onto a tree which 
splits as a large (or very large) 
graph of actions.

Thus assume that $T$ is the dual tree of a graph of actions, with
skeleton $S$. Let $U$ be the tree obtained from $S$ by insertion of the
vertex trees. The tree $U$ may not be very small, but it admits
an alignment preserving map onto $T$ obtained by equivariantly 
collapsing each edge of $U$ to a point.

If there is an edge in $S$ with trivial 
(or either trivial or infinite cyclic) stabilizer
then there is an edge $e$ in $U$ with trivial 
(or infinite cyclic) stabilizer.
The tree $V$ obtained from $U$ by equivariantly collapsing
those edges of $U$ to points
which are not contained in the orbit of $e$
is very small, and $V_d\not=\emptyset$. 
There is an edge in $V$ with trivial (or either trivial or
infinite cyclic) stabilizer.
Moreover, $V$ admits
an alignment preserving map onto $T$.
We then have $\psi([T])=\Theta$ (or 
$\psi_{\cal C}([T])=\Theta$) by Corollary \ref{projectionalign}
and  Proposition \ref{graphofgroups}.

If there is no edge $e$ in the skeleton 
$S$ with trivial (or either trivial or infinite cyclic) stabilizer
then choose an arbitrary edge $e$ in $U$ and 
equivariantly 
collapse all edges of $U$ to a point which are 
not contained in the $F_n$-orbit of $e$.
This defines the structure of a graph of actions 
for $T$ whose skeleton projects
to a one-edge graph of groups decomposition with edge group
which is not trivial (or neither trivial nor infinite cyclic). 

If there is more than one $F_n$-orbit of vertex trees for 
this graph of actions
then we can collapse the vertex trees in all 
but one of these orbits to points. 
We obtain a tree $W$ which is dual to a graph of actions 
with a single orbit of vertex trees. 
The edge stabilizer of each edge of the skeleton
is not trivial
(or neither trivial nor infinite cyclic). 

Let $T_v$ be 
one of the vertex trees of the $F_n$-tree $W$. 
Its stabilizer is a finitely
generated subgroup $H$ of $F_n$ (and hence a finitely generated free group)
which acts on $T_v$ with dense orbits. 
By Proposition \ref{projection}, the $H$-tree $T_v$ either 
admits an alignment preserving map onto an indecomposable
tree or onto a tree which splits as a graph of actions.

Assume first that $T_v$ admits
an alignment preserving map onto an indecomposable
$H$-tree. Then $T$ admits an alignment preserving map onto 
a tree $Z$ which splits as a graph of actions, with a single orbit
of indecomposable vertex trees. Moreover, $Z$ is dual to a graph of 
actions with a single orbit of edges whose stabilizers are 
not trivial (or neither trivial 
not infinitely cyclic). In particular, $Z$ splits as a 
large (or as a 
very large) graph of 
actions provided that the group $H$ is not contained in a proper
free factor of $F_n$.

To see that the latter property indeed holds true
consider the one-edge
splitting of $F_n$ defined by 
the $F_n$-quotient of the skeleton of $Z$. 
Assume first that this splitting is a one-edge
two-vertex splitting. Then 
this splitting is  
of the form $H*_CB$ where $C$ is a 
nontrivial free subgroup of $H$
(or a free subgroup of $H$ of rank at least two). 

Let $A<F_n$ be the smallest free factor 
containing $H$. Assume to the contrary
that  $A$ is a proper subgroup of $F_n$. 
Then 
there is a free splitting $F_n=A*B_1$ where 
$B_1<B$, and there is a refinement 
$(H*_C*B_0)*B_1$ of 
$H*_CB$. Since 
there is a unique orbit of indecomposable vertex trees for $Z$
whose stabilizers are conjugate to $H$, 
the group $B$ fixes a vertex in $Z$, and
the tree $Z$ 
can be represented as a graph of actions with an edge
with trivial edge group.
Proposition \ref{graphofgroups} now shows that
$\psi([T])=\Theta$ which contradicts the assumption on $T$.
Thus we have $A=F_n$ and 
$Z$ splits as a large (or as a very
large) graph of actions. 

If the splitting of $F_n$ defined by an edge in
the skeleton of  
$Z$ is a one-edge one-vertex splitting and if the 
vertex group $H$ is contained in a proper free factor $A$ of $F_n$
then the vertex group is a free factor of rank $n-1$
(since $F_n$ is generated by $A$ and a single primitive element)
and the splitting is  a one-loop 
free splitting of $F_n$ which violates once more the
assumption that $\psi([T])\in \partial {\cal F\cal S}$ (or 
$\psi_{\cal C}([T])\in \partial {\cal C\cal S}$).
This completes the proof of the
proposition in the case that $T_v$ admits an
alignment preserving map onto an indecomposable tree.

We are left with the case that a vertex tree $T_v$ 
of the graph of actions with dual tree $W$ admits an
alignment preserving map onto a tree $Y_v$ which splits 
as a graph of actions for the stabilizer of $T_v$.
Then the tree $W$ admits
an alignment preserving map onto the tree $Y$ 
which is obtained by collapsing 
each tree in the orbit of $T_v$ to a tree 
in the orbit of $Y_v$.

Iterating the above argument, in the case that
$Y_v$ has more than one orbit of vertex trees then
all but one of these orbits can be collapsed to points.
Extending the collapsing map equivariantly 
yields an alignment preserving
map of $Y$ onto a tree $Y^\prime$ 
which splits as a graph of actions,
with a single orbit of vertex trees. The sums of the ranks
of conjugacy classes of point stabilizers in $Y^\prime$
is strictly bigger than those of $Y$.

Since the sum of the ranks
of conjugacy classes of point stabilizers in 
a very small $F_n$-tree is uniformly bounded
\cite{GL95}, we find 
in finitely many such steps an $F_n$-tree 
$Q$ which splits as a graph of actions with 
a single $F_n$-orbit of vertex trees, and each
such vertex tree is indecomposable for its stabilizer.
Moreover, there is an alignment preserving map $T\to Q$.
The skeleton for the structure of a graph of actions for $Q$
is an $F_n$-tree with finite quotient. Stabilizers of 
edges are not trivial  
(or neither trivial nor cyclic). In other words, 
$Q$ splits as a large (or as a very large) graph of actions.
This completes then proof of the proposition.
\end{proof}

The following example was shown to me by Camille Horbez.

\bigskip

\noindent{\bf Example:} Let $T$ be an indecomposable $F_n$-tree with 
a point stabilizer which contains a free factor $H$ of $F_n$ 
of rank $k\geq 3$.
Let $B$ be a free group of rank $\ell<k$ and choose a finite
index subgroup $H^\prime$ 
of $B$ which is a free group of rank $k$. Identifying $H$ and $H^\prime$
defines a one-edge two vertex graph of groups decomposition
$F_n*_HB$ for a free group $F_m$ where $m=n+\ell-k<n$. 
Let the group $B$ act trivially on a point and define a graph of 
actions for $F_m$ with a single orbit of edges in its skeleton
containing $T$ as a vertex tree. 
This graph of actions is very large.

We can also iterate this construction 
by beginning with an indecomposable
$F_n$-tree $T$ with more than one $F_n$-orbit of points 
with point stabilizer containing a free factor of rank
$k\geq 3$ and constructing from $T$ a 
very large graph of actions defined by a tree whose
skeleton has more than one connected component.

\bigskip

The final goal of this section is to show

\begin{proposition}\label{graphofaction2} 
%\begin{enumerate}
If $[T]\in \partial {\rm CV}(F_n)$ splits as a graph of actions then
$\psi_{\cal F}([T])=\Theta$.
%\item If $[T]\in \partial {\rm CV}(F_n)$ splits as a graph of actions
%which is not small 
\end{proposition}
\begin{proof} By Corollary \ref{hierarchy}, Corollary \ref{projectionalign}
and Proposition \ref{verylarge}, it suffices to show
that $\psi_{\cal F}([T])=\Theta$ for every 
projective tree $[T]$ which splits as a very large
graph of actions.

By Corollary \ref{pointstab}, this indeed holds true if 
a tree which splits as a very large graph of actions
admits a point stabilizer which contains a free factor
of $F_n$. That this is indeed the case is due to Reynolds
(Proposition 10.3 of \cite{R12}).
\end{proof}

Proposition \ref{graphofgroups}, Corollary \ref{graphofaction2},
Proposition \ref{projection}
and Corollary \ref{projectionalign} immediately imply

\begin{corollary}\label{indecomposable3}
Let $[T]\in \partial{\rm CV}(F_n)$ be such that
$\psi_{\cal F}([T])\in \partial {\cal F\cal F}$.
Then $T$ admits an alignment preserving map onto
an indecomposable tree.
\end{corollary}

\section{The boundary of the free factor graph}

In this section we complete the proof of Theorem \ref{positive}
from the introduction.

An indecomposable tree $T$ is called \emph{arational} if 
no point stabilizer for $T$ contains a free factor. 
We begin with collecting some additional  
information on arational trees. Part of what we need is covered
by the main result of \cite{R12} which shows that an arational
tree $T$ either is free or dual to a measured lamination of
an oriented surface with a single boundary component.
We will not use this information.

%We also describe for every $\ell\leq n-2$ 
%indecomposable trees which 
%do not correspond to points in the boundary of ${\cal F\cal F}_\ell$.

A closed $F_n$-invariant
subset $C$ of $\partial F_n\times \partial F_n-\Delta$
\emph{intersects a free factor} if there is a proper
free factor $H$ of $F_n$ so that 
$C\cap \partial H\times \partial H-\Delta\not=\emptyset$.
%The group $H$ is a \emph{minimal} 
%proper free factor intersecting $C$ if 
%there is no proper free factor $H^\prime<H$ 
%which is intersected by $C$.
%We say that $C$ is \emph{contained} in $H$
%if $C=F_n (C\cap \partial H\times \partial H-\Delta)$.

The following lemma is a simple consequence
of Theorem 4.5 of \cite{R11}.

\begin{lemma}\label{indecfree}
Let $T\in \partial cv(F_n)$ be indecomposable.
If the zero lamination $L^2(T)$ of $T$ intersects a proper free
factor of $F_n$ then there is a point stabilizer
for the action of $F_n$ on  $T$ which contains a 
proper free factor.
\end{lemma}
\begin{proof} The intersection of $L^2(T)$ with 
a proper free factor $H<F_n$ is contained in the zero lamination
of the minimal $H$-invariant subtree $T_H$ of $T$.
By Theorem 4.5 of \cite{R11}, since $T$ is indecomposable 
the action on $T$ of
any proper free factor $H$ 
is discrete and hence $T_H$ is simplicial.
Thus the intersection of $L^2(T)$ with $H$ is non-empty if and only if 
there is a leaf of $L^2(T)$ which is  
carried by a point stabilizer for the action of $H$ on
$T_H$. 

Since $T$ is indecomposable, the stabilizer of any non-degenerate
segment of $T$ is trivial. As a consequence, 
the tree $T_H$ is very small simplicial, with trivial 
edge stabilizers. Therefore
$T_H/H$ defines a graph of groups decomposition of $H$ with 
trivial edge groups, i.e. it defines a free splitting of $H$. 
Each vertex group is
a free factor of $H$. Hence if the action of $H$ on $T_H$ is not
free, then there is a point stabilizer of $T_H$
which is a proper free factor of $H$
and hence of $F_n$.
The lemma follows. 
\end{proof}

We use Lemma \ref{indecfree} to show that only arational 
trees can give rise to points in the boundary of 
the free factor graph ${\cal F\cal F}$.

Say that a measured lamination
$\mu$ is \emph{supported} 
in a finitely generated
subgroup $H$ of $F_n$ 
if ${\rm Supp}(\mu)=
F_n({\rm Supp}(\mu)\cap \partial H\times \partial H-\Delta)$.
Our next goal is to understand 
measured laminations whose supports
are contained in the point stabilizer
of an indecomposable tree.
For this we use the following 
consequence of Theorem 49 of \cite{Ma95}.
%Recall also from the introduction that  
%the rank ${\rm rk}(\mu)$ 
%of a measured lamination $\mu$ is defined
%to be the smallest rank of a free factor containing the support of $\mu$.

\begin{lemma}\label{notfactor}
Let $H<F_n$ be a finitely generated subgroup 
of infinite index
%Assume that
%every proper free factor of $F_n$ intersects
%$H$ in a subgroup of infinite index. Then there is no 
%Then the rank of every 
%measured lamination for $F_n$ supported in  $H$ is at most 
%$n-1$.
which
does not intersect a free factor. Then 
$\partial H\times \partial H-\Delta$ does not
support a measured lamination for $F_n$.
\end{lemma}
\begin{proof} 
%Let $\nu$ be an ergodic measured lamination
%supported in $H$ which is not supported in any free factor.
Assume to the contrary that there is a measured
lamination $\nu$ for $F_n$ supported in 
$\partial H\times \partial H-\Delta$.
% with ${\rm rk}(\nu)=n$.
Let $\ell\in \partial H\times \partial H-\Delta$
be a density point for $\nu$. Then there is no proper
free factor $A$ of 
$F_n$ so that 
$\ell\in \partial A\times \partial A-\Delta$.

Choose a free basis ${\cal A}$  for $F_n$ and represent
$\ell$ by a biinfinite word 
$(w_i)$ in that basis. Since $H<F_n$ is finitely
generated of infinite
index,  $H$ is quasiconvex and hence we may assume that
there is a sequence
$n_i\to \infty$ such that each of the  
prefixes $(w_{n_i})$ of the word $(w_i)$ 
represents an
element of $H$. 
%By adjusting the sequence 
%we also may assume that none of these
%words is contained in a proper free factor of $F_n$.

By Theorem 49 of \cite{Ma95} (which is attributed to 
Bestvina), for each $i$ there is a free basis ${\cal B}_i$ 
for $F_n$ such that
the Whitehead graph of $w_{n_i}$ with respect to ${\cal B}_i$ 
is connected and does not have a cut vertex. Since
$\ell$ is a density point for $\nu$, the Whitehead graph of 
$\nu$ for the basis ${\cal B}_i$ contains the 
Whitehead graph of $w_{n_i}$, in particular it does not
have a cut vertex and is connected. On the other hand,
since $\nu$ is a measured lamination, Proposition 21 of \cite{Ma95} shows
that the Whitehead graph of $\nu$ with respect to 
${\cal B}_i$ has a cut vertex or is disconnected. This is 
a contradiction and shows the lemma.
\end{proof}

In Section \ref{indecom} we defined a measured lamination
$\mu\in {\cal M\cal L}$ to be regular for an
indecomposable tree $T$ if $\langle T,\mu\rangle=0$ and
$\mu(L^2(T)-L_r(T))=0$.
By Lemma \ref{indecfree} and 
Lemma \ref{notfactor}, if $T$ is arational then 
every measured lamination $\mu$ with
$\langle T,\mu\rangle=0$ is regular for $T$.

%\begin{corollary}\label{rank}
%Let $T$ be indecomposable.
%An ergodic measured lamination $\mu$ with 
%$\langle T,\mu\rangle =0$ is regular if and only if it
%is of rank $n$.
%\end{corollary}
%\begin{proof} By the main result of \cite{R11}, a leaf
%$\ell\in L_r(T)$ is not contained in any nontrivial free factor,
%so if $\mu$ is regular then the rank of $\mu$ equals $n$.

%On the other hand, an isolated leaf in the support of a measured
%lamination is a pair of fixed points for an element of $F_n$.
%Thus if $\mu$ is not regular then by ergodicity,
%$\mu$ is supported
%in a point stabilizer of $T$. Such a point stabilizer is a
%finitely generated subgroup $H$ of $F_n$ \cite{GL95}
%of infinite index. By Lemma \ref{notfactor}, the 
%rank of $\mu$ does not exceed $n-1$.
%\end{proof}

\bigskip

\noindent{\bf Example:} Let $n=2g\geq 4$ and let $S$ be 
an oriented surface
of genus $g$ with non-empty connected boundary.
The fundamental group of $S$ is the group $F_n$. Every 
measured lamination $\mu$ on $S$ is dual to an $F_n$-tree
$T$. If the support of $\mu$ is minimal and decomposes $S$ into
ideal polygons, then this
tree is arational. The free homotopy class of the boundary
circle defines the unique conjugacy class in $F_n$ whose 
elements have fixed points in $T$.
The point stabilizers of $T$ do not support a measured lamination.

\bigskip

%For the remainder of this section we
%consider the map $\psi_{\cal F}:\partial {\rm CV}(F_n)\to 
%\partial {\cal F\cal F}\cup \Theta$.
Denote as before by ${\cal I\cal T}\subset \partial 
{\rm CV}(F_n)$ the set of indecomposable 
projective trees.
Let 
\[{\cal F\cal T}\subset {\cal I\cal T}\subset \partial{\rm CV}(F_n)\]
be the set of arational trees. 
Corollary \ref{indecomposable3}, Corollary \ref{pointstab} 
and Lemma \ref{indecfree} show that
$\psi({\cal F\cal T})\supset\partial {\cal F\cal F}$.
Our next task is to show that 
${\cal F\cal T}\subset \psi^{-1}_{\cal F}(\partial {\cal F\cal F})$. 
To this end recall that 
each conjugacy class of a primitive element $g\in F_n$ 
determines a measured lamination which 
is the set of all Dirac masses on the pairs of 
fixed points of the elements in the class.
The measured lamination is called 
dual to the conjugacy class. 

A primitive conjugacy class $\alpha$ in $F_n$ is 
\emph{short} for a tree $T\in cv_0(F_n)$ 
if it can be represented
by a loop on $T/F_n$ of length at most $4$.
Recall that the length pairing 
$\langle , \rangle$ on $\overline{cv(F_n)}\times
{\cal M\cal L}$ is continuous.

We have

\begin{lemma}\label{shortloop}
Let $[T_i]\subset {\rm CV}(F_n)$ be a sequence
converging to $[T]\in \partial {\rm CV}(F_n)$.
For each $i$ let $T_i\in cv_0(F_n)$ be a representative
of $[T_i]$ and let 
$\alpha_i$ be a primitive short conjugacy
class on $T_i$ with dual measured lamination $\mu_i$. 
If $[T]$ is dense then up to passing to a subsequence, 
there is a sequence $b_i\subset (0,1]$ such that
the measured laminations $b_i\mu_i$ converge weakly
to a measured 
lamination $\mu$ with $\langle T,\mu\rangle =0$.
Moreover, either up to scaling
$\mu$ is dual to a primitive conjugacy
class, or $b_i\to 0$.
\end{lemma}
\begin{proof} For each $i$ let $T_i\in cv_0(F_n)$ be a
representative of $[T_i]$, let 
$T$ be a representative of 
$[T]$ and let $a_i\in (0,\infty)$ be such that
$a_iT_i\to T$. 
Since the $F_n$-orbits on $T$ are dense,
we have $a_i\to 0$ $(i\to \infty)$.

Fix some tree $S\in cv_0(F_n)$. 
Then the set 
\[\Sigma=\{\zeta\in {\cal M\cal L}\mid  
\langle S,\zeta\rangle =1\}\] defines a continuous section of 
the projection ${\cal M\cal L}\to {\cal P\cal M\cal L}$ for the
weak$^*$-topology.
In particular, the space $\Sigma$ is compact.

There is a number
$\epsilon >0$ so that $\langle S,\zeta\rangle \geq \epsilon$
whenever $\zeta$ is dual to any primitive conjugacy class.
Let $\mu_i$ be the measured lamination dual to a primitive
short conjugacy class $\alpha_i$ on $T_i$.
If $b_i>0$ is such that
$b_i\mu_i\in \Sigma$ then the sequence $(b_i)$ is  
\emph{bounded}. Since $\Sigma$ is compact,
by passing to a subsequence
we may assume that 
$b_i\mu_i\to \mu$ for some
measured lamination $\mu\in \Sigma$.

Now $\langle a_iT_i,\mu_i\rangle \leq ka_i$ 
where $k\geq 2$ is as in the definition of a 
short conjugacy class (see Section 2) and hence
since $a_i\to 0$ $(i\to \infty)$ and since
the sequence $(b_i)$ is bounded, we have
\[\langle a_iT_i,b_i\mu_i\rangle \to 0\,(i\to \infty).\]
The first part of the lemma now follows from
continuity of the length pairing.
Moreover, either $b_i\to 0$ or the length on $S/F_n$ of the
conjugacy classes $\alpha_i$ is uniformly bounded. 
However, there are only finitely many conjugacy classes
of primitive elements which can be represented by
a loop on $S/F_n$ of uniformly bounded length. Thus
either $b_i\to 0$, or the sequence $(\alpha_i)$ contains only finitely
many elements and hence there is some primitive conjugacy
class $\alpha$ so that $\alpha_i=\alpha$ 
for infinitely many $i$. Then clearly $\mu$ is a multiple of the
dual of $\alpha$.
\end{proof}

The following proposition is the main remaining step towards
the proof of Theorem \ref{positive}. For its formulation,
define two trees $[S],[T]\in {\cal I\cal T}$ to be \emph{equivalent} if
the trees $\hat S,\hat T$ are $F_n$-equivariantly homeomorphic
with respect to the observer's topology.

\begin{proposition}\label{fillsone}
\begin{enumerate}
\item Let $[T]\in {\cal I\cal T}$; then 
$\psi_{\cal F}([T])\in \partial {\cal F\cal F}$
if and only if $[T]\in {\cal F\cal T}$.
\item If $[T],[T^\prime]\in {\cal F\cal T}$ then
$\psi_{\cal F}([T])=\psi_{\cal F}([T^\prime])$
if and only if $[T],[T^\prime]$ are equivalent.
\end{enumerate}
\end{proposition}
\begin{proof} Let $\Upsilon_{\cal F}=\Omega\circ\Upsilon:cv_0(F_n)\to 
{\cal F\cal F}$ be the map constructed in Section 2.
By the discussion preceding Lemma \ref{shortloop},
it suffices to show that $\psi([T])\not=\Theta$ for every
$[T]\in {\cal F|cal T}$. Since Skora paths map to
uniform unparametrized quasi-geodesics in ${\cal F\cal F}$,
this holds true if for every $[T]\in {\cal F\cal T}$ the image under
$\Upsilon_{\cal F}$ of a Skora path converging to $[T]$
is unbounded. 

We show more generally the following.  
If $[T_i]\subset {\rm CV}(F_n)$ is any
sequence which converges to $[T]\in {\cal F\cal T}$ 
and if $T_i\in cv_0(F_n)$ is a representative of $[T_i]$ 
then the
sequence $\Upsilon_{\cal F}(T_i)\subset {\cal F\cal F}$ is unbounded.
For this we use a variant of an argument of 
Luo as explained in \cite{MM99}. 

We argue by contradiction and we 
assume that after passing to a subsequence, the
sequence $\Upsilon_{\cal F}(T_i)$ remains
in a bounded subset in ${\cal F\cal F}$.

Since by Lemma \ref{equivariantfs2} the
map $\Upsilon_{\cal F}$ is a quasi-isometry for the
metric $d_{ng}=d_{ng}^1$ on $cv_0(F_n)$ 
which only assumes integral values, 
after passing to another subsequence 
we may assume that for all $i\geq 1$ 
the distance between $T_i$ and $T_0$ 
in $(cv_0(F_n),d_{ng})$ 
equals $m$ for some $m\geq 0$ which does not
depend on $i$. 

By the definition of the metric $d_{ng}$,
this implies that for all $i\geq 1$ there is a sequence
$(T_{j,i})_{0\leq j\leq m}\subset cv_0(F_n)$ 
with $T_{0,i}=T_0$ and $T_{m,i}=T_i$
so that for all $j<m$ the trees
$T_{j,i}$ and $T_{(j+1),i}$ are one-tied.
In particular, for each $j< m$
there is a primitive conjugacy class 
$\alpha_{j,i}$ which can be represented
by a curve of  length at 
most $4$ on both $T_{j,i}/F_n$ and
$T_{(j+1),i}/F_n$. 
Let $\mu_{j,i}$ be the measured lamination which is dual to
$\alpha_{j,i}$. 

By assumption, we have 
$[T_{m,i}]\to [T]$ $(i\to \infty)$ 
in $\overline{{\rm CV}(F_n)}$. Since $T$ is dense, 
Lemma \ref{shortloop} implies that up to passing to a subsequence, 
there is a bounded sequence $(b_i)$ such that the 
measured laminations 
$b_i\mu_{m-1,i}$ converge
as $i\to \infty$ to a measured lamination
$\nu_{m-1}$ supported in the zero lamination of 
$T$. Since $[T]\in {\cal F\cal T}$, by Lemma \ref{indecfree} 
the support
of $\nu_{m-1}$ does not intersect a free factor and hence  
$b_i\to 0$ by Lemma \ref{shortloop}.

By passing
to another subsequence, we may assume that the projective
trees $[T_{(m-1),i}]$ converge as $i\to \infty$ to a projective
tree $[U_{m-1}]$. 
Choose a representative $U_{m-1}$ of $[U_{m-1}]$.
Since $b_i\mu_{m-1,i}\to \nu_{m-1}$ for a sequence
$b_i\to 0$ and since $\langle T_{m-1,i},\mu_{m-1,i}\rangle \leq 4$ for 
all $i$. Lemma \ref{shortloop} shows that
$\langle U_{m-1},\nu_{m-1}\rangle=0$. In particular,
$\nu_{m-1}$ is supported in the zero lamination of 
$U_{m-1}$. 
Proposition \ref{fillisfull} 
now shows that $U_{m-1}$ admits
an alignment preserving map onto $T$.
Moreover, there is a subsequence
of the sequence $[\mu_{m-2,i}]$ which converges as
$i\to \infty$ to a projective measured lamination supported in the zero
lamination of $[U_{m-1}]$ and hence of $[T]$.
(Since $[T]$ is arational, it can easily be seen that 
$[U_{m-1}]\in {\cal F\cal T}$, a fact which is immediate from \cite{R12}).

Repeat this argument with the sequence 
$[T_{(m-2),i}]$ and the tree $[U_{m-1}]$. After 
$m$ steps we conclude that 
$T_0$ admits an alignment
preserving map onto $T$ which is impossible. The
first part of the proposition is proven.

The second part of the proposition is shown in the same way.
By Corollary \ref{projectionalign}, we only have to show that
if  $[T],[T^\prime]\in {\cal F\cal T}$ are
such that $\psi_{\cal F}([T])=\psi_{\cal F}([T^\prime])$ then
$[T],[T^\prime]$ are equivalent. 
Let $(x_t),(y_t)$ be Skora paths connecting 
a point in a standard simplex $\Delta$ to $[T],[T^\prime]$. Since 
the paths $\Upsilon_{\cal F}(x_t),\Upsilon_{\cal F}(y_t)$ are uniform 
reparametrized quasi-geodesics in ${\cal F\cal F}$,
by hyperbolicity there is a number $m\geq 0$ and 
for each $t>0$ there is a number $s(t)>0$ so that
$d_{ng}(\Upsilon_{\cal F}(x_t),\Upsilon_{\cal F}(y_{s(t)}))\leq m$.

As above, by passing to a subsequence (and perhaps changing the 
constant $m$) we may assume that
there is a sequence $t_i\to \infty$ so that 
$d_{ng}(\Upsilon_{\cal F}(x_{t_i}),\Upsilon_{\cal F}(y_{s(t_i)}))=m$ for all $i$.
For each $i$ let $u_i\in cv_0(F_n)$ be such that
$d_{ng}(x_{t_i},u_i)=1$ and 
$d_{ng}(u_i,y_{s(t_i)})=m-1$. 
By the above discussion, up to passing to a subsequence
the sequence $[u_i]$ converges to a point
$[U]\in \partial {\rm CV}(F_n)$ so that $U$ admits an alignment preserving
map onto $T$.
Repeat with the sequence $(u_i)$. After $m$ steps we conclude that
$[T^\prime]$ is equivalent to $[T]$ which is what we wanted to show.
\end{proof}

\begin{remark}\label{measlamex}
The proof of Proposition \ref{fillsone} shows that for every $[T]\in {\cal F\cal T}$
there is a measured lamination $\mu$ with $\langle T,\mu\rangle =0$. 
Thus Corollary \ref{sometimesclosed} applies to arational trees.
\end{remark}

We showed so far that the map 
$Y=\psi_{\cal F}\vert {\cal F\cal T}$ is a continuous closed surjection of
${\cal F\cal T}$ onto $\partial {\cal F\cal F}$.
Each fibre consists
of the closed set of 
all trees in a fixed equivalence class for the relation $\sim$.
Thus the Gromov boundary
of ${\cal F\cal F}$ is 
homemorphic to ${\cal F\cal T}/\sim$. Theorem \ref{positive}
is proven.

We complete this section with an easy consequence which will be
useful in other context. For its formulation, following
\cite{H09} we call a pair $(\mu,\nu)\in {\cal M\cal L}\times 
{\cal M\cal L}$ \emph{positive} if 
for every tree $T\in \overline{cv(F_n)}$ we have
$\langle T,\mu\rangle +\langle T,\nu\rangle >0$.

\begin{corollary}\label{positivepair}
Let $\mu,\nu\in {\cal M\cal L}$ be measured laminations
which are supported in the zero lamination of trees
$[T],[S]\in {\cal F\cal T}$.
If $Y([T])\not= Y([S])$ then $(\mu,\nu)$ is a positive pair. 
\end{corollary}

\begin{remark}
The topology for the boundary of the free factor graph can also be
described as a measure forgetful topology in the following sense.
The projection to $\partial {\cal F\cal F}$
of a sequence $[T_i]\subset {\cal F\cal T}$ converges to the projection of $[T]$ 
if and only if the following holds true. For each $i$
let $\beta_i$ be a measured lamination with $\langle T_i,\beta_i\rangle =0$.
Assume that the measured laminations converge to a measured lamination
$\beta$; then $\langle T,\beta\rangle =0$. 
\end{remark}

\section{The boundaries of the free and the cyclic splitting graph}
\label{theboundaries}

In this section we identify the Gromov 
boundary of the free and of the cyclic splitting graph.

We continue to use all 
assumptions and notations from the previous sections.
In particular, we use the maps 
\[\psi:\partial {\rm CV}(F_n)\to \partial {\cal F\cal S}\cup \Theta
\text{ and } 
\psi_{\cal C}:
\partial {\rm CV}(F_n)\to \partial {\cal C\cal S}\cup \Theta\]
defined in Section \ref{boundaries}.

As in the introduction, let ${\cal S\cal T}\subset \partial {\rm CV}(F_n)$ be the 
${\rm Out}(F_n)$-invariant set of projective dense trees which 
either are indecomposable or split as large graphs of actions. 
It contains the ${\rm Out}(F_n)$-invariant subspace 
${\cal C\cal T}$ of 
projective trees which either are indecomposable or split as 
very large graph of actions.
By Corollary \ref{surjective}, Corollary \ref{projectionalign} and 
Proposition \ref{verylarge}, 
we have  
\[\psi({\cal S\cal T})\supset \partial {\cal F\cal S}\text{ and }
\psi_{\cal C}({\cal C\cal T})\supset
\partial {\cal C\cal S}.\]

We have to show that $\psi^{-1}(\partial {\cal F\cal S})={\cal S\cal T}$ and
$\psi^{-1}(\partial {\cal C\cal S})={\cal C\cal T}$.
We begin with an extension of Lemma \ref{closure}.
To this end we call as before two dense trees $T,T^\prime\in 
\partial cv(F_n)$ equivalent if 
the unions $\hat T,\hat T^\prime$ of their metric completions with their 
Gromov boundaries are equivariantly homemorphic.

\begin{lemma}\label{closure2}
Let $[T]\in {\cal S\cal T}$, let $T^\prime\in \partial cv(F_n)$ 
be dense and 
assume that there is a tree $S\in \partial cv(F_n)$ 
which admits an alignment preserving map onto both $T,T^\prime$.
Then $T^\prime$ admits an alignment preserving map
onto a tree which 
is equivalent to $T$.
\end{lemma} 
\begin{proof} By Proposition \ref{projection} we may assume 
that $T^\prime$ is either indecomposable or splits as a graph of 
actions. 

We argue as in the proof of  Lemma \ref{closure}. 
Namely, if $T$ splits as a large graph of actions then 
 let $V\subset T$ be a vertex tree of the transverse
covering of $T$ defined by the structure of a 
large graph of actions. 
If $T$ is indecomposable
then we let $V=T$. Let $H<F_n$ be the stabilizer of $V$.
Then $H$ is finitely generated, and $V$ is indecomposable
for the action of $H$.
Let $S_H$ be the minimal
$H$-invariant subtree of $S$. By equivariance, 
there is a one-Lipschitz alignment preserving map 
$S_H\to V$, and there also is an alignment preserving
map $S_H\to T^\prime_H$ where $T^\prime_H$ is the 
minimal $H$-invariant subtree of $T^\prime$.

Using again Proposition \ref{projection} we may assume
that $T_H^\prime$ either is indecomposable or splits
as a graph of actions. 
Now $S_H$ is a minimal very small $H$-tree and 
therefore by the results in \cite{R11}, there is no
alignment preserving map from $S_H$ onto both an 
indecomposable tree and a tree which splits as a graph of 
actions. 
Since $V$ is indecomposable, 
this shows that $T_H^\prime$ is indecomposable.

An application of 
Lemma \ref{closure} to $S_H$ and $T_H,T_H^\prime$ now shows that
$\hat T_H^\prime$ is 
equivariantly homeomorphic to $V$. 
Since there is a single $F_n$-orbit of vertex trees
for the structure of a large graph of actions 
on $T$ we conclude that $T^\prime$ admits an
alignment preserving map onto a tree which is equivalent to 
$T$. This is what we wanted to show.
\end{proof}

The next proposition is the main remaining step for the proof of 
Theorem \ref{freesplitbd} and 
Theorem \ref{arational}. It entirely relies on the work of
Handel and Mosher \cite{HM13} and its variation 
due to Bestvina and Feighn \cite{BF12}.
Recall that two trees $T,T^\prime$ with dense orbits are
equivalent if the union of their metric completions with their
Gromov boundaries are equivariantly homeomorphic.

\begin{proposition}\label{relativeunbounded}
Let $\Delta\subset \overline{cv_0(F_n)}$ be a standard simplex, let 
$[T]\in {\cal S\cal T}$
and let 
$(x_t)$ be a Skora path connecting
a point $x_0\in \Delta$ to $[T]$. 
\begin{enumerate}
\item   
$\Upsilon(x_t)\subset {\cal F\cal S}$ is unbounded.
\item If $[T]\in {\cal C\cal T}$ then 
$\Upsilon_{\cal C}(x_t)\subset {\cal C\cal S}$ is unbounded.
\end{enumerate}
\end{proposition}
\begin{proof} Let $(x_t)\subset
\overline{cv(F_n)}$ be an unnormalized Skora path
connecting a point $x_0\in \Delta$ to a representative $T$
of $[T]$.

%We begin with showing that $\Upsilon(x_t)$ is unbounded
%in the first barycentric subdivision of the free
%splitting graph ${\cal F\cal S}$. To this end 
For the proof of the first part of the proposition 
we argue
by contradiction and we assume that 
$\Upsilon(x_t)\subset {\cal F\cal S}$ is bounded.
Let $d$ be the distance in ${\cal F\cal S}$.
Since the distance between two vertices in ${\cal F\cal S}$ 
assumes only integral values,
there is some $k\geq 0$ 
and a sequence $t_i\to \infty$ such that 
\[d(\Upsilon(x_0),\Upsilon(x_{t_i}))=k\]
for all $i$.

For each $i$ let $y_i(k)$ be the tree obtained from $x_{t_i}$ by
equivariantly rescaling edges in such a way that 
the lengths of all edges of the quotient $y_i(k)/F_n$ coincide.
This can be arranged in such a way that there is
a one-Lipschitz  
alignment preserving map $y_i(k)\to x_{t_i}$ and that 
the volume of $y_i(k)/F_n$ is bounded from above by a uniform
multiple of the volume of $x_{t_i}/F_n$. Thus 
after passing to a subsequence, 
we may assume that 
the trees $y_i(k)$ converge to a
tree $T_k\in \overline{cv(F_n)}$. 
By Proposition 5.5 of \cite{G00}, there is an alignment
preserving map $T_k\to T$.

For each $i$ let $\Gamma_i(j)$ $(0\leq j\leq k)$ be a 
geodesic in ${\cal F\cal S}$ connecting $\Upsilon(x_0)$ to 
$\Upsilon(x_{t_i})$. 
Then for all $i$, the vertex $\Gamma_i(k-1)$ can 
be obtained from $x_{t_i}$ by either an expansion or
a collapse. By passing to a subsequence, assume first that
for all $i$, 
$\Gamma_i(k-1)$ is obtained from $\Gamma_i(k)$ by a
collapse. 

Let $y_i(k-1)\in \overline{cv(F_n)}$ be a simplicial tree whose quotient
$y_i(k-1)/F_n$ has edges of equal length and
defines the splitting $\Gamma_i(k-1)$. As $\Gamma_i(k-1)$ is obtained
from $\Gamma_i(k)$ by a collapse, we may assume that
the volume of $y_i(k-1)$ is bounded from below by a 
uniform multiple of the volume of $y_i(k)$ 
and that
there is a one-Lipschitz alignment preserving map 
$y_i(k)\to y_i(k-1)$. 
By passing to a subsequence, we may moreover assume that
the trees $y_i(k-1)$ converge to a tree 
$T_{k-1}\in \overline{cv(F_n)}$. 
Since $T_k$ is dense, by equivariance
the same holds true for $T_{k-1}$.
By Proposition 5.5 of \cite{G00}, 
there is an alignment preserving map $T_k\to T_{k-1}$.
Lemma \ref{closure2} then shows that 
$T_{k-1}$ admits an alignment preserving map onto 
a tree $T_{k-1}^\prime$ which is equivalent to $T$.

In the case that $\Gamma_i(k-1)$ is obtained from
$\Gamma_i(k)$ by an expansion for all but finitely many
$i$ we construct the trees $y_i(k-1)$ as before, and we
find by passing to a subsequence that
$y_i(k-1)$ converges to a tree $T_{k-1}$ which 
admits an alignment preserving map onto $T_k$ and hence
onto a tree equivalent to $T$.

Repeat this construction with the sequence $y_i(k-1)$ and
trees $y_i(k-2)$ which define the splitting
$\Gamma_i(k-2)$. 
In finitely many steps we conclude that 
$x_0$ admits an alignment preserving 
map onto a tree which is equivalent to $T$.
This is a contradiction which shows the first part of the
proposition.

The argument for the second part of the proposition is 
identical and will be omitted.
\end{proof}

As an immediate consequence of Lemma \ref{sequence}, 
Proposition \ref{relativeunbounded} and Lemma \ref{cont} we obtain

\begin{corollary}\label{contid2}
The map $\psi$ (or $\psi_{\cal C}$) restricts to 
a continuous closed 
${\rm Out}(F_n)$-equivariant 
surjection ${\cal S\cal T}\to \partial {\cal F\cal S}$ 
(or a surjection ${\cal C\cal T}\to \partial {\cal C\cal S}$).
\end{corollary}

We are left with calculating the fibres of the map.
This is done in the next lemma.

\begin{lemma}\label{chain}
\begin{enumerate}
\item $\psi([T])=\psi([S])$ for trees 
$[S],[T]\in {\cal S\cal T}$ if and only if 
$[T]$ and $[S]$ are equivalent.
\item $\psi_{\cal C}([T])=\psi_{\cal C}([S])$ for trees 
$[S],[T]\in {\cal C\cal T}$ if and only if 
$[T]$ and $[S]$ are equivalent.
\end{enumerate}
\end{lemma} 
\begin{proof} If $\hat S,\hat T$ are $F_n$-equivariantly 
homeomorphic then there exists an alignment preserving
map $T\to S$ \cite{CHL07} and therefore
$\psi([S])=\psi([T])$ by Corollary \ref{projectionalign}.

Now let $[S],[T]\in {\cal S\cal T}$ be such that
$\psi([S])=\psi([T])$. 
Let $\Delta$ be a standard simplex and 
let $(x_t),(y_t)$ be Skora paths
connecting a point in $\Delta$ to 
$[S],[T]$.
By Corollary \ref{contid2}, the paths
$\Upsilon(x_t),\Upsilon(y_t)$ are uniform unparametrized
quasi-geodesics in ${\cal F\cal S}$ of infinite diameter 
with the same endpoint. The distance between their
starting points is uniformly bounded.

By hyperbolicity, there is a number $m>0$ such that 
for every $t\geq 0$ the free splitting
$\Upsilon(x_t)$ is contained in the $m$-neighborhood of 
the quasi-geodesic ray
$(\Upsilon(y_s))_{s\geq 0}$ in ${\cal F\cal S}$. As a consequence, for each 
$t>0$ there is some $s(t)>0$ such that the distance
between $\Upsilon(x_t)$ and $\Upsilon(y_{s(t)})$ is at most $m$.

For each $i$ and each $k=0,\dots, m$ let $\beta_i(k)$ be a
a graph of groups decomposition for $F_n$
with trivial edge group
with $\beta_i(0)=x_i/F_n, \beta_i(m)=y_{s(i)}/F_n$ such that 
for each $j\leq m/2$, the splitting $\beta_i(2j)$ collapses
to both $\beta_i(2j-1)$ and $\beta_i(2j+1)$.  
For $1\leq k\leq m-1$ the splitting $\beta_i(k)$ defines 
a simplicial tree $\hat \beta_i(k)\in 
\overline{cv_0(F_n)}$ which is unique if we require that
all edges have the same length. 

Choose a sequence $(i_\ell)_{\ell\geq 0}$ so that for each $k\leq m$
the projectivizations $[\hat \beta_{i_\ell}(k)]$ of the 
trees $\hat \beta_{i_\ell}(k)$ converge in $\overline{{\rm CV}(F_n)}$
to a tree $[U_k]$. Apply
Proposition 5.5 of \cite{G00} and conclude that for each $j$ 
there is an alignment preserving map 
$U_{2j}\to U_{2j-1}$ and $U_{2j}\to U_{2j+1}$. 

As in the proof of 
Proposition \ref{relativeunbounded}, 
we can now successively apply Lemma \ref{closure2}
and deduce that  
$\hat S$ and $\hat T$ are homeomorphic
which is what we wanted to show.
\end{proof}

\bigskip

\noindent
{\bf Example:} We give an example which 
shows that the Gromov boundary of the free splitting
graph does not coincide with the Gromov boundary of 
the cyclic splitting graph.

Namely, let $S$ be a compact surface of genus 
$g\geq 2$ with connected boundary $\partial S$.
The \emph{arc graph} of $S$ is defined as follows.
Vertices are embedded arcs in $S$ with both endpoints on
$\partial S$. Two such arcs are connected by an edge
of length one if they are disjoint. The arc graph of 
$S$ is quasi-isometrically embedded in the 
free splitting graph of the free group $\pi_1(S)=F_{2g}$
\cite{HH11}. The Gromov boundary of the arc graph has
been determined in \cite{H11}.
 
Let $c\subset S$ be a non-separating simple closed curve
and let $\phi$ be a pseudo-Anosov mapping class of 
$S-c$. It follows from the results in \cite{MS13}, \cite{H11} 
and \cite{HH11} that $\psi$ acts as a hyperbolic isometry
on the free splitting graph of $F_{2g}$. However,
$\psi$ preserves the cyclic splitting
$A*_{\langle c\rangle}$ and therefore it acts as an elliptic 
element on the cyclic splitting graph.

We conclude this section with some remarks and open questions.

\bigskip

\noindent
{\bf Question 1:} 
An element $\phi\in {\rm Out}(F_n)$ 
acts on the free factor graph as a hyperbolic isometry if and only if
it is irreducible with irreducible powers. 
An irreducible element with irreducible powers $\phi$ 
acts on $\partial {\rm CV}(F_n)$ with north-south dynamics \cite{LL03}, 
fixing precisely two points. 
It was shown in \cite{BFH97} (see also \cite{KL11}) that the stabilizer in 
${\rm Out}(F_n)$ of such a fixed point of $\phi$ is virtually cyclic. 

A fixed point $[T]$ of $\phi$
defines a
point in the boundary of the free factor graph and hence 
 is an arational tree.
In analogy of the
mapping class group action on the Thurston boundary of Teichm\"uller space,
it can be shown 
that $[T]$ is uniquely ergodic, i.e. there is a
unique projective measured lamination $\mu$ 
with $\langle T,\mu\rangle =0$ \cite{HH14}.

More generally, 
in \cite{H09} we defined an $F_n$-invariant set 
${\cal U\cal T}$ of projective trees in $\partial {\rm CV(F_n)}$
as follows. If $[T]\in {\cal U\cal T}$ and if 
$\mu\in {\cal M\cal L}$ is such that 
$\langle T,\mu\rangle =0$ then
the projective class of $\mu$ is unique, and 
$[T]$ is the unique projective tree with 
$\langle T,\mu\rangle =0$. It follows immediately
from this work that ${\cal U\cal T}\subset {\cal F\cal T}$
(see also \cite{R12,HH14}).

Define the $\epsilon$-thick part ${\rm Thick}_\epsilon(F_n)$ 
of $cv_0(F_n)$ to consist
of simplicial trees with quotient of volume one which do not
admit any essential loop of length smaller than $\epsilon$.
In analogy to properties of the curve graph and Teichm\"uller space,
we conjecture that whenever $(r_t)$ is a normalized Skora path in  
$cv_0(F_n)$ with the property that $r_{t_i}\in {\rm Thick}_\epsilon(F_n)$ 
for a sequence $t_i\to \infty$ and some 
fixed number $\epsilon >0$ then 
$[r_t]$ converges as $t\to \infty$ to a tree
$[T]\in {\cal U\cal T}$ (see \cite{Ho11} and \cite{HH14} 
for results in this
direction). Moreover, if $(r_t)\subset {\rm Thick}_\epsilon(F_n)$
for all $t$ then the path $t\to \Upsilon(r_t)$ is a parametrized
$L$-quasi-geodesic in ${\cal F\cal S}$ for a number $L>1$ only
depending on $\epsilon$.

\bigskip

\noindent
{\bf Question 2:} Let $\phi\in {\rm Out}(F_n)$ be a 
reducible outer automorphism which 
acts as a hyperbolic isometry on the free splitting graph ${\cal F\cal S}$. 
It follows from Theorem \ref{freesplitbd} that $\phi$ fixes a tree
$[T]\in {\cal S\cal T}$. 
It is true that 
up to scale there is a unique measured lamination
$\mu$ with $\langle T,\mu\rangle =0$ and 
$\mu(L^2(T)-L_r(T))=0$? What can we say about the subgroup of 
${\rm Out}(F_n)$ which fixes $[T]$? Recent work of Handel and Mosher
\cite{HM13d} indicates that this group may be quite large (i.e. it may not
be infinitely cyclic).

Vice versa, given a map $\phi\in {\rm Out}(F_n)$, it is known that 
$\phi$ can be represented by a relative train track map. 
It was observed by Handel and Mosher that $\phi$ defines a hyperbolic
automorphism of the free splitting complex if the 
support of the top stratum of the
relative train track is all of $F_n$. 
Theorem \ref{freesplitbd} shows that the
converse is also true: 
If the action of $\phi$ on ${\cal F\cal S}$ is hyperbolic,
then the support of the top stratum of a relative train track for $\phi$ is
all of $F_n$.

Note that Handel and Mosher showed
\cite{HM13d} that there are non-torsion elements $\phi\in {\rm Out}(F_n)$ 
which act with bounded orbits on the free splitting graph, but for which there
is no $k\geq 1$ so that $\phi^k$ fixes a point in ${\cal F\cal S}$.

\appendix
\section{Splitting control}

In this appendix we collect some
information from the work of Handel and Mosher \cite{HM13}.
%which is used for the proof of Theorem \ref{arational}.

Consider a free group $F_n$ of rank $n\geq 3$.
Denote by ${\cal F\cal S}$ and ${\cal C\cal S}$ the 
first barycentric subdivision of the free splitting graph
and the cyclic splitting graph
of $F_n$, respectively.
Vertices of ${\cal F\cal S}$ 
are graphs of groups decompositions of $F_n$
with trivial 
edge groups. Such a graph of groups
decomposition is 
a finite graph $G$ whose vertices are labeled with
(possibly trivial) free factors of $F_n$. 
Two vertices $G,G^\prime$ in ${\cal F\cal S}$ 
are connected
by an edge if $G^\prime$ is either a collapse or an
expansion of $G$.
Vertices of ${\cal C\cal S}$ are graph of groups
decompositions of $F_n$ with at most cyclic edge groups.
Two such vertices $G,G^\prime$ are connected by an edge
if $G^\prime$ is either a collapse or an expansion of $G$.

Throughout this appendix we follow the appendix
of \cite{BF12}.

A simplicial $F_n$-tree 
$T\in \overline{cv(F_n)}$ with at least one 
trivial edge stabilizer
projects to a finite metric graph $T/F_n$ which defines a 
graph of groups decomposition for $F_n$ with 
trivial edge groups, i.e. it defines 
a vertex in ${\cal F\cal S}$.
Namely, 
an arbitrary simplicial tree $T\in \overline{cv(F_n)}$
defines a vertex in ${\cal C\cal S}$.

A finite \emph{folding path} is a path 
$(x_t)\subset\overline{cv(F_n)}$ $(t\in [0,L])$ 
which
is guided by an optimal morphism $f:x_0\to x_L$
as introduced in Section 3.
The points along the path are obtained from $x_0$
by identifying directions which are mapped to the
same direction in $x_L$. 
We allow folding at any speed, and we also allow
rest intervals. Such paths are called \emph{liberal} 
folding paths in \cite{BF12}.

An  
optimal morphism projects to a map of the quotient
graphs which 
minimizes the Lipschitz constant in 
its homotopy class. 
There are induced optimal maps 
\[f_t:x_t\to x_L\] 
for all $t\in [0,L]$.

Let $G_t\subset \overline{cv(F_n)}$ 
$(0\leq t\leq L)$ be any folding path
and let $F_L\subset G_L$ be a proper equivariant
forest. Then for each $t$ 
the preimage $F_t\subset G_t$ of $F_L$ 
under the map $f_t:G_t\to G_L$ 
is defined, and 
$G_t^\prime=G_t/F_t$ is an $F_n$-tree. 
By Proposition 4.4 of \cite{HM13} (see also 
Lemma A.1 of \cite{BF12}), the path 
$t\to G_t^\prime$ is a liberal folding path, and we
obtain a commutative diagram
\begin{equation}\begin{CD}\label{dia1}
\overset{G_0^\prime}{\bullet} @ >>>  \overset{G_L^\prime}{\bullet} \\
@ AAA            @   AAA\\
\underset{G_0}{\bullet}      @ >>>  \underset{G_L}{\bullet}
\end{CD}
\end{equation}

Similarly, if $G_t$ $(0\leq t\leq L)$ is a folding path
and if $G_L^\prime\to G_L$ is a collapse map then there is 
a folding path $G_t^\prime$ so that the diagram
\begin{equation}\label{dia2}\begin{CD}
\overset{G_0^\prime}{\bullet} @ >>>  \overset{G_L^\prime}{\bullet} \\
@ VVV            @   VVV\\
\underset{G_0}{\bullet}      @ >>>  \underset{G_L}{\bullet}
\end{CD}
\end{equation} 
%diagram
commutes (Lemma A.2 of \cite{BF12}). 
Here one may have to insert rest intervals
into the folding path $G_t$.

The following lemma is a version of Lemma A.3 of \cite{BF12}.

\begin{lemma}\label{estimate}
There is a number $M>0$ with the following property.
Let $G,G^\prime\in \overline{cv(F_n)}$ 
be such that
$G/F_n,G^\prime/F_n$ are finite metric 
graphs defining a graph of groups
decomposition for $F_n$.
% with $A$ contained in a vertex group.
%and with trivial edge groups.
Let $f:G\to G^\prime$ be an optimal morphism.
Assume that there is a point $y\in G^\prime$
%and an interior point $y\in e^\prime$ 
such that
the cardinality of $f^{-1}(y)$ is at most $p$.
Then the distance in ${\cal C\cal S}$
between $G,G^\prime$ is at most $Mp$.
If the edge groups of $G/F_n,G^\prime/F_n$ are trivial then
the same holds true for the distance in 
${\cal F\cal S}$. 
\end{lemma}
\begin{proof} The proof of Lemma 4.1 of \cite{BF12} is valid without
modification. A variation of the argument is used in Section 11, so we
provide a sketch.

Connect $G$ to $G^\prime$ by a folding path
$G_t$ guided by $f$, with $G_0=G$.  
For each $t$ there are optimal
morphisms $\phi_t:G\to G_t$ and 
$f_t:G_t\to G^\prime$ so that $f=f_t\circ \phi_t$.
The number of preimages of $y$ under $f_t$ 
decreases with $t$.

Let $t>0$ be such that this number coincides
with the number of preimages of $y$ in $G$.  
Then there is  point $z_t\in G_t/F_n$ whose preimage
under the quotient of 
$\phi_t$ consists of a single point $z$. 
If $z_t$ is a vertex then the preimage 
of a nearby interior point of an adjacent edge
consists of a single point as well,  
so may assume that the points $z,z_t$ are contained in
the interior of edges $e,e_t$. A loop in $G/F_n$ not passing 
through the projection
of $z$ is mapped by the quotient
of $\phi_t$ to a loop in $G_t/F_n$ 
not passing through $z_t$.
Thus collapsing the complement of the
edges $e,e_t$ in $G/F_n,G_t/F_n$ 
to a single point yields two identical graph of groups
decompositions containing $A$ as a vertex group.
In particular, 
$G/F_n$ and $G_t/F_n$ collapse to 
the same vertex in ${\cal C\cal S}(A)$, and  
if the edge group of $e$ is trivial then 
$G(F_n$ and $G_t/F_n$ collapse to the same poin in 
${\cal F\cal S}$.

Let $t_0>0$ be the first point so that the number of preimages
of $y$ in $G_{t_0}$ 
is strictly smaller than the number of preimages of $y$
in $G_0$.
For $t<t_0$ sufficiently close to $t_0$, the graphs
$G_t/F_n,G_{t_0}/F_n$ collapse to the same graph of groups decomposition
of $F_n$.
Hence the lemma now follows by induction.
\end{proof}

\bigskip

\noindent
MATHEMATISCHES INSTITUT DER UNIVERSIT\"AT BONN\\
ENDENICHER ALLEE 60\\
53115 BONN\\
GERMANY

\smallskip\noindent
e-mail: ursula@math.uni-bonn.de

\end{document}